\documentclass[a4paper, english]{amsart}
\usepackage{hyperref}
\hypersetup{
colorlinks,
allcolors=blue
}

\usepackage[all]{xy}
\usepackage{subfig}
\usepackage{mathrsfs}
\usepackage{enumerate}

\usepackage[normalem]{ulem}

\usepackage[T1]{fontenc}
\usepackage[latin9]{inputenc}
\usepackage{xcolor}
\usepackage{verbatim}
\usepackage{units}
\usepackage{amstext}
\usepackage{amsthm}
\usepackage{amssymb}

\makeatletter
\numberwithin{equation}{section}
\numberwithin{figure}{section}
  \theoremstyle{remark}
  \newtheorem*{rem*}{\protect\remarkname}
  \theoremstyle{plain}
  \newtheorem*{lem*}{\protect\lemmaname}
  \theoremstyle{conjecture}
  \newtheorem*{conjecture*}{\protect\conjecturename}
\theoremstyle{plain}
\newtheorem{thm}{\protect\theoremname}[section]
  \theoremstyle{plain}
  \newtheorem{lem}[thm]{\protect\lemmaname}
  \theoremstyle{definition}
  \newtheorem{defn}[thm]{\protect\definitionname}
  \theoremstyle{plain}
  \newtheorem{prop}[thm]{\protect\propositionname}
  \theoremstyle{plain}
  \newtheorem{cor}[thm]{\protect\corollaryname}
   \theoremstyle{definition}
  \newtheorem{remark}[thm]{\protect\remarkname}
\newtheorem*{ack}{Acknowledgements}

\makeatother

\usepackage{babel}
  \providecommand{\corollaryname}{Corollary}
  \providecommand{\definitionname}{Definition}
  \providecommand{\lemmaname}{Lemma}
  \providecommand{\propositionname}{Proposition}
  \providecommand{\remarkname}{Remark}
  \providecommand{\conjecturename}{Conjecture}
\providecommand{\theoremname}{Theorem}

\newcommand{\sfd}{\mathsf{d}}
\newcommand{\m}{\mathfrak{m}}
\newcommand{\mm}{\mathfrak{m}}
\newcommand{\R}{\mathbb{R}}
\newcommand{\RR}{\mathbb{R}}
\newcommand{\N}{\mathbb{N}}
\newcommand{\RCD}{\mathsf{RCD}}
\newcommand{\CD}{\mathsf{CD}}
\newcommand{\MCP}{\mathsf{MCP}}
\newcommand{\CDE}{\mathsf{CDE}}
\newcommand{\Geo}{{\rm Geo}}
\newcommand{\Opt}{\mathrm{OptGeo}}
\newcommand{\ee}{{\rm e}}
\newcommand{\Ent}{{\rm Ent}}
\newcommand{\weakto} {\rightharpoonup}
\newcommand{\orb}{\mathcal{O}}
\newcommand{\vphi}{\varphi}
\newcommand{\GTB}{(\mathsf{GTB})}
\def\In{\subseteq}

\global\long\def\G{\mathsf{G}}
\global\long\def\H{\mathsf{H}}
\global\long\def\Lip{\operatorname{Lip}}
\global\long\def\LIP{\operatorname{LIP}}
\global\long\def\quotient{\mathsf{p}}
\global\long\def\lip{\operatorname{lip}}
\global\long\def\Ch{\operatorname{Ch}}
\global\long\def\OptGeo{\operatorname{OptGeo}}

\global\long\def\ord{\operatorname{ord}}
\global\long\def\vol{\operatorname{vol}}
\def\mc{\mathcal}

\global\long\def\supp{\operatorname{supp}}

\global\long\def\Ric{\operatorname{Ric}}

\global\long\def\vol{\operatorname{vol}}

\begin{document}

% TITLE

\title{On quotients of spaces with  Ricci curvature bounded below}

% AUTHORS

% AUTHOR 1

\author[F.~Galaz-Garc\'ia]{Fernando Galaz-Garc\'ia$^*$}
\address[F.~Galaz-Garc\'ia]{Institut f\"ur Algebra und Geometrie, Karlsruher Institut f\"ur Technologie (KIT), Karlsruhe, Germany}
\email{galazgarcia@kit.edu}
\thanks{$^{*}$ Supported in part by the Deutsche Forschungsgemeinschaft grant GA 2050 2-1 within the  Priority Program SPP 2026 ``Geometry at Infinity''.}

% AUTHOR 2
\author[M.~Kell]{Martin Kell}
\address[M.~Kell]{Fachbereich Mathematik, Universit\"at T\"ubingen, Germany}
\email{martin.kell@math.uni-tuebingen.de}

% AUTHOR 3
\author[A.~Mondino]{Andrea Mondino} 
\address[A.~Mondino]{Mathematics Institute, The University of Warwick, United Kingdom}
\email{A.Mondino@warwick.ac.uk}

% AUTHOR 4
\author[G.~Sosa]{Gerardo Sosa}
\address[G.~Sosa]{Max Planck Institute for Mathematics in the Sciences, Leipzig, Germany}
\email{gsosa@mis.mpg.de}

\setcounter{tocdepth}{1}

% KEYWORDS

\keywords{Ricci curvature, group actions, optimal transport, quotients}

\bibliographystyle{plain}

% ABSTRACT

\begin{abstract}
 Let $(M,g)$ be a smooth Riemannian manifold and $\G$ a compact Lie group acting on $M$ effectively and by isometries. It is well known that a lower bound of the sectional curvature of $(M,g)$ is again a bound for the curvature of the quotient space, which is an Alexandrov space of curvature bounded below. Moreover, the analogous stability property holds for metric foliations and submersions. 
\\The goal of the paper is to prove the corresponding stability properties for synthetic Ricci curvature lower bounds. Specifically, we show that such stability holds for quotients of $\RCD^{*}(K,N)$-spaces, under isomorphic compact group actions and more generally under metric-measure foliations and submetries. An $\RCD^{*}(K,N)$-space is a metric measure space with an upper dimension bound $N$ and weighted Ricci curvature bounded below by $K$ in a generalized sense. In particular, this shows that if $(M,g)$ has Ricci curvature bounded below by $K\in \R$ and dimension $N$, then the quotient space is an $\RCD^{*}(K,N)$-space. Additionally, we tackle the same problem for the $\CD/\CD^*$ and $\MCP$  curvature-dimension conditions. 

We provide as well geometric applications which include: A generalization of Kobayashi's Classification Theorem of homogenous manifolds to $\RCD^{*}(K,N)$-spaces with \emph{essential minimal dimension} $n\leq N$; a structure theorem for $\RCD^{*}(K,N)$-spaces admitting actions by \emph{large (compact) groups}; and geometric rigidity results for orbifolds such as Cheng's Maximal Diameter and Maximal Volume Rigidity Theorems. 

Finally, in two appendices  we apply the methods of the paper to study quotients by isometric group actions  of discrete spaces  and of (super-)Ricci flows.
\end{abstract}

\maketitle

% TABLE OF CONTENTS

\tableofcontents

% SECTION: INTRODUCTION
%----------------------------------------------------------------------------------

\section{Introduction}
Studying the geometry of isometry groups has proven to be advantageous for the understanding  of Riemannian manifolds. For instance, this point of view has been particularly successful in the construction of new examples under the assumption that the sectional curvature of a Riemannian manifold is non-negative or positive  (see, for example, the surveys by K.\ Grove \cite{Grove} and W.\ Ziller \cite{Ziller} and references therein). A main motivation to consider these types of bounds is that they are preserved under quotients of isometric actions. More precisely, the possibly non-smooth orbit space satisfies the same sectional curvature lower bounds of the original space in a synthetic fashion: in the sense of comparison triangles \`a la Alexandrov.   
A natural question that arises is weather a similar stability statement holds, and if so in which sense, when the assumptions are weakened by considering instead Ricci curvature lower bounds. The goal of the present paper is  to answer this question and to provide applications.

In order to get a feeling of the problem and of the state-of-the-art research, we discuss the simpler case when the quotient space is a smooth Riemannian manifold. A typical example is  when a Lie group $\G$ acts isometrically, freely and properly on a Riemannian manifold $(M,g)$. Then the projection %$\quotient: M\to M^{*}$ 
onto the quotient space %$M^{*}:=M/\G$ 
equipped with the quotient metric is a Riemannian submersion, and the problem amounts to understanding how curvature bounds behave under such maps. In this context, O'Neill's formula \cite{ONeill} shows that sectional curvature lower bounds are preserved. However, C.\ Pro and F.\ Wilhelm recently showed \cite{ProWil} that this is not the case for the \emph{standard Ricci curvature tensor}. To be more precise, Pro and Wilhelm presented  examples of Riemannian submersions from compact manifolds with positive Ricci curvature to manifolds that have small neighborhoods of arbitrarily negative Ricci curvature. On the other hand,   J.\ Lott \cite{Lott} showed in this framework that Ricci curvature lower bounds in a \emph{weighted sense} are stable. Specifically, he considered the \emph{Bakry-\'Emery Ricci tensor} rather than the standard Ricci curvature tensor under the condition that the  fiber transport preserves the measure up to constants. Inspired by this result it is then natural to ask whether Ricci curvature lower bounds are preserved, in a synthetic sense, in a more general case. Namely, without assuming that the original and quotient spaces are smooth manifolds.

\emph{Synthetic Ricci curvature lower bounds}  for  \emph{metric measure spaces} have been introduced in the seminal papers of Lott-Villani \cite{LottVillani} and Sturm \cite{sturm:I,sturm:II}. As usual, a metric measure space $(M,\sfd,\mm)$ is a complete separable metric space $(M,\sfd)$ endowed with a  non-negative, locally finite Borel measure $\mm$. The rough idea is to analyze the convexity properties of suitable functionals, called \emph{entropies}, along geodesics in the space of probability measures endowed with the quadratic Wasserstein distance. In a nutshell,  the more the entropy is \emph{convex} along geodesics the more the space is \emph{Ricci curved}. This led to the definition of $\CD(K,N)$-spaces which corresponds to metric measure spaces having \emph{synthetic Ricci curvature}   bounded below by $K$ and dimension bounded above by $N$, for $K\in \R$ and $N \in [1,\infty]$. See Section \ref{sec:Prel} for the precise notions. 

In order to isolate \emph{Riemannian} from \emph{Finslerian} structures, Ambrosio-Gigli-Savar\'e  \cite{Ambrosio-Gigli-Savare11b}  proposed to strengthen the classical $\CD$ conditions with the requirement that the heat flow is linear (see  \cite{AmbrosioGigliMondinoRajala} for the present simplified axiomatization, and \cite{AmbrosioMondinoSavare, BS10, ErbarKuwadaSturm, gigli:laplacian} for the finite dimensional refinements).  Such a strengthening led to the notion of  $\RCD(K,\infty)$ and $\RCD^{*}(K,N)$ spaces 
 which correspond to \emph{Riemannian} metric measure spaces having \emph{synthetic Ricci curvature}   bounded below by $K$ and dimension bounded above by $N$. For more details see the beginning of  Section \ref{sec:StructQuotRCD}. Let us mention that both $\CD$ and $\RCD$ conditions are compatible with the smooth counterpart and  are stable under pointed measured Gromov-Hausdorff convergence.
\\

To describe the contents and main results of the article, we introduce some notation.  Let $\G$ be a compact Lie group acting on the metric measure space $(M,\sfd,\mm)$ by isomorphisms. That is, by writing $\tau_{g}(x):= g x$ for the translation map by $g \in \G$, we require that $\tau_{g}$ is a measure-preserving isometry:  
\[
\sfd(gx, gy)=\sfd(x,y) \quad \text{and}  \quad (\tau_{g})_{\sharp} \mm=\mm, \qquad   \text{ for all } g \in \G \text{ and  for all }  x,y \in M.
\]
Let us assume that $(M,\sfd)$ is a  geodesic space.
We denote by $M^{*}:=M/\G$ the quotient space and by $\quotient:M \to M^{*}$ the quotient map. Furthermore, we endow $M^{*}$ with the pushforward measure $\mm^{*}:=\quotient_{\sharp} \mm$ and the quotient metric 
\[
\sfd^{*}(x^{*},y^{*}):=\inf_{x \in \quotient^{-1}(x^{*}), y \in \quotient^{-1}(y^{*})} \sfd(x,y).
\]
It has recently been shown by Guijarro and Santos-Rodr\'iguez \cite{GS2016IsoRCD}, and independently by the the fourth author \cite{Sosa2016}, that the  group of (measure-preserving) isometries of an $\RCD^*(K,N)$ space is a Lie group. Moreover, as shown in \cite{Sosa2016}, the corresponding statements are also valid for strong $\CD/\CD^*(K,N)$-spaces and essentially non-branching $\MCP$-spaces for which tangent cones are \emph{well-behaved},  yet might fail to be Euclidean.  We recall that granted that the space is compact, then its isometry group must be compact as well.

Our first  main result states that the quotient space $(M^{*},\sfd^{*}, \mm^{*})$ inherits the synthetic Ricci curvature lower bounds from  $(M,\sfd,\mm)$.

% THEOREM - strong CD preserved under quotients

\begin{thm}[Theorem \ref{thm:CDbyGisCD} and Theorem \ref{thm:M*RCD}]
\label{thm:quotCDRCD}
Assume that $(M,\sfd,\m)$  satisfies one of the following conditions: strong $\CD(K,N)$, strong $\CD(K,\infty)$, strong  $\CD^{*}(K,N)$,  $\RCD^{*}(K,N)$, or $\RCD(K,\infty)$. Then the quotient metric measure space  $(M^{*},\sfd^{*},\m^{*})$ satisfies the corresponding condition for the same parameters $K$ and $N$.
\end{thm}

The next corollary follows from the compatibility of the $\RCD^{*}(K,N)$-condition with its smooth counterpart. 

% COROLLARY 

\begin{cor}
Let $(M,g)$ be a smooth complete $N$-dimensional Riemannian manifold with $\Ric_{g} \geq K \, g$ and denote by $\sfd$  the Riemannian distance and by $\mm$ the Riemannian volume measure on $M$. Then the quotient  metric measure space  $(M^{*},\sfd^{*},\m^{*})$ satisfies the $\RCD^{*}(K,N)$-condition. 
\end{cor}

In particular we recover Lott's result \cite[Theorem 2]{Lott}  saying  that $\Ric_{g^{*}, \Psi^{*},N} \geq K \, g^{*}$  holds, granted that  $(M^{*},\sfd^{*},\m^{*})$ is isomorphic as a m.m.\ space to a smooth weighted Riemannian manifold $(M^{*}, g^{*}, \Psi^{*} \vol_{g^{*}})$. Above, $\Ric_{g^*,\Psi^{*},N}$ is the \emph{Bakry-\'Emery $N$-Ricci tensor} which is defined, for $n\leq N$, as
\begin{equation*}
\textrm{Ric}_{g^{*}, \Psi^{*}, N} :=\textrm{Ric}_{g^{*}}- (N-n) \frac{\nabla_{g^*}^2 (\Psi^{*})^{\frac{1}{N-n}}}{(\Psi^{*})^{\frac{1}{N-n}}},
\end{equation*}
where $n$ is the topological dimension of $M^{*}$ and, when $n=N$, $\Psi^*$ is assumed to be constant and the bound reads as $\Ric_{g^*,\Psi^{*},N} :=\Ric_{g^{*}} \geq K g^{*}$. See Theorem \ref{ref:weakONeillRicciFormula} for a more general statement of this result.

The proof of Theorem \ref{thm:quotCDRCD} is based on a detailed isometric identification of  the space of probability measures on the quotient space $\mathcal{P}_{2}(M^{*})$  with the space of $\G$-invariant probability measures on the original space $\mathcal{P}_{2}^{\G}(M)$  (see  Theorem \ref{thm:LiftIsom}). Let us mention that the stability of the strong $\CD(0,N)$ and strong $\CD(K,\infty)$ conditions for compact spaces under quotients of isometric compact group actions had already been proved,  via an independent argument, in Lott-Villani \cite{LottVillani}.  Beyond dropping the compactness assumption on $(M,\sfd)$, and considering arbitrary lower bounds in the finite dimensional case, the real advantage of Theorem \ref{thm:quotCDRCD} compared to  \cite[Theorem 5.35]{LottVillani} is the extra information that the quotient space is $\RCD^{*}(K,N)$ when the starting  space is so. For us, such additional information is fundamental to obtain the  geometric  applications that we discuss below.

As a first step we establish the next key technical result roughly saying that, under mild assumptions, most of  the orbits are homeomorphic. In order to state the result (see Theorem $\ref{thm:PrincipleOrbit}$ for the complete statement), recall that the isotropy group  of $x\in M$ is the subgroup $\G_{x}\leq \G$ that contains all elements that fix $x$. 
Moreover, by definition we say that  $(M,\sfd,\mm)$ has \emph{good transport behavior} if for all pairs of measures $\mu(\ll \mm),\nu \in \mathcal{P}_{2}(M)$ any $L^{2}$-optimal coupling from $\mu$ to $\nu$ is induced by a map. In this case we write  $\GTB$ for short. Examples of spaces having $\GTB$ are $\RCD^{*}(K,N)$-spaces and, more generally, strong $\CD^{*}(K,N)$-spaces and essentially non-branching $\MCP(K,N)$-spaces  (see Theorem \ref{thm:ExGTB} collecting results from  \cite{CavallettiHuesmann, CMMapsNB, GRS2016}).

% THEOREM - Principal orbit for GTB spaces

\begin{thm}[Principal Orbit Theorem \ref{thm:PrincipleOrbit}]
\label{thm:POIntro}
Assume that $(M,\sfd,\m)$ has $\GTB$. Then, there exists (up to conjugation) a unique subgroup  $\G_{\min}\le\G$  such that the orbit $\G(y)$ of $\,\m$-a.e. $y\in M$  is homeomorphic to the quotient $\nicefrac{\G}{\G}_{\min}$.  
We call $\nicefrac{\G}{\G}_{\min}$ the principal orbit of the action of $\G$  over $(M,\sfd,\m)$.
\end{thm}

The theorem generalizes a previous  result  of  Guijarro and the first author   in the framework of  finite dimensional Alexandrov spaces  \cite{GGG2013}. Recall that for some suitable parameters such spaces are $\RCD^{*}(K,N)$, thus they satisfy $\GTB$. Our proof is independent and uses optimal transport arguments which rely on the $\GTB$ rather than on the strong geometric information granted by the  Alexandrov assumption.

By using the  Principal Orbit Theorem \ref{thm:POIntro} we get the next corollary on cohomogeneity one actions, i.e. actions whose orbit space is one-dimensional.  Compare with the analogous statements in Alexandrov geometry proved in  \cite{GGS2009}.

% COROLLARY
\begin{cor}[Corollary \ref{co:bundle-s1} and Corollary \ref{co:bundle-R}]\label{co:CoHom1}
Let $(M,\sfd,\m)$ have $\GTB$ and let $\G_{\min}$ and $\G/\G_{\min}$ denote a minimal isotropy group and the  principal orbit of the action respectively.
\begin{itemize}
\item Assume that  $(M^{*},\sfd^{*})$ is homeomorphic to a circle $S^{1}$. Then $(M,\sfd)$ is homeomorphic to  a fiber bundle over $S^1$ with fiber the homogeneous space $\G/\G_{\min}$. In particular, $(M,\sfd)$ is a topological manifold.
\item Assume that   $(M^{*},\sfd^{*})$ is homeomorphic  to an interval $[a,b]\subset \R$. Then the open dense set $
\quotient^{-1}((a,b))\subset M$
has full $\m$-measure and is homeomorphic to $(a,b)\times (\G / \G_{\min})$.  In case $(M^*,\sfd^{*})$ is  homeomorphic to $\mathbb R$ we have that $(M,\sfd)$ is homeomorphic to the product ${\mathbb R} \times (\G / \G_{\min})$,  in particular  $(M,\sfd)$ is a topological manifold.
\end{itemize}
\end{cor}

In Section \ref{sec:StructQuotRCD} we establish geometric applications of Theorem \ref{thm:quotCDRCD}  and Theorem \ref{thm:POIntro} to $\RCD^{*}(K,N)$-spaces admitting \emph{large} isomorphic group actions.  

Recall that in \cite{MN2014}  it was proved that $\RCD^{*}(K,N)$-spaces with $N<\infty$
have $\m$-almost everywhere unique Euclidean tangent spaces of possibly
varying dimension. Since the infinitesimal regularity is a metric-measure property and the group $\G$ acts on $(M,\sfd,\mm)$ by measure-preserving isometries, we have that $x\in M$  has a unique tangent space $\R^{n}$ if and only if all the points in the orbit of $x$, $\G(x)$, satisfy the same property. To stress the dependence on the base point we denote the dimension of the tangent space by  $n(x)$ for $\mm$-a.e. $x \in M$,  and by $n(x^{*})$ for $\mm^{*}$-a.e. $x^{*}\in M^{*}$. Set also
\begin{equation}\label{eq:defnIntro}
n:=\underset{x \in M}{{\rm{ess\,inf}}} \;  n(x).
\end{equation}
We now equip $\G$ with a bi-invariant inner metric $\sfd_\G$, whose existence is granted by the compactness of $\G$. Moreover, the arguments that follow are independent of the choice of such a metric since any other bi-invariant inner metric on $\G$ is bi-Lipschitz equivalent to $\sfd_\G$ (see for instance \cite[Theorem 7]{Berestovskii1}). In the next result we strengthen the  assumptions on the group: we assume that the metric on $\G$ is such that the map $\star_y:\G\to\G(y), g\mapsto gy$, is locally Lipschitz and co-Lipschitz continuous, for some (and hence for all) $y$ with principal orbit type.  Specifically, in view of the bi-invariance of the metric,  we assume that for every $y\in M$ of principal orbit, there exist constants $R,C>0$ such that, for all $r\in(0,R)$,
\[
 B_{C^{-1} r}(y)\cap\G(y) \subset \{g \cdot y ~|~ g\in B^\G_r(1_\G)\}\subset B_{Cr}(y) \cap\G(y),
\]
where $B_r^\G(1_\G)$ denotes the $\sfd_\G$-ball of radius $r$ around the identity $1_\G\in \G$.  \\

One of our main geometric applications is the next rigidity/structure result. 

% THEOREM - Structure of RCD under groups actions
\begin{thm}[Theorem  \ref{thm:Rigidity}, Theorem \ref{thm:cohom1RCD}, Theorem \ref{thm:DimM*}]\label{thm:StructRig}
Let $(M,\sfd,\mm)$ be an $\RCD^{*}(K,N)$-space for some $K \in \R, N \in (1,\infty)$ and let $n$  be defined in \eqref{eq:defnIntro}.   Let $\G$ be a compact Lie group  acting effectively  both Lipschitz and co-Lipschitz continuously by measure-preserving isometries of $(M,\sfd,\mm)$. Then the following hold:
\begin{itemize}
\item The group $\G$ has dimension at most $n(n+1)/2$. Moreover, if equality is attained, then the action is transitive and $(M,\sfd,\mm)$ is isomorphic as a m.m.\ space to either  $\mathbb{S}^{n}$ or $\mathbb{RP}^{n}$, up to multiplying the measure $\mm$ by a normalizing constant.
\item Assume that  $N \in [2,\infty)$,  that   the action is not transitive, and  that 
\[
\dim(\G) \geq \frac{(n-1)n}{2}.
\]
Then $\G$ has dimension  $\frac{(n-1)n}{2}$ and acts on $M$ by cohomogeneity one with principal orbit homeomorphic to $\mathbb{S}^{n-1}$ or $\mathbb{RP}^{n-1}$. Moreover, $M^*$ is isometric to a circle or a possibly unbounded closed  interval (i.e. possibly equal to the real line or half line). In the former case, $(M,\sfd)$ is (equivariantly) homeomorphic to a fiber bundle with fiber the principal orbit and base $\mathbb{S}^1$. In particular, in this case, $(M,\sfd)$ is a topological manifold.
\item  If  $\dim(\G) \geq \frac{(n-2)(n-1)}{2}$ and none of the above two possibilities hold, then  $\mm^{*}$-a.e. point in $M^{*}$ is regular with unique tangent  space isomorphic to the 2-dimensional Euclidean space $\R^{2}$.\end{itemize}
\end{thm}
In case that $(M,\sfd,\mm)$ is a smooth Riemannian manifold  the first item is a celebrated theorem due to Kobayashi  \cite[Ch.~II, Theorem~3.1]{Kobayashi}, which for the reader's convenience will be recalled in Theorem \ref{thm:kobayashi} (note that the compactness assumption on $\G$ rules out all the cases except those of $\mathbb{S}^{n}$ and $\mathbb{RP}^{n}$).  This theorem was generalized by Guijarro and Santos-Rodriguez \cite{GS2016IsoRCD} to $\RCD^{*}(K,N)$-spaces, who obtained a bound  in terms of $\lfloor N\rfloor$, the integer part of $N$. Let us mention that in case $(M,\sfd,\mm)$ is a   pointed measured Gromov-Hausdorff limit of a sequence of Riemannian manifolds with Ricci $\geq K$ and dimension $\leq N$, then the assumption that the  topological dimension of $M$ coincides with the integer part of $N$ implies that such a Ricci limit space is non-collapsed \cite{CC1}.
Therefore the first item of Theorem \ref{thm:StructRig}  includes possibly collapsed Ricci limits.
The second and third items of Theorem \ref{thm:StructRig} seem to be new, to the best of our knowledge, even for smooth Riemannian manifolds with Ricci curvature bounded below.
\\

In Sections \ref{Sec:OrbFol} and \ref{sec:foliations} we apply the methods of the paper to study curvature properties of orbifolds, orbispaces, and foliations.

For orbifolds we  show that the $\RCD^{*}(K,N)$-condition  is equivalent to requiring a lower Ricci curvature bound $K$ on the regular part and an upper dimension bound $N$ of the orbifold. Using such equivalence we  prove rigidity results for orbifolds that extend previous works by Borzellino \cite{Borzellino,Borzellino1993MaxDiam}:   we generalize to weighted Riemannian orbifolds Cheng's celebrated Maximal Diameter Theorem,  and we prove a  rigidity result for orbifolds having maximal volume. Consult Theorems \ref{THM:ORBI_EQUIV}, \ref{thm:Cheng}, \ref{cor:OrbRig} for the precise statements.

In Theorem \ref{thm:StabMetricOrbSpa} and in Theorem \ref{thm:RCDorbispaces},  we generalize the stability under quotients of synthetic Ricci curvature lower bounds to metric-measure orbispaces and  give  applications to almost free discrete group actions in Theorem \ref{thm:discreteActRCD} and in Corollary  \ref{cor:discreteActRCD}.

Finally, we show the stability of synthetic Ricci curvature lower bounds under  metric-measure foliations and submetries in Theorem \ref{thm:CDbyMMFisCD} and Corollary \ref{cor:CDbyMMFisCD}. Regarding this point, let us mention that the corresponding result for sectional curvature lower bounds was proved in the celebrated paper of Burago-Gromov-Perelman \cite[Section 4.6]{BGP}. More precisely, it was shown in \cite{BGP} that if $(M,g)$ is a smooth manifold with \emph{sectional} curvature bounded below by $K\in \R$ endowed with a metric foliation, then the quotient space is Alexandrov  with curvature bounded below by $K$. We prove the analogous statement in the category of metric-measure spaces with Ricci curvature bounded below. 
If $(M,\sfd,\m)$ is an $\RCD^{*}(K,N)$-space endowed with a \emph{metric-measure} foliation, then the quotient space is an $\RCD^{*}(K,N)$-space. In particular, this is valid for smooth Riemannian $N$-manifolds with Ricci curvature bounded below by $K$. We refer to Section \ref{sec:foliations} for the precise notions and statements.
\\

As a consequence of Theorem \ref{thm:quotCDRCD} and  the results in Sections \ref{Sec:OrbFol} and \ref{sec:foliations} we can enlarge the list of examples of $\RCD^{*}(K,N)$-spaces to include quotients by isometric group actions and foliations of Riemannian manifolds with $\Ric_{g} \geq K g$, and more generally, of $\RCD^{*}(K,N)$-spaces  as well as orbifolds with  $\Ric_{g} \geq K g$ on the regular part. Note that, since  orbifolds can be seen as a generalization of cones, the $\RCD$ property of orbifold (see Theorem \ref{THM:ORBI_EQUIV}) can be seen as an extension of the $\RCD$ property of cones proved by Ketterer \cite{Ketterer}.
\\

To conclude, in the two appendices we show that the methods of the paper can be useful also to study quotients by isometric group actions of discrete spaces (see Appendix \ref{App:Discrete}) and (super-)Ricci flows (see Appendix \ref{App:RF}). In particular we show that if $(M,g_{t})_{t\in [0,T]}$ is a Ricci flow of complete metrics with bounded curvature,  and if $\G$ is a compact Lie subgroup of the isometry group of  the initial datum $(M,g_{0})$ then the quotient flow is a super-Ricci flow in the sense of Sturm \cite{Sturm:SuperRicciI}. See Proposition \ref{Prop:QuotSRF} and Corollary \ref{cor:QuotRF} for the precise statements.

\begin{ack}
The paper was initiated during the Follow-up Workshop to the Junior Trimester Program ``Optimal Transportation'' at the  Hausdorff Center of Mathematics in Bonn. The authors wish to thank the HCM and the organizers of the workshop for the excellent working conditions and the  stimulating atmosphere.
\end{ack}

% SECTION: PRELIMINARIES
%----------------------------------------------------------------------------------

\section{Preliminaries}\label{sec:Prel}

% SUBSECTION: MMS and optimal tranport
%----------------------------------------------------------------------------------

\subsection*{Metric measure spaces and optimal transport theory}
Throughout the paper $(M,\sfd)$ denotes a complete separable metric space and $\m$ is a $\sigma$-finite non-negative Borel measure on it with full support, that is, $\mathrm{supp}(\m)=M$. The triple  $(M,\sfd, \m)$ is called a metric measure space, m.m. space for short.
\\A metric space is a \emph{geodesic space} if and only if for each $x,y \in M$ 
there exists $\gamma \in \Geo(M)$ so that $\gamma_{0} =x, \gamma_{1} = y$, with
$$
\Geo(M) : = \{ \gamma \in C([0,1], M):  \sfd(\gamma_{s},\gamma_{t}) = |s-t| \sfd(\gamma_{0},\gamma_{1}), \text{ for every } s,t \in [0,1] \}.
$$
For simplicity, we will always assume $(M,\sfd)$ to be geodesic. Given a map $f: M\to M$, we denote by $f_{\sharp} \m$ the push-forward measure defined by $[f_{\sharp} \m] (B):=\m(f^{-1}(B))$ for every Borel set $B\subset M$.
We denote with $\mathcal{P}(M)$ the space of Borel  probability measures over $M$, and with   $\mathcal{P}_{p}(M)$ the subspace of probabilities with finite $p$-moment,  $p \in [1,\infty)$:
\begin{equation*}
 \mathcal{P}_{p}(M):=\{\mu \in  \mathcal{P}(M)\,:  \int_{M} \sfd^{p}(x,x_{0})  d \mu(x) < 	\infty \text{ for some (and then for every) } x_{0} \in M \}. 
\end{equation*}
We endow the space $\mathcal{P}_{p}(M)$ with the  $L^{p}$-Wasserstein distance  $W_{p}$ defined as follows:  for $\mu_0,\mu_1 \in \mathcal{P}_{p}(X)$ we set
\begin{equation}\label{eq:Wdef}
  W_p(\mu_0,\mu_1)^p = \inf_{ \pi} \int_{M\times M} \sfd^p(x,y) \, d\pi(x,y),
\end{equation}
where the infimum is taken over all $\pi \in \mathcal{P}(M \times M)$ with $\mu_0$ and $\mu_1$ as the first and the second marginal.
\\Assuming the space $(M,\sfd)$ to be geodesic, also the space $(\mathcal{P}_p(M), W_p)$ is geodesic. 
\\Any geodesic $\{\mu_t\}_{t \in [0,1]}$ in $(\mathcal{P}_p(M), W_p)$  can be lifted to a measure $\Pi \in {\mathcal {P}}(\Geo(M))$, 
so that $({\rm e}_t)_\sharp \, \Pi = \mu_t$ for all $t \in [0,1]$.  {Moreover, in this case,  $({\rm e}_0, {\rm e}_1)_{\sharp}\Pi$ is a minimizer in \eqref{eq:Wdef} called \emph{optimal coupling induced by the geodesic $\{\mu_{t}\}_{t \in [0,1]}$}.
Here ${\rm e}_{t}$ denotes the evaluation map which is defined for $t\in [0,1]$ by: 
\[
  {\rm e}_{t} : \Geo(M) \to M, \qquad {\rm e}_{t}(\gamma) : = \gamma_{t}.
\]
\noindent
Given $\mu_{0},\mu_{1} \in \mathcal{P}_{p}(M)$, we denote by 
$\Opt_{p}(\mu_{0},\mu_{1})$ the space of all $\Pi \in \mathcal{P}(\Geo(X))$ for which $({\rm e}_0,{\rm e}_1)_\sharp\, \Pi$ 
realizes the minimum in \eqref{eq:Wdef}. By the observations above we have that if $(M,\sfd)$ is geodesic then the set  $\Opt_{p}(\mu_{0},\mu_{1})$ is non-empty for any $\mu_0,\mu_1\in \mathcal{P}_p(M)$.  

It is also worth introducing the subspaces  $ \mathcal{P}^{ac}(M,\sfd,\m):=\{\mu \in  \mathcal{P}(M)\,:\, \mu \ll \m \}$ and  $\mathcal{P}_{p}^{ac}(M,\sfd,\m):= \mathcal{P}_{p}(M)  \cap \mathcal{P}^{ac}(M,\sfd,\m)$. We  simply write $\mathcal{P}^{ac}(M)$ whenever it is clear from the context to which reference measure we refer.

We  say that an optimal dynamical plan  $\Pi\in \Opt_{p}(\mu_{0},\mu_{1})$ is given by the map $G:X\to \Geo(M)$ if $\Pi=G_{\sharp} \mu_{0}$. 
Notice that in this case, in particular, the optimal transference plan $(\ee_{0}, \ee_{1})_{\sharp} \Pi$ is induced  by the optimal map $\ee_{1}\circ G$.
%
% DEFINITION - Set of non-branching geodesics
\begin{defn}
We call a set $\Gamma \subset \Geo(M)$ \emph{non-branching} if for any $\gamma^{1},\gamma^{2} \in \Gamma$ we have: 
if there exists $t \in (0,1)$ such that $\gamma^{1}_{s} = \gamma^{2}_{s}$ for all $s \in [0,t]$, then $\gamma^{1} = \gamma^{2}$.

\medskip
\noindent
A measure $\Pi \in \mathcal{P}(\Geo(M))$ is \emph{concentrated on a set of non-branching geodesics} if there exists a non-branching Borel set $\Gamma \subset \Geo(M)$ such that $\Pi(\Gamma) = 1$.
\end{defn}

We now recall the following definition given for the first time in \cite{RS2014}.

% DEFINITION - Essentially non-branching
\begin{defn}\label{def:essNB}
A metric measure space $(M,\sfd,\m)$ is \emph{essentially non-branching} if for every $\mu_{0},\mu_{1} \in \mathcal{P}^{ac}(M)$
any $\Pi \in \Opt_{2}(\mu_{0},\mu_{1})$ is concentrated on a set of non-branching geodesics.
\end{defn}

Let us recall the definition of $c_{p}$-transform which will turn out to be useful in the sequel. Let $\psi:M \to \R\cup\{\pm \infty\}$ be any function. Its \emph{$c_{p}$-transform} $\psi^{c_{p}}:M \to \R\cup\{\pm \infty\}$ is defined as
\begin{equation} \label{eq:defcpTransf}
\psi^{c_{p}}(x):= \inf_{y \in M} \left( \sfd^{p}(x,y) - \psi(y)\right ).
\end{equation}
We say that $\phi: M \to \R\cup\{-\infty\}$ is \emph{$c_{p}$-concave}  if there exists $\psi:M \to \R\cup\{- \infty\}$ such that $\phi=\psi^{c_{p}}$. The $c_p$-superdifferential of a $c_p$-concave function $\varphi$ is defined as the set
\[
\partial^{c_p}\varphi := \{(x,y)\in M \times M \,|\, \varphi(x)+\varphi^{c_p}(y)=\sfd^{p}(x,y)\}.
\]
In what follows, we will also make use of the dual formulation of the optimal transport problem \eqref{eq:Wdef} stating that 
\begin{equation}\label{eq:dual}
W_p(\mu_0,\mu_1)^p = \sup_{(\phi,\psi)} \quad \int_{M} \phi \, d \mu_{0} + \int_{M} \psi \, d \mu_{1}, 
\end{equation}
the supremum taken over all the pairs  $(\phi,\psi)$ with $\phi \in L^{1}(\mu_{0}), \psi \in L^{1}(\mu_{1})$ such that 
$$\phi(x)+\psi(y)\leq \sfd^{p}(x,y), \quad \forall x,y \in M.$$
It is a classical result of optimal transport  that the supremum in the dual problem \eqref{eq:dual} is always attained by a  maximizing couples of the form $(\phi, \phi^{c_{p}})$ for some $c_{p}$-concave function $\phi$ called \emph{$c_{p}$-Kantorovich potential}.
\\Finally let us also recall that if $\pi$ is a $p$-optimal coupling for $(\mu_{0},\mu_{1})$ then there exists a $c_{p}$-Kantorovich potential $\phi$ such that 
$$\phi(x)+\phi^{c_{p}}(y) = \sfd(x,y)^{p}, \quad \text{for $\pi$-a.e. } (x,y). $$
In this case we say that $(\phi, \phi^{c_{p}})$ is a $p$-dual solution  corresponding to $\pi$.

It goes beyond the scope of this short section to give
a comprehensive introduction to optimal transport, for this purpose we refer the reader to \cite{AGUser} and \cite{Villani}.

% SUBSECTION: RICCI CURVATURE LOWER BOUNDS FOR METRIC MEASURE SPACES
%----------------------------------------------------------------------------------

\subsection*{Ricci curvature lower bounds for metric measure spaces}

In order to formulate the curvature properties for a metric measure space $(M,\sfd,\m)$, we introduce the following distortion coefficients: given two numbers $K,N\in \R$ with $N\geq0$, we set for $(t,\theta) \in[0,1] \times \R_{+}$, 
% EQUATION - Distortion coefficients
\begin{equation}\label{E:sigma}
\sigma_{K,N}^{(t)}(\theta):= 
\begin{cases}
\infty, & \textrm{if}\ K\theta^{2} \geq N\pi^{2}, \crcr
\displaystyle  \frac{\sin(t\theta\sqrt{K/N})}{\sin(\theta\sqrt{K/N})} & \textrm{if}\ 0< K\theta^{2} <  N\pi^{2}, \crcr
t & \textrm{if}\ K \theta^{2}<0 \ \textrm{and}\ N=0, \ \textrm{or  if}\ K \theta^{2}=0,  \crcr
\displaystyle   \frac{\sinh(t\theta\sqrt{-K/N})}{\sinh(\theta\sqrt{-K/N})} & \textrm{if}\ K\theta^{2} \leq 0 \ \textrm{and}\ N>0.
\end{cases}
\end{equation}
\noindent
We also set, for $N\geq 1, K \in \R$ and $(t,\theta) \in[0,1] \times \R_{+}$
\begin{equation} \label{E:tau}
\tau_{K,N}^{(t)}(\theta): = t^{1/N} \sigma_{K,N-1}^{(t)}(\theta)^{(N-1)/N}.
\end{equation}
Before stating the definitions of  curvature-dimension conditions proposed  in the pioneering works of Lott-Villani \cite{LottVillani} and Sturm \cite{sturm:I,sturm:II} (and Bacher-Sturm \cite{BS10} for the  reduced condition $\CD^{*}(K,N)$), let us recall that the  Shannon (relative) entropy functional $\Ent_{\m}: \mathcal{P}(M) \to [-\infty, +\infty]$ is defined as
\begin{equation}\label{eq:defEntm}
\Ent_{\m}(\mu):= \int_{M} \rho \, \log \rho \, d \m,
\end{equation}
if  $\mu=\rho \m \in  \mathcal{P}^{ac}(M,\sfd,\m)$ and if  the positive part of $ \rho \, \log \rho$ is integrable. For other measures in   $\mathcal{P}(M)$ we set $\Ent_{\m}(\mu)=+\infty$.
 
% DEFINITION - CD(K,N)
\begin{defn}[Curvature-dimension conditions]
Let $(M,\sfd,\m)$ be a metric measure space and fix  $K\in \R, N \in [1,+\infty]$. We say that $(M,\sfd,\m)$ satisfies, for $N<+\infty$, the  
\begin{itemize}
\item \emph{$\CD(K,N)$-condition}  if for each pair $\mu_{0}=\rho_{0} \m,\, \mu_{1}=\rho_{1} \m \in \mathcal{P}_{2}^{ac}(M,\sfd,\m)$ there exists  a $W_{2}$-geodesic $\{\mu_{t}=\rho_{t} \m\}_{t \in [0,1]} \subset  \mathcal{P}_{2}^{ac}(M,\sfd,\m) $ such that
\begin{align}\label{def:CDKN}
\int_{M} \rho_{t}^{1-\frac{1}{N'}} d\m \geq \int_{M\times M} \Big[ \tau^{(1-t)}_{K,N'} \big(\sfd(x,y) \big) \rho_{0}(x)^{-\frac{1}{N'}} +    \tau^{(t)}_{K,N'} \big(\sfd(x,y) \big) \rho_{1}(y)^{-\frac{1}{N'}} \Big]   d \pi(x,y)
\end{align}
holds for all $t\in [0,1]$ and $N'\in [N, \infty)$, where $\pi$ is the  optimal coupling  induced by $\{\mu_{t}\}_{t \in [0,1]}$.

\item \emph{$\CD^{*}(K,N)$-condition}  if for each pair $\mu_{0}=\rho_{0} \m,\, \mu_{1}=\rho_{1} \m \in \mathcal{P}_{2}^{ac}(M,\sfd,\m)$ there exists  a $W_{2}$-geodesic $\{\mu_{t}=\rho_{t} \m\}_{t \in [0,1]} \subset  \mathcal{P}_{2}^{ac}(M,\sfd,\m) $ such that
\begin{align}\label{def:CD*KN}
\int_{M} \rho_{t}^{1-\frac{1}{N'}} d\m \geq \int_{M\times M} \Big[ \sigma^{(1-t)}_{K,N'} \big(\sfd(x,y) \big) \rho_{0}(x)^{-\frac{1}{N'}} +    \sigma^{(t)}_{K,N'} \big(\sfd(x,y) \big) \rho_{1}(y)^{-\frac{1}{N'}} \Big]   d \pi(x,y)
\end{align}
holds for all $t\in [0,1]$ and $N'\in [N, \infty)$, where $\pi$ is the  optimal coupling  induced by $\{\mu_{t}\}_{t \in [0,1]}$.
\item  \emph{$\MCP(K,N)$-condition}  if for each $o\in M$ and $\mu_{0}=\rho_{0} \m \in  \mathcal{P}_{2}^{ac}(M,\sfd,\m)$, denoting with $\mu_{1}=\delta_{o}$,  there exists  a $W_{2}$-geodesic $\{\mu_{t}\}_{t \in [0,1]}\subset \mathcal{P}_{2}(M)$ such that 
\begin{align}\label{def:MCP}
\int_{M} \rho_{t}^{1-\frac{1}{N'}} d\m \geq \int_{M}  \tau^{(1-t)}_{K,N'} \big(\sfd(x,o) \big) \rho_{0}(x)^{1-\frac{1}{N'}} \, d \m,  
\end{align}
holds for all $t\in [0,1]$ and $N'\in [N, \infty)$, where we write $\mu_{t}=\rho_{t} \m+\mu_{t}^{s}$ with $\mu_{t}^{s} \perp \m$.
\end{itemize}
Additionally, we say that $(M,\sfd,\m)$ satisfies the 
\begin{itemize}
\item  \emph{$\CD(K,\infty)$-condition}  if for each $\mu_{0},\mu_{1} \in  \mathcal{P}_{2}^{ac}(M,\sfd,\m)$ there exists  a $W_{2}$-geodesic $\{\mu_{t}\}_{t \in [0,1]}$ along which $\Ent_{\m}$ is $K$-convex, i.e.
\begin{align}\label{def:CDKinfty}
\Ent_{\m}(\mu_{t}) \leq (1-t)  \Ent_{\m}(\mu_{0}) + t  \Ent_{\m}(\mu_{1})- \frac{K}{2} t (1-t) W_{2}(\mu_{0},\mu_{1})^{2},
\end{align}
holds for all $t\in [0,1]$
\end{itemize}
\noindent
We also say that  $(M,\sfd,\m)$ satisfies the  \emph{strong $\CD(K,N)$-condition} (resp. \emph{strong $\CD^{*}(K,N)$-condition,  strong $\CD(K,\infty)$}) if \eqref{def:CDKN} (resp.  \eqref{def:CD*KN},  \eqref{def:CDKinfty}) holds for \emph{every}  $W_{2}$-geodesic $\{\mu_{t}\}_{t \in [0,1]}$ connecting  any given pair  $\mu_{0}=\rho_{0} \m, \mu_{1}=\rho_{1} \m \in P_{2}^{ac}(M, \sfd,\m)$.
\end{defn}

In contrast to the smooth framework, in general the curvature-dimension conditions defined above are not equivalent, in some cases ones being more restrictive than others. However there is an inclusion relation between them that can be written as: 
\begin{center}
$\CD^*$-spaces	$\subsetneq \CD$-spaces $\subsetneq \MCP$-spaces,
\end{center}
 where all the inclusions are proper and the parameters $K,N$ might vary. Additionally, the inclusion $\CD (K,N<\infty)\subset \CD (K,\infty)$ holds as well.

% REMARK

\begin{remark}\label{rem:EssNBsCD}
A m.m. space $(M,\sfd,\m)$ satisfies \emph{strong} $\CD(K,N)$  (resp. \emph{strong} $\CD^{*}(K,N)$) for some $K\in \R, N\in [1,+\infty)$ if and only if it is essentially non-branching and satisfies $\CD(K,N)$  (resp.  $\CD^{*}(K,N)$). Indeed, on the one hand   strong $\CD(K,N)$  (resp. strong $\CD^{*}(K,N)$)  implies essentially non-branching \cite[Theorem 1.1]{RS2014} and, of course, $\CD(K,N)$. On the other hand a $\CD(K,N)$  (resp.  $\CD^{*}(K,N)$) space is proper for $N<+\infty$ and moreover, if it is essentially non-branching, then \cite[Corollary 5.3]{CMMapsNB} implies that there is a unique $W_{2}$-geodesic connecting any pair $\mu_{0},\mu_{1} \in \mathcal{P}_{2}^{ac}(M,\sfd,\m)$, therefore the strong $\CD(K,N)$-condition trivially holds.

Let us also mention that thanks to \cite[Theorem 1.1]{CMMapsNB} if  $(M,\sfd,\m)$ is an essentially non-branching m.m. space satisfying  $\MCP(K,N)$, for some $K\in \R, N\in [1,+\infty)$, then  for  any pair $\mu_{0},\mu_{1} \in \mathcal{P}_{2}^{ac}(M,\sfd,\m)$
 there is a unique $W_{2}$-geodesic connecting  them and the corresponding dynamical plan is given by a map.
\end{remark}

% SUBSECTION: GROUP ACTIONS AND LIFTS OF MEASURES
%----------------------------------------------------------------------------------

\subsection*{Group actions and lifts of measures.}

Let $\G\times M\to M$, $(g,x)\mapsto gx$, be an action by isomorphisms of m.m. spaces of 
a compact Lie group $\G$ on the m.m. space $(M,\sfd, \m)$, meaning the following:
\begin{itemize}
\item $(gh) \,x= g \,(hx)$ for all $x \in M$ and all $g,h \in \G$;
\item $1_{\G} \, x=x$, for all  $x \in M$,  where $1_{\G}$ is the neutral element of the group $\G$;
\item for every fixed $x \in X$, the map $\star_x:\G\to M$ given by $g\mapsto gx$ is continuous;
\item for every fixed $g \in \G$, the map $\tau_{g}:M\to M$ given by $x\mapsto \tau_{g}(x):= g x$ is an isomorphism of m.m. spaces, i.e. $\tau_{g}$ is an  isometry and moreover $(\tau_{g})_{\sharp} \m=\m$. 
\end{itemize}
 The \textit{isotropy
group} at $x\in M$ is defined by  $\G_{x}=\{g\in\G\mid gx=x\}$. Observe that
the orbit $\G(x)$ is homeomorphic to $\G/\G_{x}$. We denote the\textit{
orbit space }by $M/\G$ or $M^{*}$. 
Let $\mathsf{\quotient}\colon M\to M^{*}$ be the projection onto
the orbit space. We let $X^{\ast}$ denote the image of a subset $X$
of $M$ under $\quotient$, i.e. $X^{*}=\quotient (X)$. The action is called \textit{effective}
if the intersection of all isotropy subgroups of the action is trivial.
An element of $M^{*}$ will be denoted by $x^{*}= \quotient(x)= \quotient(\{gx\mid g\in\G\})$.
On $M^{*}$ we define the following distance:
\[
\sfd^{*}(x^{*},y^{*}):=\inf_{g\in\G} \sfd(gx,y)=\inf_{g,h\in\G} \sfd(gx,hy)= \sfd(\G(x),\G(y)),
\]
where in the last equality we consider $\G(x)$ and $\G(y)$ as subsets
of $M$. Note that $\sfd^{*}$ is well defined because, by assumption, $\G$ acts isometrically; the compactness of $\G$ assures that the function $\sfd^{*}$
is a complete metric on $M^{*}$ and that $(M, \sfd^{*})$ is separable. Moreover, if $(M,\sfd)$ is a geodesic space then also $(M^{*},\sfd^{*})$ is geodesic.

Given a function $\phi:M^{*}\to \R\cup \{-\infty\}$ we define $\hat{\phi}:M\to  \R\cup \{-\infty\}$ as
\begin{equation}\label{eq:defhatphi}
\hat{\phi}(x):=\phi(x^{*}), 
\end{equation}
and refer to $\hat{\phi}$ as the $\emph{lift}$ of the function $\phi$. 

Let $\m^{*}:=\quotient_{\sharp}\m$ be the push-forward measure on $M^{*}$; note that $\m^{*}$ is a non-negative $\sigma$-finite Borel measure over the complete and separable metric space $(M^{*}, \sfd^{*})$.
By the Disintegration Theorem (see for instance \cite[Section 452]{Fremlin}) there is an $\m^{*}$-almost everywhere
uniquely defined measurable assignment of probability measures $x^{*}\mapsto\m_{x^{*}}\in\mathcal{P}(M)$
such that $\m_{x^{*}}$ is concentrated on $\G(x)\subset M$
and, for any Borel set $B\subset M$, the following identity holds:
\begin{equation}\label{eq:mmDis}
\m(B)=\int_{M^{*}}\m_{x^{*}}(B\cap\G(x))d\m^{*}(x^{*}).
\end{equation}

We say that a  Borel measure $\mu$ is \emph{invariant} under the action of $\G$ if   $\mu(\tau_{g}(B))=\mu(B)$
for any $g\in\G$ and any Borel set $B\subset M$, and denote  the set of probability measures on $M$ that are $\G$-invariant with $\mathcal P^\G(M)\subset \mathcal P (M)$. We stress that $\mu$ being invariant is equivalent to $g_\sharp \mu =\mu$ for every $g\in\G$ since $\G$ acts by isomorphisms of m.m.  spaces.

Given $x \in M$ recall that the map $\star_x:\G \to M$ is defined as $g \mapsto gx$. Using the Haar  probability measure $\nu_{\G}$ of $\G$ we can define a $\G$-invariant measure $\nu_{x} \in \mathcal{P}(M)$ supported on $\G(x)$ as follows
\begin{equation}\label{eq:defmux}
\nu_{x}:= (\star_x)_\sharp \nu_\G = \int_{\G}\delta_{gx}\, d\nu_{\G}(g),
\end{equation}
for every $x \in M$.  Note that $\nu_{x}$ is the unique $\G$-invariant probability measure that satisfies $\quotient_{\sharp}\nu_{x}=\delta_{x^{*}}$ and that the assignment  $x\mapsto\nu_{x}$ is a measurable function. Furthermore, since $\nu_{x}=\nu_{y}$ whenever $x=gy$ for some $g\in\G$, also the assignment $x^{*}\mapsto\nu_{x}$ is well-defined and measurable.
Using the measures $\nu_x$, we can define a lifting map $\Lambda$ of measures on $M^*$ as
\begin{align}\label{eq:liftmeas}
\begin{split}
\Lambda : \;& \mathcal P (M^*) \to  \mathcal P^\G (M)\\
&\mu \mapsto \hat{\mu}:=\Lambda(\mu)=\int_{M^{*}}\nu_{x} \,d\mu(x^{*}).
\end{split}
\end{align}
Note that by definition $\hat{\mu}$ is $\G$-invariant and that $\quotient_{\sharp}\hat{\mu}=\mu$, hence our referring to $\Lambda$ as a lifting map. In the upcoming section we show that the lift $\Lambda$ is in fact a canonical isometric embedding that preserves useful measurable properties. 

% SECTION 
% --------------------------
\section{Equivariant Wasserstein geometry and stability of  Ricci  bounds}

We begin this section by showing that $c_p$-concavity is preserved when lifting functions.  

% LEMMA

\begin{lem}
\label{lem:c-concave-lifts}
The lift of a $c_{p}$-concave function is $c_{p}$-concave. That is, if $\phi:M^{*}\to\mathbb{R}\cup\{-\infty\}$ is a $c_{p}$-concave function on $M^{*}$ then $\hat{\phi}:M\to\mathbb{R}\cup\{-\infty\}$
defined by $\hat{\phi}(x):=\phi(x^{*})$ is $c_{p}$-concave on $M$.
\end{lem}

% PROOF

\begin{proof}
Let $\psi:M^{*}\to  \R\cup \{-\infty\}$ be  such that $\phi=\psi^{c_{p}}$ and denote
its lift by $\hat{\psi}:M\to  \R\cup \{-\infty\}$, $\hat{\psi}(x):=\psi(x^{*})$. Then 
\begin{align*}
\inf_{y\in M} \big ( \sfd(x,y)^{p}-\hat{\psi}(y) \big )=\inf_{y^{*}\in M^{*}} \big ( \sfd^{*}(x^{*},y^{*})^{p}-\psi(y^{*}) \big ) =\phi(x^{*})=\hat{\phi}(x),
\end{align*}
which shows that $\hat{\phi}=\hat{\psi}^{c_{p}}$.  
\end{proof}

Note that the proof of the lemma only depends on the definition of the quotient metric and thus also holds for quotients of non-compact groups.
The next result  will be a key tool in the proof of the main result of the section. 

% THM - equivariant lifting is an isometric embedding

\begin{thm}\label{thm:LiftIsom}
Let $\Lambda: \mathcal P (M^*) \to  \mathcal P^\G (M)$ be the lift of measures defined in (\ref{eq:liftmeas}). Assume that $\mu_0,\mu_1\in\mathcal P (M^*)$ and denote by $\hat{\mu}_0, \hat{\mu}_1 \in \mathcal P^G(M)$ their respective lifts. Then the following holds:
\begin{enumerate}
\item $\hat{\mu}_0$ is the unique $\G$-invariant probability measure with $\quotient_{\sharp} \hat{\mu}_0=\mu_0$;
\item $\Lambda (\mathcal{P}^{ac}(M^*)) = \mathcal P ^{ac}(M) \cap \mathcal P^G(M)$;
\item $\Lambda (\mathcal{P}_{p}(M^*)) = \mathcal P_{p}(M)\cap \mathcal P^G(M)$; 
\item $W_{p}(\hat{\mu}_{0},\hat{\mu}_{1})=W_{p}(\mu_{0},\mu_{1})$ whenever $\mu_0, \mu_1 \in \mathcal P_p(M^*)$.
\end{enumerate}
That is, $\Lambda: \mathcal P_p (M^*)\hookrightarrow \mathcal P_p (M)\cap \mathcal P^\G(M) $ is an isometric embedding  which preserves absolutely continuous measures. In particular, lifts of $W_{p}$-geodesics in $\mathcal{P}_{p}(M^{*})$ are $\G$-invariant $W_p$-geodesics in $\mathcal{P}_{p}(M)$.
\end{thm}
\begin{proof}
As noted above, the definition of the lift $\Lambda$ implies that $\hat{\mu}_0$ is $\G$-invariant and that $\quotient_{\sharp}\hat{\mu}_0=\mu_0$. Moreover, for every $x\in M$, the measure $\nu_{x}$ defined in  (\ref{eq:defmux}) is characterized as the unique $\G$-invariant probability  measure on $M$ satisfying $\quotient_{\sharp}\nu_{x}=\delta_{x^{*}}$. Therefore, the uniqueness part of the Disintegration Theorem implies that $\hat{\mu}$ is the unique probability measure on $M$ satisfying the aforementioned properties, which proves the first statement.

Next, assume that $\mu$ is absolutely continuous. Then there is a
non-negative $\m^{*}$-measurable function $f\in L^{1}(\m^{*})$ such
that $\mu=f\m^{*}$. We denote its lift by $\hat{f}:M\to[0,\infty)$ and set $\hat{\mu}':= \hat{f} \m$. Then we get that $\hat{\mu}'$ is $\G$-invariant and 
$\quotient_{\sharp} \hat{\mu}'=\mu$. The uniqueness of the lift gives  that $\hat{\mu}=\hat{\mu}'$ and in particular $\hat{\mu} \ll \m$; thus  (2) follows. 

To finish the proof we define a way to lift couplings of measures. Assume that $\mu_0,\mu_1\in\mathcal P_p(M^*)$ and fix an admissible coupling $\pi\in \mathcal{P}(M^{*}\times M^{*})$ of $(\mu_0,\mu_1)$.  Consider the subset of pairs of points in $M\times M$ that achieve the distance between orbits,
\begin{equation}\label{eq:defOD}
\mathcal{OD}:=\{(x,y)\in M\times M\,|\, \sfd(x,y)=\sfd^{*}(x^{*},y^{*})\}.
\end{equation}

As a first step we define for every $(x,y)\in\mathcal{OD}$  a $\G$-invariant coupling $\pi_{x,y}\in\mathcal{P}(M\times M)$ with $(p_{1})_{\sharp}\pi_{x,y}=\nu_{x}$ and $(p_{2})_{\sharp}\pi_{x,y}=\nu_{y}$, where $\G$ acts on $M\times M$ in the canonical diagonal manner. Specifically,
for $(x,y)\in\mathcal{OD}$ we set
\begin{equation}\label{eq:dedpixy}
\pi_{x,y}:=(\star_{(x,y)})_\sharp\nu_\G=\int_{\G}(\delta_{gx}\times\delta_{gy})d\nu_{\G}(g),
\end{equation}
where $\star_{(x,y)}:\G \to M\times M$ is the map $g \mapsto (gx,gy)$. Note that the measure $\pi_{x,y}$ is concentrated on $\mathcal{OD}\cap \G((x,y))$,
and that $\pi_{x,y}=\pi_{gx,gy}$ for every $x,y\in\mathcal{OD},$
and $g\in\G$. 
Thus, we obtain that 
\begin{equation}\label{eq:pixyd*}
\int_{M\times M}\sfd(w,z) ^{p} \, d\pi_{x,y}(w,z)  =  \int_{M\times M} \sfd^{*}(w^*,z^*) ^{p} \,d\pi_{x,y}(w,z) =  \sfd^{*}(x^{*},y^{*})^{p}.
\end{equation}
However, since $\nu_{x}$ is the unique $\G$-invariant lift of $\delta_{x^{*}}$,
and $\sfd(x,y)\ge \sfd^{*}(x^{*},y^{*})$ it always holds that
\begin{equation}\label{eq:Wpd*}
W_{p}(\nu_{x},\nu_{y})\ge \sfd^{*}(x^{*},y^{*}).
\end{equation}
The combination of   \eqref{eq:pixyd*} and \eqref{eq:Wpd*} gives that 
\begin{align}\label{eq:Wp=d*}
\begin{split}
& W_{p}(\nu_{x},\nu_{y})= \sfd^{*}(x^{*},y^{*}) \text{ and } \\
& \pi_{x,y}  \text{ is a $p$-optimal coupling for $(\nu_{x}, \nu_{y})$}, \text{ for all } (x,y) \in \mathcal{OD}. 
\end{split}
\end{align}

Since by assumption the $\G$-orbits in $M$ are compact, for every pair of equivalence classes $x^{*},y^{*} \in M^{*}$ we can find (non-unique) representatives $\bar{x}_{(x^{*},y^{*})} \in \quotient^{-1}(x^{*}),$ $\bar{y}_{(x^{*},y^{*})}  \in \quotient^{-1}(y^{*})$ such that $(\bar{x}_{(x^{*},y^{*})}, \bar{y}_{(x^{*},y^{*})}) \in \mathcal{OD}$.  By a standard measurable selection argument we can assume that the map $$M^{*}\times M^{*}\ni (x^{*},y^{*}) \mapsto (\bar{x}_{(x^{*},y^{*})}, \bar{y}_{(x^{*},y^{*})}) \in M\times M$$ is measurable.  At this point we are able to define a  lift $\hat{\pi} \in {\mathcal P}(M\times M)$ of $\pi\in {\mathcal P}(M^{*}\times M^{*})$ as 
\begin{equation}\label{eq:defhatpi}
\hat{\pi}:=\int_{M^{*}\times M^{*}}\pi_{\bar{x}_{(x^{*},y^{*})}, \bar{y}_{(x^{*},y^{*})} } \,d\pi(x^{*},y^{*})= \int_{M^{*}\times M^{*}} \left( \int_{\G} \delta_{g \bar{x}_{(x^{*},y^{*})}} \times  \delta_{g \bar{y}_{(x^{*},y^{*})}} \, d \nu_{\G}(g) \right) \,d\pi(x^{*},y^{*}) .
\end{equation}
Let us stress that while the lift $\hat{\mu}\in  {\mathcal P}(M)$ of a measure $\mu\in  {\mathcal P}(M^{*})$ is canonical, instead, the lift $\hat{\pi} \in {\mathcal P}(M\times M)$ of  a plan $\pi\in {\mathcal P}(M^{*}\times M^{*})$ defined in \eqref{eq:defhatpi} in \emph{not canonical} as we made a choice of representatives of the classes;  nevertheless this will be good enough for our purpose of understanding  equivariant Wasserstein geometry.
 It follows directly from the definitions \eqref{eq:defmux}, \eqref{eq:dedpixy} and \eqref{eq:defhatpi} that 
\begin{align*}
\hat{\pi}(B\times M)&= \int_{M^{*}\times M^{*}}\pi_{\bar{x}_{(x^{*},y^{*})}, \bar{y}_{(x^{*},y^{*})} }  (B\times M) \,d\pi(x^{*},y^{*})  = \int_{M^{*}\times M^{*}}\nu_{\bar{x}_{(x^{*},y^{*})}}(B)\, d\pi(x^{*},y^{*}) \\
&= \int_{M^{*}\times M^{*}}\nu_{x}(B)\, d\pi(x^{*},y^{*})=  \int_{M^{*}}\nu_{x}(B)\, d\mu_{0}(x^{*}) =\hat{\mu}_{0}(B), 
\end{align*}
for any Borel subset $B\subset M$, where, in the third identity, we used that $\nu_{\bar{x}_{(x^{*},y^{*})}}=\nu_{x}$ since by construction $\quotient (\bar{x}_{(x^{*},y^{*})})= x^{*}$. Analogously, we have that $\hat{\pi}(M \times B)=\hat{\mu}_{1}(B)$. In other words, the lift $\hat{\pi}$ is an admissible coupling for $(\hat{\mu}_{0},\hat{\mu}_{1})$ and thus 
\begin{equation}\label{eq:defWpled}
W_{p}(\hat{\mu}_{0},\hat{\mu}_{1})^{p}  \le  \int \sfd(x,y)^{p}\, d\hat{\pi}(x,y) =  \int \sfd^{*}(x^{*},y^{*})^{p} \,d\pi(x^{*},y^{*}).
\end{equation}
The last inequality shows that
\begin{equation}\label{eq:WphatleqWp}
W_{p}(\hat{\mu}_{0},\hat{\mu}_{1})\le W_{p}(\mu_{0},\mu_{1}).
\end{equation}

In particular, by choosing $\mu_{1}$ (resp. $\mu_{0}$) to be a Dirac mass it follows that if $\mu_{0} \in {\mathcal P}_{p}(M^{*})$  (resp. $\mu_{1} \in {\mathcal P}_{p}(M^{*})$) then also $\hat{\mu}_{0}  \in {\mathcal P}_{p}(M)$
 (resp. $\hat{\mu}_{1} \in {\mathcal P}_{p}(M)$), indeed:
 \begin{align*}
 \left(\int_{M} \sfd(x,x_{1}) ^{p} \, d\hat{\mu}_{0} \right)^{\frac{1}{p}}=W_{p}(\hat{\mu}_{0},\delta_{x_{1}}) \leq W_{p}(\hat{\mu}_{0},\nu_{x_{1}})+W_{p}(\nu_{x_{1}}, \delta_{x_{1}}) \leq W_{p}(\mu_{0},\delta_{x_{1}^*})+ \sup_{g\in \G} \sfd(gx_{1}, x_{1})<\infty,
 \end{align*}
 where in last inequality we used that the $\G$-orbits are bounded since by assumption $\G$ is compact.

Now let $(\phi,\psi)$ be a $p$-dual solution corresponding to  a $p$-optimal coupling $\pi$, i.e.
\begin{align*}
\phi(x^{*})+\psi(y^{*})\le & \,\sfd^{*}(x^{*},y^{*})^{p} \quad \text{ for all } x^{*},y^{*}\in M^{*}, \\ 
\phi(x^{*})+\psi(y^{*})= &\, \sfd^{*}(x^{*},y^{*})^{p} \quad \text{ for $\pi$-a.e. }(x^{*},y^{*});
\end{align*}
in particular 
\begin{equation} \label{eq:OptDual1} 
W_{p}(\mu_{0},\mu_{1})^{p}= \int_{M^{*}} \phi \, d\mu_{0} +  \int_{M^{*}} \psi \, d\mu_{1}.
\end{equation}
By defining the lifts $\hat{\phi}(x):=\phi(x^{*})$, and $\hat{\psi}(y):=\psi(y^{*})$, we get  that 
\begin{equation*}
\hat{\phi}(x)+\hat{\psi}(y)  =  \phi(x^{*})+\psi(y^{*}) \le  \sfd^{*}(x^{*},y^{*})^{p} =  \inf_{g\in\G} \sfd(gx,y)^{p} \le  \sfd(x,y)^{p},  \quad \text{ for all } x,y\in M,
\end{equation*}
that is to say, the couple $(\hat{\phi}, \hat{\psi})$ is $p$-admissible in $M$.
Therefore,  we have that
\begin{align}\label{eq:WpleWphat}
\begin{split}
W_{p}(\mu_{0},\mu_{1})^{p} & =  \int_{M^{*}}\phi(x^{*})\,d\mu_{0}(x^{*})+\int_{M^{*}}\psi(y^{*}) \, d\mu_{1}(y^{*}) =  \int_{M}\hat{\phi}(x) \, d\hat{\mu}_{0}(x)+\int_{M}\hat{\psi}(y) \, d\hat{\mu}_{1}(y)\\
 & \le  W_{p}(\hat{\mu}_{0},\hat{\mu}_{1})^{p}. 
\end{split}
\end{align}

The combination of  \eqref{eq:WphatleqWp} and  \eqref{eq:WpleWphat} then yields the claimed identity $W_{p}(\hat{\mu}_{0},\hat{\mu}_{1})=W_{p}(\mu_{0},\mu_{1})$.
Moreover, we conclude from \eqref{eq:defWpled} that $\hat{\pi}$ is a $p$-optimal coupling for $(\hat{\mu}_{0}, \hat{\mu}_{1})$ if $\pi$ is $p$-optimal for $({\mu}_{0}, {\mu}_{1})$ and that $\sfd(x,y)=\sfd^{*}(x^{*},y^{*})$ for $\hat{\pi}$-a.e. $(x,y)\in M\times M$ since, by construction, $\hat{\pi}$ is concentrated on $\mathcal{OD}$. Additionally, \eqref{eq:WpleWphat} grants that the lifted pair $(\hat{\phi}, \hat{\psi})$ is $p$-optimal for $(\hat{\mu}_{0}, \hat{\mu}_{1})$ if $(\phi, \psi)$ is $p$-optimal for $({\mu}_{0}, {\mu}_{1})$. In particular, $\hat{\phi}(x)+\hat{\psi}(y)=\sfd(x,y)^{p}$ for $\hat{\pi}$-a.e. $(x,y)\in M\times M$. 
\end{proof}

% REMARK

\begin{remark}
A result in the direction of Theorem~\ref{thm:LiftIsom} had already been shown by Lott and Villani in \cite[Lemma 5.36]{LottVillani}. In comparison to their work, we are more explicit in the construction of lifts of measures and optimal plans. This allows us additionally to show that the natural lifts of dual solutions are dual solutions as well. 
 This information will be useful in the proof of Theorem~\ref{thm:CDbyGisCD}.
\end{remark} 

As we just mentioned, we have shown that optimal plans and dual solutions are preserved by lifts.

% COROLLARY

\begin{cor}\label{co:liftoppl}
Let $\mu_0,\mu_1\in\mathcal P_p(M^*)$. Then for  every $p$-optimal coupling $\pi\in \mathcal{P}(M^{*}\times M^{*})$ of $(\mu_{0},\mu_{1})$ the (non-canonical) lifted coupling $\hat{\pi} \in \mathcal{P}(M\times M)$ defined in (\ref{eq:defhatpi}) is an optimal coupling of $(\hat{\mu}_{0},\hat{\mu}_{1})$ for which $\sfd(x,y)=\sfd^{*}(x^{*},y^{*})$ holds for $\hat{\pi}$-almost every $(x,y)\in M\times M$. Furthermore, whenever $(\phi,\psi)$ is a $p$-dual solution corresponding
to $\pi$ then the lift $(\hat{\phi},\hat{\psi})$ is a $p$-dual solution corresponding
to $\hat{\pi}$.
\end{cor}

The following corollary also follows.

% COROLLARY

\begin{cor}\label{lem:ProjEssNB} 
If $(M,\sfd,\m)$ is essentially non-branching then the quotient space $(M^{*},\sfd^{*},\m^{*})$ is essentially non-branching.
\end{cor}
\begin{proof}
Note that a branching geodesic in $M^{*}$ lifts to a $\G$-invariant
family of branching geodesics in $M$. Moreover, since  absolutely continuous
measures are lifted to absolutely continuous measures and any optimal dynamical
coupling on $M^{*}$ lifts to an optimal dynamical coupling on $M$, we
see that any $\gamma\in\OptGeo(\mu_{0},\mu_{1})$ between $\mu_{i}\in\mathcal{P}_{2}^{ac}(M^{*})$, $i=0,1$,
must be concentrated on a set of non-branching geodesics.
\end{proof}

We also get the following result of independent interest.

% COROLLARY - push-forward induced by group action is a submetry

\begin{cor}
The push-forward $\quotient_{\sharp}:\mathcal{P}_{p}(M)\to\mathcal{P}_{p}(M^{*})$
is a submetry (see Definition \ref{def:submetry}).
\end{cor}

% PROOF

\begin{proof}
We know already that $\quotient_{\sharp}$ is onto. To see that $\quotient_{\sharp}$
is $1$-Lipschitz note that if $\pi\in \mathcal{P}(M\times M)$ is an admissible  coupling for $(\mu,\nu)$, then $\pi^{*}=(\quotient\times\quotient)_{\sharp}\pi$ is admissible for $(\quotient_{\sharp}\mu,\quotient_{\sharp}\nu)$
and 
\[
\int \sfd(x,y)^{p}d\pi(x,y)\ge\int \sfd^{*}(x^{*},y^{*})^{p}d\pi^{*}(x^{*},y^{*}).
\]
Let $\mu\in\mathcal{P}_{p}(M)$ and $\mu_{\G}\in\mathcal{P}_{p}(M)$ be its $\G$-average,
\begin{equation}\label{eq:defmuG}
\mu_{\G}=\int_{\G} (\tau_{g})_{\sharp}\mu \; d\nu_{\G}(g).
\end{equation}
Then 
\[
\mu^{*}:=\quotient_{\sharp}\mu=\quotient_{\sharp}\mu_{\G}.
\]
Fix an arbitrary $\nu^{*}\in\mathcal{P}_{p}(M^{*})$ with $W_{p}(\mu^{*},\nu^{*})=r$
and denote its lift by $\nu_{\G}$. Then, from Theorem \ref{thm:LiftIsom}, we have $W_{p}(\mu_{\G},\nu_{\G})=r$. Thanks to  Corollary \ref{co:liftoppl} there exists  a $p$-optimal coupling  $\pi$ for $(\mu_{\G},\nu_{\G})$ which is supported on the set $\mathcal{OD}$, defined on  \eqref{eq:defOD}.   
Via the Disintegration Theorem, 
we obtain that there exists a $\mu_{\G}$-a.e. well-defined measure-valued assignment $x\mapsto \tilde{\nu}_{x}$ such  that for every Borel set $B\subset M$ the  function $x\mapsto \tilde{\nu}_{x}(B) \in \R\cup\{+\infty\}$ is $\mu_{\G}$-measurable and
\[
\pi=\int\delta_{x}\times\tilde{\nu}_{x} \; d\mu_{\G}(x).
\]
From the construction of $\pi$ performed in Corollary \ref{co:liftoppl}  (see in particular \eqref{eq:defhatpi}, noting that here the lift is denoted by $\pi$ instead of $\hat{\pi}$), one can check that $\delta_{x}\times\tilde{\nu}_{x}$ is  supported on $\mathcal{OD}$ for  $\mu_{\G}$-a.e. $x$. 

Assume now that $\mu\ll\mu_{\G}$. Then we can define a new coupling
\begin{equation}\label{eq:deftildepi}
\tilde{\pi}=\int\delta_{x}\times\tilde{\nu}_{x} \; d\mu(x).
\end{equation}
Set $\nu=(p_{2})_{\sharp}\tilde{\pi}$ and observe that $\nu^{*}=\quotient_{\sharp}\nu_{\G}$; therefore $\tilde{\pi}$ is an admissible coupling for $(\mu,\nu)$ and, recalling that $\pi$ was $p$-optimal for $(\mu_{\G},\nu_{\G})$, we get  
\begin{align*}
W_{p}(\mu,\nu)^{p}&\leq \int_{M\times M} \sfd(x,y)^{p}\; d\tilde{\pi}(x,y)=  \int_{M} \int_{M} \sfd(x,y)^{p} \;d{\tilde{\nu}}_{x}(y) \, d\mu(x) =  \int_{M} \int_{M} \sfd^{*}(x^{*},y^{*})^{p} \;d{(\quotient_{\sharp}\tilde{\nu}}_{x})(y) \, d\mu^{*}(x^{*}) \\
&= \int_{M} \int_{M} \sfd(x,y)^{p} \;d{\tilde{\nu}}_{x}(y) \, d\mu_{\G}(x) = \int \sfd^{p}(x,y) \; d\pi(x,y)=W_{p}(\mu_{\G},\nu_{\G})^{p}= W_{p}(\mu^{*},\nu^{*})^{p}=r^{p}.
\end{align*}
Since we already proved that $\quotient_{\sharp}$ is $1$-Lipschitz, the last inequality  implies that  $W_{p}(\mu,\nu)=r$. Note that in particular we also get that $\tilde{\pi}$ is a  $p$-optimal coupling for $(\mu,\nu)$. 
Because $\nu^{*}$ was an arbitrary measure at $W_{p}$-distance $r$ from $\mu^{*}$,  we infer  that $\quotient_{\sharp}(B_{r}^{W_{p}}(\mu))=B_{r}^{W_{p}}(\mu^{*}).$

The  case for general $\mu \in \mathcal{P}_{p}(M)$ follows by approximation, since for such a measure the set of $\tilde\mu\in\mathcal{P}_{p}(M)$
with $\tilde\mu\ll\mu_{\G}$ is dense in $\mathcal{P}_{p}(M)$. More precisely,
one can construct a sequence $(\mu_{n})_{n\in\mathbb{N}}$ in $\mathcal{P}_{p}(M)$
with $\mu_{n}\ll(\mu_{n})_{\G}=\mu_{\G}$ such that $\mu_{n}\to\mu$
in $(\mathcal{P}_{p}(M), W_{p})$. 
\\Let $\tilde{\pi}_{n}$ be the couplings from $\mu_{n}$ obtained as in \eqref{eq:deftildepi} and set $\nu_{n}=(p_{2})_{\sharp}\tilde{\pi}_{n}$.  
From the above arguments, we know that $\tilde{\pi}_{n}$ is a $p$-optimal coupling for $(\mu_{n},\nu_{n})$.
 Because the orbits are compact and $\nu^{*}=\quotient_{\sharp}\nu_{n}$, we
see that $\{\nu_{n}\}_{n\in\mathbb{N}}$ and $\{\tilde{\pi}_{n}\}_{n\in\mathbb{N}}$
are weakly compact.  Thus there exist $\nu\in  \mathcal{P}(M), \;\tilde{\pi}\in \mathcal{P}(M\times M)$ such that, up to subsequences, $\nu_{n}$
and $\tilde{\pi}_{n}$ converge weakly to $\nu$ and  $\tilde{\pi}$ respectively.
\\Then, by \cite[Theorem 5.20]{Villani},  $\tilde{\pi}$ is a $p$-optimal coupling for $(\mu,\nu)$ and it holds 
\[
W_{p}(\mu,\nu)\le\liminf_{n\to\infty}W_{p}(\mu_{n},\nu_{n})=r.
\]
However, since $\mu^{*}=\quotient_{\sharp}\mu$ and $\nu^{*}=\quotient_{\sharp}\nu$
this must be an equality proving the claim for all $\mu^{*} \in \mathcal{P}_{p}(M^{*})$.
\end{proof}

We now prove that the convexity of the entropy in Wasserstein spaces  is stable under quotients.

% THEOREM - Stability of CD by quotients
\begin{thm}\label{thm:CDbyGisCD}
Assume that the metric measure space $(M,\sfd,\m)$  satisfies the strong $\CD(K,N)$-condition (resp. the strong  $\CD^{*}(K,N)$-condition, or the strong $\CD(K,\infty)$-condition). Then the quotient metric measure space  $(M^{*},\sfd^{*},\m^{*})$ satisfies
 the strong $\CD(K,N)$-condition (resp. the strong  $\CD^{*}(K,N)$-condition,  or the strong $\CD(K,\infty)$-condition). 
\end{thm}

% PROOF

\begin{proof}
We prove the result for the  strong $\CD(K,N)$-condition, $K\in \R, N<\infty$; the proofs for strong $\CD^{*}(K,N)$ or  $\CD(K,\infty)$ are completely analogous.

Let $\mu_{0},\mu_{1}\in\mathcal{P}_{2}^{ac}(M^{*})$ and let  $\{\mu_{t}\}_{t \in [0,1]}$ be a $W_{2}$-geodesic between them inducing the $2$-optimal coupling $\pi$. 
For every $t \in [0,1]$,  write $\hat{\mu}_{t}$ for the lift of $\mu_{t}$ given by  \eqref{eq:liftmeas}.
From Theorem \ref{thm:LiftIsom} we know that $\mu_{0},\mu_{1}\in\mathcal{P}_{2}^{ac}(M)$ and that   $\{\hat{\mu}_{t}\}_{t \in [0,1]}$ is a $\G$-invariant $W_{2}$-geodesic in $\mathcal{P}_{2}(M)$. By the strong $\CD(K,N)$-condition, we have $\hat{\mu}_{t}=\hat{\rho}_{t} \m \in  \mathcal{P}_{2}^{ac}(M)$ with
 \begin{equation*}
\int_{M}\hat{\rho}_{t}^{1-\frac{1}{N'}}\, d\m\ge\int_{M\times M}\left[\tau_{K,N'}^{(1-t)}(\sfd(x,y))\hat{\rho}_{0}^{-\frac{1}{N'}}(x)+\tau_{K,N'}^{(t)}(\sfd(x,y))\hat{\rho}_{1}^{-\frac{1}{N'}}(y)\right]d\hat{\pi}(x,y),
\end{equation*}
for every $t \in [0,1]$ and every $N'\geq N$, where $\hat{\pi}$ is the lift of $\pi$ defined in \eqref{eq:defhatpi} or, equivalently, the $2$-optimal coupling from $\hat{\mu}_{0}$ to $\hat{\mu}_{1}$ induced by the geodesic $\{\hat{\mu}_{t}\}_{t \in [0,1]}$.

Since the measure  $\hat{\mu}_{t}$ is $\G$-invariant, its density $\hat{\rho}_{t}$ is also $\G$-invariant. Thus, $\rho_{t}(x^{*}):=\hat{\rho}_{t}(x)$ is well defined and 
it coincides with the density of $\mu_{t}$ with respect to $\m^{*}$. Therefore, it holds that
\begin{align*}
\int_{M^{*}}&\rho_{t}^{1-\frac{1}{N'}}(x^{*}) \, d\m^{*}(x^{*})  =  \int_{M^{*}}\int_{\G(x)}\hat{\rho}_{t}^{1-\frac{1}{N'}}(y) \, d\m_{x^{*}}(y) \, d\m^{*}(x^{*})=\int_{M}\hat{\rho}_{t}^{1-\frac{1}{N'}}(y)\, d\m(y) \\
&  \ge  \int_{M\times M}\left[\tau_{K,N'}^{(1-t)}(\sfd(x,y))\hat{\rho}_{0}^{-\frac{1}{N'}}(x)+\tau_{K,N'}^{(t)}(\sfd(x,y))\hat{\rho}_{1}^{-\frac{1}{N'}}(y)\right]d\hat{\pi}(x,y)\\
 & =  \int_{M^{*}\times M^{*}}\left[\tau_{K,N'}^{(1-t)}(\sfd^{*}(x^{*},y^{*}))\rho_{0}^{-\frac{1}{N'}}(x^{*})+\tau_{K,N'}^{(t)}(\sfd^{*}(x^{*},y^{*}))\rho_{1}^{-\frac{1}{N'}}(y^{*})\right]d\pi(x^{*},y^{*}),
\end{align*}
for every $t \in [0,1]$ and $N'\geq N$, where we have used that $\rho_{i}(x^{*})=\hat{\rho}_{i}(x)$ and
$\sfd(x,y)=\sfd^{*}(x^{*},y^{*})$ for $\hat{\pi}$-almost all $(x,y)\in M\times M$. 
\end{proof}

% REMARK

\begin{remark}
As recalled  in Remark  \ref{rem:EssNBsCD}, for $N<\infty$ the strong $\CD(K,N)$-condition, $\CD^{*}(K,N)$-condition respectively, is equivalent to essentially non-branching plus (weak) $\CD(K,N)$, essentially non-branching plus (weak) $\CD^*(K,N)$.
Furthermore, essentially non-branching plus $\MCP(K,N)$ imply uniqueness of  the $W_{2}$-geodesic  starting from $\mu_{0}\in \mathcal{P}_{2}^{ac}(M,\sfd, \m)$ to an arbitrary measure $\mu_{1}\in \mathcal{P}_{2}(M)$. Therefore, the verbatim argument above together with Lemma \ref{lem:ProjEssNB} proves that if $(M,\sfd, \m)$ is an essentially non-branching $\MCP(K,N)$-space then also $(M^{*},\sfd^{*}, \m^{*})$ satisfies essentially non-branching plus $\MCP(K,N)$. 

As a matter of fact, note that the key property needed for the proof is that the curvature-dimension condition under consideration is satisfied along $\G$-invariant Wasserstein geodesics.
\end{remark}

% REMARK

\begin{remark}
The same arguments show that the strong $\CD_{p}(K,N)$-condition defined by the second
author in \cite{KellCDp} holds for quotients of strong $\CD_{p}(K,N)$-spaces. 
Additionally, granted that the group $\G$ is finite, also the intermediate $p$-Ricci lower curvature bounds in terms of optimal transport introduced by Ketterer and the third author \cite{KetMon} are preserved under quotients.
\end{remark}

% SECTION: PRINCIPAL ORBIT THEOREM AND COHOM 1 ACTIONS
%----------------------------------------------------------------------------------

\section{Principal orbit theorem and cohomogeneity one actions} 
We prove a principal orbit theorem in the context of metric measure spaces and use it to study cohomogenity one actions for spaces with nice optimal transport properties. The theorem  states that the orbits of $\m$-almost all points in $M$  are of a unique maximal type. For this result, we are inspired by  a paper of  Guijarro and the first author   \cite{GGG2013} in the framework of  Alexandrov spaces. In order to compensate the lack of information about branching of geodesics in the present context, we introduce the notion of   \emph{good optimal transport behavior} requiring, roughly speaking, that each optimal transport from an absolutely continuous measure is induced by an optimal map (see Definition \ref{def:GTB}). Such a notion is known to be implied by  lower Ricci curvature conditions such as $\CD^{*}(K,N)$ and, more generally, $\MCP(K,N)$ coupled with the  essentially non-branching property (see  Gigli-Rajala-Sturm \cite{GRS2016} and
Cavalletti-Mondino \cite{CMMapsNB}).  Thanks to the results proved in the previous section,  we will easily show that the good optimal transport behavior  is stable under quotients.

As mentioned above, it will be convenient in the sequel to consider metric measure spaces satisfying the following property which is closely related to  $p$-essential non-branching. 

% DEF

\begin{defn}[Good transport behavior]\label{def:GTB}
 A metric measure space $(M,\sfd,\m)$ has
\emph{good transport behavior} $\GTB_{p}$ for $p\in(1,\infty)$ if, for
all $\mu,\nu\in\mathcal{P}_{p}(M)$ with $\mu\ll\m$,  any $p$-optimal coupling from $\mu$ to $\nu$ is induced by a map.
\end{defn}

The good transport behavior is satisfied by a large class of examples, as recalled in the next theorem.

% THM: SPACES SATISFYING GOOD TRANSPORT BEHAVIOR

\begin{thm}[Cavalletti-Huesmann \cite{CavallettiHuesmann}, Cavalletti-Mondino\cite{CMMapsNB}, Gigli-Rajala-Sturm\cite{GRS2016}] \label{thm:ExGTB}
 The
following spaces satisfy $\GTB_{p}$: 
\begin{itemize}
\item Essentially non-branching $\mathsf{MCP}(K,N)$-spaces for $p=2$,
$K\in\mathbb{R}$, and $N\in[1,\infty)$. In particular, this includes 
essentially non-branching $\mathsf{CD}^{*}(K,N)$-spaces, essentially
non-branching $\text{\ensuremath{\mathsf{CD}}}(K,N)$-spaces, and
$\mathsf{RCD^{*}(K,N)}$-spaces.
\item Non-branching spaces with weak quantitative $\MCP$-property for all
$p\in(1,\infty)$.
\end{itemize}
\end{thm}

Using the fact that the geometry of the $p$-Wasserstein space is preserved  under quotients we can also show the next result. 

% LEMMA

\begin{lem}
Let $(X,\sfd,\m)$ be a m.m. space satisfying $\GTB_{p}$. Then $(X^{*},\sfd^{*},\m^{*})$
satiesfies $\GTB_{p}$ as well. 
\end{lem}

% PROOF

\begin{proof}  This is a consequence of  Theorem \ref{thm:LiftIsom} and  Corollary \ref{co:liftoppl}.  Indeed for any $\mu^{*},\nu^{*}\in  {\mathcal P}_{p}(M^{*})$ with $\mu^{*}\ll \mm^{*}$, any $p$-optimal coupling $\pi \in \mathcal P(M^*\times M^*)$ for $(\mu^{*},\nu^{*})$ can be written as $(\quotient\times\quotient)_\sharp \hat{\pi}$, where $\hat{\pi}$ is the lift defined by \eqref{eq:defhatpi}.
The lifts of $\mu^{*}, \nu^{*}$ given by Theorem \ref{thm:LiftIsom}, which we write as $\mu=\Lambda(\mu^{*})$ and $\nu=\Lambda(\nu^{*})$, satisfy the following: $\mu \ll \mm$ and the lifted coupling $\hat{\pi}$ is a $p$-optimal for $(\mu,\nu)$. Since 
$(M,\sfd,\mm)$ satisfies $\GTB_{p}$, the coupling  $\hat{\pi}$ is induced by a map $T$. Moreover, by the equivariant structure of $\hat{\pi}$ stated in Corollary \ref{co:liftoppl} it is easy to check that  one can choose $T$ to be  $\G$-equivariant, i.e. $T(x)=T(g x)$ for all $g \in \G$, for $\mu$-a.e. $x$. Therefore $T$ passes to the quotient, giving an optimal map which induces the coupling $\pi$.
\end{proof}
Metric measure spaces that have good transport behavior enjoy nice geodesic properties. For example, by picking a Dirac mass as the terminal measure in Definition \ref{def:GTB}, it is possible  to see that  for every $x\in M$  and $\m$-a.e. $y\in M$ there exist a unique geodesic joining $x$ to $y$. With the following two lemmas we show that actually a stronger statement is true. 

% LEMMA 

\begin{lem}\label{lem:cpmonGammast} Let $\Gamma\subset M\times M$ be a $c_{p}$-cyclically monotone set.
Then for any $s,t\in[0,1]$ the set 
\[
\Gamma_{s,t} = (\ee_s,\ee_t)\left((\ee_0,\ee_1)^{-1}\left(\Gamma\right)\right)
\] 
is $c_{p}$-cyclically monotone.
\end{lem}
\begin{proof} 
Choose $(x_{s}^{(i)},x_{t}^{(i)})\in\Gamma_{s,t}$, $i=1,\ldots,n$. By assumption  there are geodesics $\gamma^{(i)}$, $i=1,\ldots,n$, such that 
$(\gamma_{s}^{(i)},\gamma_{t}^{(i)})=(x_{s}^{(i)},x_{t}^{(i)})$ and 
\[
\bigcup_{i=1}^{n}\{(\gamma_{0}^{(i)},\gamma_{1}^{(i)})\}\subset\Gamma.
\]
Since by assumption $\Gamma$ is $c_{p}$-cyclically monotone, we have that
\[
\sigma=\frac{1}{n}\sum\delta_{\gamma^{(i)}}
\]
is a $p$-optimal dynamical coupling. We conclude by observing that
\[
\bigcup_{i=1}^{n}\{(x_{s}^{(i)},x_{t}^{(i)})\}=\bigcup_{i=1}^{n}\{(\gamma_{s}^{(i)},\gamma_{t}^{(i)})\}=\supp\left[(\ee_{s},\ee_{t})_{\sharp}\sigma\right]
\]
is $c_{p}$-cyclically monotone.
\end{proof}

% LEMMA - GTB and c-superdifferential

\begin{lem}
\label{lem:GTB-sup}
Let $p \in (1, \infty)$.
A metric measure space $(M,\sfd,\m)$ satisfies $\GTB_{p}$
if and only if, for $\mm$-a.e. $x \in M$ and every $c_{p}$-concave function $\varphi$, the $c_p$-superdifferential $\partial^{c_{p}}\varphi(x)$
 contains  at most one point.

In particular, if $(M,\sfd,\m)$ satisfies $\GTB_{p}$ then for any $c_{p}$-concave
function $\varphi$ and $\m$-almost all $x\in M$ there exists a
unique geodesic connecting $x$ and $\partial^{c_{p}}\varphi(x)$,
whenever the set $\partial^{c_{p}}\varphi(x)$ is non-empty.
\end{lem}
\begin{proof} To get the ``only if''  implication recall that the graph of the $c_{p}$-superdifferential
of any $c_{p}$-concave function is $c_{p}$-cyclically monotone,
and that every coupling supported on a $c_{p}$-cyclically monotone
set is optimal (see for instance \cite[Theorem 1.13]{AGUser}).  If there exist a $c_{p}$-concave function $\varphi$ and  a (without loss of generality we assume bounded) set $E\subset M$ with $\mm(E)>0$ such that   $\partial^{c_{p}}\varphi(x)$ is not single valued for every $x \in E$, called $\mu=\mm(E)^{-1} \; \mm\llcorner E$ it is not difficult to construct a plan $\pi\in {\mathcal P}(M\times M)$ concentrated on $\partial^{c_{p}}\varphi$  (and therefore $p$-optimal) with $\mu$ as first marginal, which is not induced by a map; this clearly contradicts $\GTB_{p}$. For instance a possible construction of $\pi$ is as follows: by a standard measurable selection argument we can find two measurable maps $T_{1},T_{2}:E\to M$ $\mu$-a.e. well defined such that $T_{1}(x), T_{2}(x) \in \partial^{c_{p}}\varphi(x)$, $T_{1}(x)\neq T_{2}(x)$ for $\mu$-a.e. $x$; set  $\pi= \int_{M} \frac{1}{2} \left( \delta_{(x, T_{1}(x))}+ \delta_{(x, T_{2}(x))}\right) \, d\mu(x)$. 

The ``if'' implication is a classical argument in optimal
transport theory (see for instance \cite[Theorem 1.13]{AGUser}): let $\mu,\nu\in\mathcal{P}_{p}(M)$
with $\mu\ll\m$ and recall that for any $p$-optimal plan $\pi$ for $(\mu,\nu)$  there exists a  $c_{p}$-concave function $\varphi$ such that
$\mathrm{supp}(\pi)\subset\partial^{c_{p}}\varphi$.
The hypothesis that $\partial^{c_{p}}\varphi(x)$ is singled valued for $\m$-almost every
$x$ gives that $\pi$ is induced by a map, as desired.

To prove the last statement, one possibility is to  argue via Lemma \ref{lem:cpmonGammast}. However the next argument seems quicker.
 Suppose $(X,\sfd,\m)$ has $\GTB_{p}$
and note that if $\varphi$ is a $c_{p}$-concave function then $\varphi_{t}=t^{p-1}\varphi$
is $c_{p}$-concave as well, for all $t \in [0,1]$  (see for instance  \cite[Lemma 2.9]{KellCDp} and the discussion after  \cite[Definition 2.1]{KellCDp}; note also that the arguments in the proof of  \cite[Lemma 2.9]{KellCDp} do not require the compactness of $M$). Moreover, if
$\gamma$ is a geodesic connecting $x$ and $y\in\partial^{c_{p}}\varphi(x)$
then $\gamma_{t}\in\partial^{c_{p}}\varphi_{t}(x)$.

Let $D_{t}=\{x\in M\,|\,\partial^{c_{p}}\varphi_{t}(x)\ne\varnothing\}$
and observe that $D_{t}\subset D_{t'}$ whenever $0\le t'\le t\le1$.
Let $(t_{n})_{n\in\mathbb{N}}$ be dense in $(0,1]$ with $t_{1}=1$
and choose a measurable set $\Omega_{n}\subset D_{1}$ of full $\m$-measure
in $D_{1}$ such that $\partial^{c_{p}}\varphi_{t_{n}}(x)$ is single-valued
for all $x\in\Omega_{n}$. Then $\Omega=\cap_{n\in\mathbb{N}}\Omega_{n}$
also has full $\m$-measure in $D_{1}$. Let $\gamma$ and $\eta$
be two geodesics connecting $x\in\Omega$ and $y\in\partial^{c_{p}}\varphi(x)$.
If $\gamma$ and $\eta$ were distinct then there is an open interval
$I\subset(0,1)$ such that $\gamma_{s}\ne\eta_{s}$ for all $s\in I$.
In particular, there is an $n$ such that  $t_{n}\in I$. Hence
$\gamma_{t_{n}}\ne\eta_{t_{n}}$ and  $\partial^{c_{p}}\varphi_{t_{n}}(x)$ 
is not single-valued. However, this is a contradiction as $x\in\Omega\subset\Omega_{n}$ implies
that $\partial^{c_{p}}\varphi_{t_{n}}(x)$ is single-valued.
\end{proof}

% REMARK
 
\begin{remark}
On proper geodesic spaces, one has $\partial^{c_{p}}\varphi(x)\ne \varnothing$ for all 
\[
x \in \operatorname{int}(\{\varphi>-\infty\}),
\] 
where $\varphi$ is  a $c_p$-concave function (see \cite{GRS2016} for $p=2$ and   \cite{KellCDp} for general $p$).
\end{remark}

We can now prove that the good transport behavior is enough to show an $\m$-almost everywhere uniqueness of the orbit type.
% THEOREM 
\begin{thm}[Principal Orbit Theorem]\label{thm:PrincipleOrbit}Let $(M,\sfd,\m)$ be a m.m. space satisfying $\GTB_p$ for some $p \in (1,\infty)$. Fix $x^{*}\in M^{*}$, and denote the orbit over $x^*$ by $\G(x) = \quotient^{-1}(x^*)$. Then for $\m$-a.e. $y\in M$ there exists a unique $x_{y}\in\G(x)$ and a unique geodesic connecting $x_{y}$ and $y$ such that  
$
\sfd(x_{y},y)=\sfd^{*}(x^{*},y^{*}).
$
In particular,   $\G_{y}\leq \G_{x_y}$ for the isotropy groups of $y$ and $x_y$, 
and there exists a unique subgroup $\G_{\min}\le\G$ (up to conjugation) such that for $\m$-a.e. $y\in M$ the orbit $\G(y)$ is homeomorphic to the quotient $\nicefrac{\G}{\G}_{\min}$.  
\\We call $\nicefrac{\G}{\G}_{\min}$ the principal orbit of the action of $\G$  over $(M,\sfd,\m)$.
\end{thm}
\begin{proof}
Since by Lemma \ref{lem:c-concave-lifts}, the lifts of $c_p$-concave functions are also $c_p$-concave, we get that the function $
\varphi(y):=\sfd^{*}(x^{*},y^{*})^{p}=\sfd(\G(x),y)^{p}$
is $c_{p}$-concave. Recall that  the $c_p$-superdifferential of $\varphi$ at $y$ can be equivalently described as  (see for instance \cite[(1.3)]{AGUser}, \cite{Villani}):  $x' \in \partial^{c_p}\varphi(y)$ if and only if
\begin{equation}\label{eq:superdiffdpGx}
\sfd(\G(x),y)^{p}- \sfd(y,x')^{p}=\varphi(y) - \sfd(y,x')^{p}\geq \varphi(z) - \sfd(z,x')^{p}= \sfd(\G(x),z)^{p} - \sfd(z,x')^{p} \quad \text{for all } z\in M.
\end{equation}
Since by assumption $\G(x)$ is compact, there exists  a point $x_{y}\in \G(x)$ at minimal distance from $y$, i.e. $\sfd(x_{y},y)=\sfd(\G(x), y)$. Using the characterization  \eqref{eq:superdiffdpGx}   of $\partial^{c_p}\varphi(y)$ it is easy to check that $x_{y} \in \partial^{c_p}\varphi(y)$. Recalling Lemma \ref{lem:GTB-sup}  we infer that, for $\mm$-a.e. $y \in M$, such a point $x_{y}\in \G(x)$ of minimal distance is unique and moreover there is a unique geodesic from $y$ to $x_{y}$.

Observe now that $\partial^{c_p}\varphi(y)$ is $\G_{y}$-invariant, that is, if $g\in\G_{y}$ and $x'\in\partial^{c_{p}}\varphi(y)$, we have
\[
\sfd(gx',y)= \sfd(gx', gy)= \sfd(x',y).
\]
Therefore, the $\m$-a.e. uniqueness of the $c_p$-superdifferential implies that $\G_{y}\le\G_{x_{y}}$, for $\mm$-a.e. $y$. 

Finally, pick $x_0$ such that the isotropy subgroup $\G_{x_{0}}\leq \G$ is minimal among all isotropy groups. More precisely, we mean that $\G_{x_{0}}$ is such that for every $y\in M$ there exists $g\in \G$ such that the inclusion $\G_{x_{0}}\leq g \G_y g^{-1}$ holds. Therefore, for $\m$-a.e. $y\in M$, it holds that the isotropy groups $\G_{y}=g \G_{x_{0}} g^{-1}$ and $\G_{x_{0}}$ are conjugate.
Thus, since the action of $\G$ is effective and for every $x\in M $,  $\G / \G_{x}$ acts transitively on the orbit $\G(x)$ we conclude that the orbits 
\[
\G(x_{0})\simeq\nicefrac{\G}{\G_{x_{0}}}\simeq\nicefrac{\G}{\G_{y}}\simeq\G(y),
\]
are homeomorphic for $\m$-almost every $y\in M$.  
\end{proof}

\begin{remark}
It is not complicated to construct examples of m.m. spaces that do not have $\GTB_p$ for which there is not a principal orbit type. Moreover, examples of m.m. spaces that satisfy the $\MCP$-condition in which a finite group of measure-preserving isometries acts but do not have a principal orbit type have been constructed in \cite{KettererRajala} by Ketterer and Rajala, and generalized in \cite{Sosa2016} by the fourth author. We note that none of these spaces has good transport behavior, or equivalently in this case, the essential non-branching condition. With these examples in mind we make emphasis on the relevance of the $\GTB_p$ assumption in the theorem above. 
\end{remark}

The following two corollaries show that, in metric measure spaces with good transport behavior, cohomogeneity one actions satisfy  similar rigidity properties as 
in the setting of Alexandrov spaces, see \cite{GGS2009}. However, note that we can not provide a completely analogous characterization because of the more general nature of the spaces that we are dealing with. We let below $\mathrm{N}_\G(\mathsf{H})\leq \G$ denote the normalizer subgroup of $\mathsf{H}$ in $\mathsf{G}$, defined by  $\mathrm{N}_\G(\mathsf{H}):=\{g \in \G\,|\, g \mathsf{H}= \mathsf{H}g\}$.

% COROLLARY - bundle characterization if quotient is S^1

\begin{cor}\label{co:bundle-s1}
Let $(M,\sfd,\m)$ be a m.m. space satisfying $\GTB_p$ for some $p \in (1,\infty)$ and let $\G_{\min}$ and $\G/\G_{\min}$ denote a minimal isotropy group and the corresponding principal orbit, respectively. Furthermore, assume that $(M^{*},\sfd^{*})$ is homeomorphic to a circle $S^{1}$. Then $(M,\sfd)$ is homeomorphic to  a fiber bundle over $S^1$ with fiber the homogeneous space $\G/\G_{\min}$ and structure group $\mathrm{N}_\G(\G_{\min})/\G_{\min}$. In particular, $(M,\sfd)$ is a topological manifold.
\end{cor}

% PROOF

\begin{proof}
It is sufficient to show, for all $x\in M$, that the orbit of $x$ is homeomorphic to the principal orbit $\G(x)\simeq \G/\G_{\min}$ since in this situation, the lift of any short enough geodesic $\gamma_t^*\in \mathrm{Geo}(S^1)$ induces a local trivialization. Moreover, the statement on the structure group follows from the observation that, for every two points in $M$, the isotropy groups are conjugate subgroups. We argue below to conclude that there is a unique orbit type. 

Since by assumption $M^{*}$ is homeomorphic to $S^{1}$, for any point $x^{*}\in M^{*}$ we can find two points $y^*,z^* \in M^{*}$ that satisfy the following conditions: the orbits $\quotient^{-1}(y^*)\simeq \quotient^{-1}(z^*) \simeq \G / \G_{\min}$ are principal, and there exist $y\in\quotient^{-1}(y^*)$ and $z\in\quotient^{-1}(z^*)$ such that $y$ is the unique element in $\quotient^{-1}(y^*)$ satisfying  $\sfd(y, z)=\sfd(\quotient^{-1}(y^{*}),z)$ and there is a unique geodesic $\gamma$ from $z$ to $y$ which intersects the orbit $\quotient^{-1}(x^*)$; call $x$ the intersection  point of  such a  geodesic with $\quotient^{-1}(x^{*})$.
The existence of such a  quadruple $(y^*,z^*,y,z)$ follows from  the proof of the Principal Orbit Theorem and Lemma \ref{lem:GTB-sup}.  Since $y$ is the unique point in $\quotient^{-1}(y^*)$ such that $\sfd(y,z)=\sfd(\quotient^{-1}(y^*), z)$, we get   that  $y$ is also  the unique point in $\quotient^{-1}(y^*)$ satisfying $\sfd(y,x)=\sfd(\quotient^{-1}(y^*), x)$.   It follows that the isotropy group of $x$ satisfies $\G_x \leq  \G_y=\G_{\min}$, since otherwise the uniqueness of $y$ would be violated. However, this means that $\G_x$ is minimal, and thus the orbit $\quotient^{-1}(x^*)$ principal. Since the argument is true for every $x^*\in X^*$ the result follows.  
\end{proof}

Arguing in the same way to show uniqueness of the principal orbit, we can also show the following result. 

% COROLLARY - bundle characterization if quotient is interval

\begin{cor}\label{co:bundle-R}
Let $(M,\sfd,\m)$ be as in Corollary \ref{co:bundle-s1} and assume that   $(M^{*},\sfd^{*})$ is homeomorphic
to an interval $[a,b]\subset \R$. Then the open dense set $
\quotient^{-1}((a,b))\subset M$
has full $\m$-measure and is homeomorphic to $(a,b)\times (\G / \G_{\min})$. In case $(M^*,\sfd^{*})$ is homeomorphic to $\mathbb R$ we have that $(M,\sfd)$ is homeomorphic to the product ${\mathbb R} \times (\G / \G_{\min})$,   in particular  $(M,\sfd)$ is a topological manifold.
\end{cor}
  
% SECTION
%----------------------------------------------------------------------------------
\section{Sobolev functions and group actions}

% SUBSECTION
%----------------------------------------------------------------------------------
\subsection{Basic definitions and properties}

% SSSECTION: LIFTING FUNCTIONS

\subsubsection*{Lifting functions}

Let $\hat{f}:M\to\mathbb{R}$ be a $\G$-invariant function, i.e. $\hat{f}(x)=\hat{f} (gx)$ for all $x \in M, g \in \G$. Then there
is a well-defined function $f:M^{*}\to\mathbb{R}$ defined 
by 
\begin{equation}\label{eq:liftfG}
f(x^{*})=\hat{f}(x).
\end{equation}
Vice versa, if $f$ is a function on $M^{*}$ then $\hat{f}$
defined by the same formula is the unique $\G$-invariant \emph{lift} of $f$.
Furthermore, since $\G$ is compact we see that $\hat{f}$ is measurable/continuous
if and only if $f$ is measurable/continuous. 

% SSSECTION: LIPSCHITZ CONSTANT

\subsubsection*{Lipschitz constants}
Given a metric space $(M,\sfd)$, we denote with $\LIP(M,\sfd)$ the space of Lipschitz functions, i.e. the space of those functions $f:M\to \R$ such that
$$ \sup_{x,y\in M, x\neq y} \frac{|f(x)-f(y)|}{\sfd(x,y)}<\infty.$$
Recall that the left hand side is called the \emph{Lipschitz constant} of $f$.
\\In order to study Lipschitz functions, two fundamental quantities are the \emph{upper asymptotic Lipschitz constant}  $\Lip$ and the \emph{lower  asymptotic Lipschitz constant}  $\lip$, defined as follows:
given  a function $f\colon M\to\mathbb{R}$ on $M$,  define
\begin{align}\label{eq:defLiplip}
\Lip f(x)  =  \limsup_{r\to0}\sup_{y\in B_{r}(x)}\frac{|f(y)-f(x)|}{r},\qquad  \lip f(x)  =  \liminf_{r\to0}\sup_{y\in B_{r}(x)}\frac{|f(y)-f(x)|}{r}.
\end{align}
An easy observation is that $\Lip f(x),\lip f(x)\le L$ if $f$ is
Lipschitz continuous with Lipschitz constant at most $L$. The converse
is true if $(M,\sfd)$ is a geodesic space. We will write $\Lip^N$ and $\lip^N$ when we want to stress the dependence on the domain $N$ of the functions. 

The two asymptotic Lipschitz constants play a central role in the study of spaces
admitting a well-defined first order Taylor expansion of Lipschitz
functions at almost all points, the so-called \emph{Lipschitz differentiability spaces}. Such a first order Taylor expansion
implies a weak form of differentiability of Lipschitz functions and
was first proven to hold by Cheeger \cite{CheegerGAFA} on spaces satisfying a doubling and Poincar\'e
inequality. A weaker condition, called the \emph{$\Lip$-$\lip$-condition}
was later discovered by Keith \cite{Keith2004}. 

% DEFINITION - Lip-lip
\begin{defn}[$\Lip$-$\lip$-condition] A m.m.  space $(M,\sfd,\m)$
is said to satisfy the \emph{$\Lip$-$\lip$-condition} if there is a positive constant $C>0$
such that, for all Lipschitz functions $f:M\to\mathbb{R}$, 
\[
\Lip f(x)\le C\lip f(x) \quad \text{ for $\mm$-a.e. $x\in M$}.
\]
A space satisfying the $\Lip$-$\lip$-condition is called a \emph{differentiability
space.}
\end{defn}
Note that it was shown in \cite{Schioppa2016, CKS2016}
that if the $\Lip$-$\lip$-condition
holds for some $C>0$ then it holds for $C=1$.

Recall that, if $y\in B_{r}(x)$, then the definition of the quotient
metric $\sfd^{*}$ implies that $y^{*}\in B_{r}(x^{*})$. In particular, if $\hat{f}$
is $\G$-invariant then 
\begin{equation}\label{eq:supfGf*}
\sup_{y\in B_{r}(x)}\frac{|\hat{f}(y)-\hat{f}(x)|}{r}=\sup_{y^{*}\in B_{r}(x^{*})}\frac{|f(y^{*})-f(x^{*})|}{r}.
\end{equation}
 Furthermore, if $\hat{f}$ is $\G$-invariant then also  $\Lip \hat{f}$
and $\lip \hat{f}$ are $\G$-invariant. Thus we can show that the lift of functions preserves the asymptotic Lipschitz constants.
% PROPOSITION - Lip and lip for equivariant lifts
\begin{prop}\label{prop:LipfGf*G}
For all functions $f:M^{*}\to\mathbb{R}$ and all $x\in M$
it holds 
\begin{equation}\label{eq:LipfGf*G}
\Lip \hat{f}(x)  =  \Lip f(x^{*})\quad \text{and} \quad \lip \hat{f}(x)  =  \lip f(x^{*}).
\end{equation}
In particular, if $(M,\sfd,\mathbf{m})$ satisfies the $\Lip$-$\lip$-condition, then
so does $(M^{*},\sfd^{*},\m^{*})$.
\end{prop}
\begin{proof}
Let $\hat{f}:M\to \R$ be the $\G$-invariant lift defined in \eqref{eq:liftfG}. The identities  \eqref{eq:LipfGf*G} are a direct consequence of identities \eqref{eq:supfGf*} and the definitions \eqref{eq:defLiplip}.
Now, if  $(M,\sfd,\m)$ satisfies the $\Lip$-$\lip$-condition, then we can find a set of full $\mathbf{m}$-measure $A\subset M$ such that
\begin{equation}\label{eq:LiplipA}
\Lip \hat{f}(x)\le C\lip \hat{f}(x) \quad \text{ for all } x \in A.
\end{equation}
Since both $\hat{f}$ and $\m$ are $\G$-invariant, without loss of generality we can assume that also
$A$ is $\G$-invariant. Therefore $A^{*}:=\quotient(A)\subset M^{*}$ is a  a set of full $\mathbf{m}^{*}$-measure and, combining \eqref{eq:LipfGf*G} with \eqref{eq:LiplipA}, we get that
\begin{equation*}\label{eq:LiplipA}
\Lip f(x^{*})\le C\lip f(x^{*}) \quad \text{ for all } x \in A^{*}, \quad \m^{*}(M^{*}\setminus A^{*})=0,
\end{equation*}
as desired \end{proof}

\subsubsection*{Sobolev functions}
There are several ways to define Sobolev functions which, under  mild regularity assumptions on the m.m. space, turn out to be equivalent (see for instance \cite{CheegerGAFA,AmbrosioGigliSavareInvent,AmbrosioGigliSavareRev,KellSobolev,KellHeat}).
We present the following approach based on weak convergence.
 
Given  $q\in(1,\infty)$ and  a function $f\in L^{q}(\m)$, the  \emph{$q$-Cheeger energy} of $f$ is defined by 
 \[
\Ch_{q}(f)=\inf\left\{\liminf_{n\to \infty}\frac{1}{q}\int(\Lip f_{n})^{q}d\m\,\Big|\,f_{n}\in\LIP(M,\sfd),\,f_{n}\rightharpoonup  f\,\text{ weak convergence in }L^{q}(\m)\right\}.
\]
In the remainder, and only when working with more than one m.m. space, to stress the dependence on the space under consideration we write $\Ch_{q}^{N}$ for the $q$-Cheeger energy of functions on $(N,\sfd,\mathfrak n)$.

We denote by $D(\Ch_{q})$ the set of all functions $f\in L^{q}(\m)$
with $\Ch_{q}(f)<\infty$.  It easily follows from the definition that  $\Ch_{q}$ is convex
and lower semicontinuous with respect to weak convergence in $L^{q}(\mm)$. In particular,  it induces a complete norm defined by 
\[
\|f\|_{W^{1,q}}=\left(\|f\|_{L^{q}}^{q}+\Ch_{q}(f)\right)^{\frac{1}{q}}
\]
on the space
\[
W^{1,q}(M,\sfd,\m)=W^{1,q}(\m)=L^{q}(\m)\cap D(\Ch_{q}).
\]
We call the Banach space $(W^{1,q}(\m),\|\cdot\|_{W^{1,q}})$ the
\emph{$q$-th Sobolev space} of the metric measure space $(M,\sfd,\m)$. 

The definition of the $q$-Cheeger energy implies that for all $f\in D(\Ch_{q})$
there is a sequence of Lipschitz functions $\{f_{n}\}_{n \in \N}$ converging weakly in $L^{q}(\m)$
to $f$ such that 
\[
\Ch_{q}(f)= \lim_{n\to \infty}\frac{1}{q}\int(\Lip f_{n})^{q}d\m.
\]
Using Mazur's Lemma and the uniform convexity of the $L^{q}$-norm, one can extract a subsequence and a convex combination, still denoted with $f_{n}$, such that the sequence $\{\Lip f_{n}\}_{n \in \N}$
converges strongly in $L^{q}(\m)$ to some $L^{q}(\m)$-function which
we denote by $|\nabla f|_{q}$. It is possible to show that $|\nabla f|_{q}$ does not depend on the chosen sequence. The function $|\nabla f|_{q}$ is
called the \emph{$q$-minimal relaxed slope} of $f$. Note that while it is always true that  $|\nabla f|_{q}\leq \Lip f$, it is possible
that 
\[
|\nabla f|_{q}\ne\Lip f
\]
for some Lipschitz functions $f$. However, if $(M,\sfd,\m)$ satisfies
a doubling and a Poincar\'e condition then from the work of Cheeger \cite{CheegerGAFA} we know that, for locally Lipschitz functions,   the $q$-minimal relaxed slope agrees $\mm$-a.e.  with the local Lipschitz constant of $f$.

Let us also recall that, for a general m.m. space, the Sobolev space $W^{1,2}(\mm)$ is a \emph{Banach} space; in case $W^{1,2}(\mm)$ is a Hilbert space, the m.m. space $(M,\sfd,\mm)$ is said to be \emph{infinitesimally Hilbertian} (see \cite{AmbrosioGigliSavareInvent, Ambrosio-Gigli-Savare11b, gigli:laplacian}). The infinitesimally Hilbertian condition is equivalent to the validity of the parallelogram rule for the $L^{2}$ Cheeger energy, i.e. 
\begin{equation}\label{eq:paralrule}
\Ch_{2}(f+g)+\Ch_{2}(f-g)= 2 \Big(\Ch_{2}(f)+\Ch_{2}(g)\Big), \quad \forall f,g \in W^{1,2}(\mm).
\end{equation}
 
To end this introductory section we present the next lemma, which easily follows from the definitions.
% LEMMA - Lq spaces for equivariant lifts
\begin{lem}\label{lem:Lqm*Lqm}
Let $q\in [1,\infty]$.
The natural lift \eqref{eq:liftfG} of functions defined on $M^{*}$
induces an isometric embedding of $L^{q}(\m^{*})$ into $L^{q}(\m)$,
whose image coincides with  the convex closed subset of $\G$-invariant functions in $L^{q}(\m)$.
\end{lem}

The rest of the section is devoted to showing that the Sobolev space $W^{1,q}(\m^{*})$ on the quotient $(M^{*},\sfd^{*},\m^{*})$ is isomorphic to the closed subspace of $\G$-invariant $W^{1,q}(\m)$-Sobolev functions on $M$.  We have divided the analysis into two cases aiming for a more accessible presentation. Readers interested in geometric applications in finite dimensions might find sufficient the results of the first subsection. These cover the case of doubling and Poincar\'e spaces, as shown in \cite{CheegerGAFA}, and in particular this includes  $\CD^{*}(K,N)$-spaces, $N\in[1,\infty)$ (see \cite{BS10, R2011}). The second subsection focuses on a more general class of m.m. spaces which includes infinite dimensional spaces.

% SUBSECTION - Lip and weak gradient agree
%----------------------------------------------------------------------------------
\subsection[$\Lip f = |\nabla f|_{q} $]{Initial case: $\Lip f = |\nabla f|_{q}$ $\m$-a.e.} In this subsection we consider the simpler case when upper asymptotic Lipschitz constants $\Lip f$  and minimal relaxed slopes $|\nabla f|_{q}$ coincide $\m$-almost everywhere for every locally Lipschitz function in $W^{1,q}(\m)$
Specifically, for $q\in (1,\infty)$, we make the assumption that on $(M,\sfd,\m)$ 
\begin{equation}\label{eq:lip=minslope}
\Lip^{M}f(x)=|\nabla^{M}f|_{q}(x)\quad \mm\text{-a.e. } x\in M, \;\text{for every locally Lipschitz } f\in W^{1,q}(\m).
\end{equation}

We first prove the following result. 
% PROPOSITION - condition lip = weak grad is preserved under quotients

\begin{prop}\label{prop:UGALSQuot}
Let $(M,\sfd,\m)$ be a m.m. space satisfying \eqref{eq:lip=minslope} for some $q\in (1,\infty)$.
Then on $(M^{*},\sfd^{*},\m^{*})$ the upper asymptotic Lipschitz constant and the
$q$-minimal relaxed slope agree $\mm^{*}$-a.e. for every local Lipschitz function in $W^{1,q}(\m^{*})$. 
\end{prop}

% PROOF

\begin{proof}
Note that the Cheeger energies $\Ch_{q}^{M}$ and $\Ch_{q}^{M^{*}}$
are the lower semicontinuous relaxation in the weak $L^{q}(\mm)$ topology, and in the weak $L^{q}(\mm^{*})$ topology respectively, of the functionals:  
\begin{align*}
L_{q}^{M}: L^{q}(\mm) \cap \LIP(M,\sfd) \ni \hat f & \mapsto\frac{1}{q}\int(\Lip^{M}\hat f)^{q}d\m \quad \text{and}\\
L_{q}^{M^{*}}:  L^{q}(\mm^{*}) \cap \LIP(M^{*},\sfd^{*}) \ni f & \mapsto\frac{1}{q}\int(\Lip^{M^{*}}f)^{q}d\m^{*}.
\end{align*}

Since for every $f \in \LIP(M^{*},\sfd^{*})$ it holds that $|\nabla f|_{q} \leq \Lip f$ $\mm^{*}$-a.e., in order to get the thesis  it is enough to show that
 \begin{equation}\label{eq:LqleqChq}
 L_{q}^{M^{*}}(f) \leq  \Ch_{q}^{M^{*}}(f), \quad 	\text{ for all } f\in  L^{q}(\mm^{*})\cap \LIP(M^{*},\sfd^{*}) \; \text{ satisfying }  L_{q}^{M^{*}}(f)<\infty.
 \end{equation}  
To this aim,  fix an arbitrary $f$ as in \eqref{eq:LqleqChq}  and let $\{f_{n}\}_{n \in \N} \subset  L^{q}(\mm^{*})\cap \LIP(M^{*},\sfd^{*})$ be such that
\begin{equation} \label{eq:ChlimLq}
\Ch_{q}^{M^{*}}(f)=\lim_{n\to \infty}  L_{q}^{M^{*}}(f_{n}), \quad f_{n}  \rightharpoonup  f \text{ weakly in $L^{q}(\mm^{*})$}.
\end{equation}

The fact that the upper asymptotic Lipschitz constant is preserved by the lift, Proposition  \ref{prop:LipfGf*G}, and the assumption that this constant coincides with the $q$-minimal relaxed slope on $M$ assure that $L_{q}^{M^{*}}(f_{n})= L_{q}^{M}(\hat{f}_{n})=\Ch_{q}^{M}(\hat{f}_{n})$ and $L_{q}^{M^{*}}(f)= L_{q}^{M}(\hat{f})=\Ch_{q}^{M}(\hat{f})$. 
By Lemma \ref{lem:Lqm*Lqm} we know that lifting is an isometry so we obtain that $\hat{f}_{n}\rightharpoonup  \hat{f}$  weakly in $L^{q}(\mm)$ and therefore, by the lower semicontinuity of  $\Ch_{q}^{M}$, we infer that
\begin{equation}\label{eq:LqfLqfn}
L_{q}^{M^{*}}(f)= L_{q}^{M}(\hat{f})=\Ch_{q}^{M}(\hat{f}) \leq \liminf_{n\to \infty}   \Ch_{q}^{M}(\hat{f}_{n})  =  \liminf_{n\to \infty}   L_{q}^{M}(\hat{f}_{n}) =  \liminf_{n\to \infty} L_{q}^{M^{*}}(f_{n}).
\end{equation}

The combination of \eqref{eq:ChlimLq} and  \eqref{eq:LqfLqfn} gives the desired inequality \eqref{eq:LqleqChq}.
\end{proof}

% COROLLARY - equivariant lifting induces isometric embedding

\begin{cor}\label{cor:W1qmw1qm*}
Let $(M,\sfd,\m)$ be a m.m. space satisfying \eqref{eq:lip=minslope} for some $q\in (1,\infty)$.
Then the space $W^{1,q}(\m^{*})\cap\LIP(M^{*},\sfd^{*})$ and the subspace of $\G$-invariant
functions in $W^{1,q}(\m)\cap\LIP(M,\sfd)$ are isometric. If, in addition, Lipschitz functions are dense in both $W^{1,q}(\m)$
and $W^{1,q}(\m^{*})$ then for all $f\in D(\Ch_{q}^{M^{*}})$ it
holds
\[
\Ch_{q}^{M}(\hat{f})=\Ch_{q}^{M^{*}}(f).
\]
In particular, the natural lift of functions defined on $M^{*}$ induces
an isometric embedding of $W^{1,q}(\m^{*})$ into $W^{1,q}(\m)$ whose
image is the set of $\G$-invariant functions in $W^{1,q}(\m)$.
\end{cor}

% PROOF

\begin{proof}
Let $f\in W^{1,q}(\m^{*})\cap\LIP(M^{*},\sfd^{*})$. The
assumptions and the fact that the upper asymptotic Lipschitz constant is preserved by lifts imply that
\[
\Ch_{q}^{M^{*}}(f)  =L_{q}^{M^{*}}(f)=L^{M}_{q}(\hat{f})=\Ch_{q}^{M}(\hat{f}).
\]
Since  by Lemma \ref{lem:Lqm*Lqm}, we have $\|f\|_{L^{q}(\m^{*})}=\|\hat{f}\|_{L^{q}(\m)}$, we get
\begin{align*}
\|f\|_{W^{1,q}(\m^{*})}^{q} & =\|f\|_{L^{q}(\m^{*})}^{q}+\Ch_{q}^{M^{*}}(f) =\|\hat{f}\|_{L^{q}(\m)}^{q}+\Ch_{q}^{M}(\hat{f})
 =\|\hat{f}\|_{W^{1,q}(\m).}^{q}
\end{align*}
Finally, note that any $\G$-invariant function in $W^{1,q}(\m)\cap\LIP(M,\sfd)$
is a lift of a function in $W^{1,q}(\m^{*})\cap\LIP(M^{*},\sfd^{*})$.  This concludes the proof of the  first part of the corollary. The second part follows by a standard density argument.
\end{proof}

Corollary \ref{cor:W1qmw1qm*} directly implies the next result.

\begin{cor}\label{cor:quotInfHilb}
Let $(M,\sfd,\m)$ be a m.m. space satisfying \eqref{eq:lip=minslope} for $q=2$, and that Lipschitz functions are dense  in both $W^{1,2}(\m)$ and   $W^{1,2}(\m^{*})$. Assume further that $(M,\sfd,\m)$ is infinitesimally Hilbertian, then so is $(M^{*},\sfd^{*},\m^{*})$.  
\end{cor}

% REMARK

\begin{remark} \label{rem:LocLipRelGradCD}
The assumptions of Corollaries \ref{cor:W1qmw1qm*} and \ref{cor:quotInfHilb} are fulfilled in case $(M,\sfd,\mm)$ is a strong $\CD^{*}(K,N)$-space, for some $K\in \R, N\in [1,\infty)$. Indeed, by Theorem \ref{thm:CDbyGisCD},  also $(M^{*},\sfd^{*},\mm^{*})$ is a strong $\CD^{*}(K,N)$-space. Moreover, by \cite{BS10, R2011}, we know that every $\CD^{*}(K,N)$-space is a locally doubling and Poincar\'e space; therefore it follows  from \cite{CheegerGAFA}  that the upper asymptotic Lipschitz constant and the $q$-minimal
relaxed slope agree $\mm$-a.e. for every locally Lipschitz function in $W^{1,q}(\m)$ and that Lipschitz functions are dense in  $W^{1,q}(\m^{*})$, for all $q \in (1,\infty)$. Actually for the density of bounded Lipschitz functions with bounded support in $W^{1,q}(\m), q \in (1,\infty)$, it is enough to assume that $(X,\sfd,\mm)$ is doubling (see \cite{ACD}).
\end{remark}

% SUBSECTION: GENERAL CASE
%----------------------------------------------------------------------------------
\subsection{The general case}
In this subsection we treat the case of a  general metric measure space, without assuming locally doubling and Poincar\'e.  The arguments of the present section will be slightly more technically involved but have the advantage of being self-contained, in particular they do not rely on the deep work of Cheeger \cite{CheegerGAFA}. As in the previous sections $\G$ is a compact Lie group acting on $(M,\sfd,\mm)$ in a continuous way via isomorphisms of m.m. spaces.

To achieve the identification of the Sobolev spaces in the general case we will need several auxiliary lemmas.

% LEMMA - Lipschitz constant on balls of radius r is continuous
\begin{lem}\label{lem:lipcont}
Let $(M,\sfd)$ be a geodesic metric space.  Given $f\colon M\to \R$ any function, for $r>0$ define 
\[
f_{r}(x)=\sup_{y\in B_{r}(x)}\frac{|f(y)-f(x)|}{r}.
\]
Then $f_{r}: M\to \R$ is continuous for every Lipschitz function $f\in \LIP(M,\sfd)$.
\end{lem}

% PROF

\begin{proof}
Fix $f\in \LIP(M,\sfd)$ and $x \in M$. Let $x_{n}\to x$ and $y_{n}\in B_{r}(x_{n})$ be such that 
\begin{equation}\label{eq:frxn}
f_{r}(x_{n})=\frac{|f(y_{n})-f(x_{n})|}{r}+\varepsilon_{n}
\end{equation}
 for some sequence $\epsilon_{n}\to0$. Let $\gamma_{n}:[0,1]\to M$
be a geodesic connecting $x$ and $y_{n}$. By continuity,
 since $\sfd(x,x_{n})\to 0$, 
there is a sequence $t_{n}\uparrow 1$ such
that
\begin{equation}\label{eq:defzn}
z_{n}=\gamma_{n}(t_{n})\in B_{r}(x).
\end{equation}
Thus, calling $L$ the Lipschitz constant of $f$, we infer  
\begin{align} 
\Big||f(y_{n})-f(x_{n})|-|f(z_{n})-f(x)|\Big|& \le  |f(y_{n})-f(z_{n})|+|f(x)-f(x_{n})| \leq L \Big( \sfd(y_{n}, z_{n})+ \sfd(x,x_{n}) \Big) \nonumber\\
&\le L\Big(  (1-t_{n})\, \sfd(x,y_{n})+\sfd(x,x_{n}) \Big). \label{eq:fynfzn}
\end{align}
Letting $n\to\infty$ we see that that the right hand side of \eqref{eq:fynfzn}  converges to $0$. Thus the combination of  \eqref{eq:frxn},\eqref{eq:defzn} and  \eqref{eq:fynfzn} yields
\begin{align*}
f_{r}(x)&\ge \limsup_{n\to \infty} \frac{|f(z_{n})-f(x)|}{r} =  \limsup_{n\to \infty} \frac{|f(y_{n})-f(x_{n})|}{r}= \limsup_{n\to \infty}f_{r}(x_{n}).
\end{align*}
To show the complementary inequality, observe that $y\in B_{r}(x)$ implies that $y\in B_{r}(x_{n})$
for sufficiently large $n$. This yields
\[
f_{r}(x)\le\liminf_{x_{n}\to x}f_{r}(x_{n}),
\]
proving that $f_{r}$ is continuous.
\end{proof}

The group $\G$ acts naturally on the set of functions defined on
$M$ as follows. Given a function $f\colon M\to\mathbb{R}$ and an element $g\in\G$, define $f_{g}:M\to\mathbb{R}$ as $f_{g}(x)=f(gx)$. 
Therefore, given a measurable function $f\colon M\to\mathbb{R}$, we can define a natural
 measurable $\G$-invariant function $f_{(\G)}$ by integrating
over $(\G,\nu_{\G})$:
\begin{equation}\label{eq:deff(G)}
f_{(\G)}(x)=\int_{\G}f_{g}(x) \, d\nu_{\G}(g).
\end{equation}
We call $f_{(\G)}$ the \emph{$\G$-average of $f$}.

Our next goal is to control the asymptotic  Lipschitz constant of $f_{(\G)}$ in terms of the one of $f$. To that aim we will make use of the next lemma.

% LEMMA - averaging does not increase Lipschitz constant on balls
\begin{lem}\label{lem:supfGsupf}
Let $(M,\sfd)$ be a geodesic metric space. Then, for any Lipschitz function
$f\colon M\to\mathbb{R}$,
\[
(f_{(\G)})_r(x)=\sup_{y\in B_{r}(x)}\frac{|f_{(\G)}(y)-f_{(\G)}(x)|}{r}\le\int_{\G} f_{r}(gx) \, d\nu_{\G}(g)=(f_r)_{(\G)}(x).
\]
\end{lem}
\begin{proof}
Let $\{g_{i}\}_{i 	\in\N}\subset\G$ be  a dense sequence,  so that  
\begin{equation}\label{eq:nuntonug}
\nu_{n}:=\frac{1}{n}\sum_{i=1}^{n}\delta_{g_{i}}  \rightharpoonup  \nu_{\G},
\end{equation}
with respect to weak convergence of probability measures in duality with $C^{0}(\G)$.

Define $f_{n}:M\to\mathbb{R}$ by  $f_{n}(x)=\int_{\G} f(gx)\, d\nu_{n}(g)$ and note that $f_{n}\to f_{(\G)}$ pointwise in $M$. In particular,
\begin{equation*}%\label{eq:suofntosupf}
\sup_{y\in B_{r}(x)}\frac{|f_{n}(y)-f_{n}(x)|}{r}\to\sup_{y\in B_{r}(x)}\frac{|f_{(\G)}(y)-f_{(\G)}(x)|}{r}.
\end{equation*}
 Furthermore, we have 
\begin{equation}\label{eq:supfnfr}
\sup_{y\in B_{r}(x)}\frac{|f_{n}(y)-f_{n}(x)|}{r}  \le  \frac{1}{n}\sum_{i=1}^{n}\sup_{y\in B_{r}(g_{i}x)}\frac{|f(y)-f(g_{i}x)|}{r} =  \int_{\G} f_{r}(g x)\, d\nu_{n}(g).
\end{equation}
Note that the map $\G\mapsto f_{r}(g x)$ is continuous, since it is the composition of the maps $\star_x:\G\to M$ and $f_{r}:M\to \mathbb R$ which, by assumption and Lemma \ref{lem:lipcont}, are continuous. The weak convergence \eqref{eq:nuntonug} of $\nu_{n}  \rightharpoonup \nu_{\G}$  then ensures that the inequality \eqref{eq:supfnfr}
is preserved in the limit, i.e. 
\[
\sup_{y\in B_{r}(x)}\frac{|f_{(\G)}(y)-f_{(\G)}(x)|}{r}\le\int_{\G} f_{r}(gx)\, d\nu_{\G}(g).
\]
\end{proof}

Using Lemma \ref{lem:supfGsupf} we can prove the next important proposition which estimates, for a Lipschitz function $f$,  the upper asymptotic Lipschitz constant of $f_{(\G)}$ in terms of $\G$-rearrengement of the one of $f$.

% PROPOSITION - averaging is non-increasing w.r.t. Lq-norm
\begin{prop}\label{prop:LipfGleLipf}
Let $(M,\sfd)$ be a geodesic metric space. Then for every Lipschitz function $f\colon M\to \R$ and all $q\in[1,\infty)$,
\[
\left[\Lip f_{(\G)}(x)\right]^{q}\le\int_{\G} \left[\Lip f(gx)\right]^{q}\, d\nu_{\G}(g) =(\left[\Lip f(x)\right ]^{q})_{(\G)}.
\]
In particular, $f_{(\G)}:M\to \R$ is Lipschitz continuous.
\end{prop}

% PROOF

\begin{proof}Let $L\geq 0$ be the Lipschitz constant of $f$ and note that $f_{r}(x)\leq L$ for all $x \in M$.
  We can then use Fatou's Lemma to infer that
\[
\limsup_{r\downarrow 0} \int_{\G} f_{r}(gx) \, d\nu_{\G}(g) \leq  \int_{\G}   \limsup_{r\downarrow 0} f_{r}(gx) \, d\nu_{\G}(g)=   \int_{\G} \Lip f (gx) \, d\nu_{\G}(g).
\]
Combining the last inequality with Lemma  \ref{lem:supfGsupf}  and with Jensen's inequality, we get
\begin{align*}
\left[\Lip f_{(\G)}(x)\right]^{q} =&\left[\limsup_{r\downarrow 0} \sup_{y\in B_{r}(x)}\frac{|f_{(\G)}(y)-f_{(\G)}(x)|}{r} \right]^{q} \leq \left[ \limsup_{r\downarrow 0} \int_{\G} f_{r}(gx) \, d\nu_{\G}(g)\right]^{q}  \\
&  \leq \left[ \int_{\G} \Lip f (gx) \, d\nu_{\G}(g) \right]^{q} \leq  \int_{\G}  \left[  \Lip f (gx)  \right]^{q} \, d\nu_{\G}(g).
\end{align*}
\end{proof}

% PROPOSITION - strong approximation of equivariant functions via equivariant functions
\begin{prop}
\label{prop:invar-approx}
Fix $q\in (1,\infty)$. Let $(M,\sfd,\mm)$ be a m.m. space with $(M,\sfd)$ geodesic.  Then, for every $\G$-invariant $f\in W^{1,q}(\m)$,
there is a sequence of $\G$-invariant Lipschitz functions $\{f_{n}\}_{n \in \N} \subset L^{q}(\m)$  such that 
\[
f_{n}\to f  \quad {and} \quad
 \Lip f_{n}\to|\nabla f|_{q}\;\mbox{strongly in }L^{q}(\m).
\]
\end{prop} 

% PROOF

\begin{proof}
We already know that given $f\in W^{1,q}(\m)$ there exists a sequence $\{h^{n}\}_{n \in \N}\subset \LIP(M,\sfd)$  such that $h^{n}\weakto f$ weakly in $L^{q}(\mm)$ and $\Lip h^{n}\to|\nabla f|_{q}$
strongly in $L^{q}(\m)$. Note that, by definition, $h_{(\G)}^{n}$ is $\G$-invariant and moreover $h_{(\G)}^{n}\in \LIP(M,\sfd)$ by  Proposition \ref{prop:LipfGleLipf}.  
First of all we claim that $h^{n}_{(G)}\weakto f$ weakly in $L^{q}(\m)$, indeed since $f$ and $\mm$ are $\G$-invariant, using Fubini-Tonelli's Theorem we have
\begin{align}
\int_{M}(f(x)-h^{n}_{(\G)}(x)) \varphi(x) \, d \mm(x)&= \int_{M} \int_{\G} (f(gx)-h^{n}(gx)) \varphi (x) \, d \nu_{\G}(g) \, d \mm(x) \nonumber \\
&=\int_{\G} \left[\int_{M} (f(gx)-h^{n}(gx)) \varphi (x) \,d \mm(x) \right] d \nu_{\G}(g) \nonumber \\
&=  \int_{\G} \left[\int_{M} (f(x)-h^{n}(x)) \varphi (g^{-1}x) \,d \mm(x) \right] d \nu_{\G}(g) \to 0 \text{ as $n\to \infty$}, \; \forall \varphi \in L^{p}(\mm). 
\end{align}
Where, of course $1/p+1/q=1$. Next we show that  $\Lip h^{n}_{{(\G)}}\to|\nabla f|_{q}$ strongly in $L^{q}(\mm)$, up to subsequences and  convex combinations. To this aim, using again the $\G$-invariance of $\mm$,  observe that
\begin{equation}\label{eq:LiphLiphg}
\int_{M}(\Lip h^{n})^{q}d\m=\int_{M}(\Lip h_{g}^{n})^{q}d\m, \qquad \forall g \in \G,
\end{equation}
where, as usual, we let $h_{g}^{n}(x)=h^{n}(gx)$. Using \eqref{eq:LiphLiphg} together with  Proposition \ref{prop:LipfGleLipf} and  Fubini-Tonelli's Theorem we infer
\[
\int_{M}(\Lip h_{(\G)}^{n})^{q}d\m\le\int_{M}\int_{\G}(\Lip h_{g}^{n})^{q}\, d\nu_{\G}(g)\,  d\m= \int_{\G}\int_{M}(\Lip h_{g}^{n})^{q}  \, d\m\, d\nu_{\G}(g)= \int_{M}(\Lip h^{n})^{q}d\m.
\]
Recalling that, by construction,  $\Lip h^{n}\to|\nabla f|_{q}$ in $L^{q}(\mm)$, we infer that
\begin{equation}\label{eq:limsupnablahnq}
  \limsup_{n\to \infty}  \int_{M}(\Lip h_{(\G)}^{n})^{q} \,d\m \leq  \limsup_{n\to \infty}  \int_{M}(\Lip h^{n})^{q}d\m =\int_{M}  |\nabla f|_{q}^{q} \, d \mm.
\end{equation} 
After extracting a subsequence we see that $\Lip h_{(\G)}^{n}$ converge
weakly in $L^{q}(\mm)$  to some $g\in L^{q}(\m)$; therefore up to considering subsequences and  convex combinations, by the Mazur Lemma, we can assume that $ h_{(\G)}^{n}\to f$ strongly in $L^{q}(\mm)$ and $\Lip h_{(\G)}^{n}$ converge strongly in $L^{q}(\mm)$ to some $g \in L^{q}(\mm)$. Clearly a   convex combination $\G$-invariant functions is still $\G$-invariant.  The uniqueness and minimality property of $|\nabla f|_{q}$ then ensures that $g=|\nabla f|_{q}$, as desired.
\end{proof}

Using the last two  propositions we can show the next result.

% THEOREM - equivariant lifting induces an isometry
\begin{thm}\label{thm:IdentCh;ChM*}
Let $(M,\sfd,\mm)$ be a m.m. space such that $(M,\sfd)$ is a  geodesic space.  Then,  for all $f\in D(\Ch_{q}^{M^{*}})$, one has  $\hat{f}\in D(\Ch_{q}^{M})$ and 
\[
\Ch_{q}^{M}(\hat{f})=\Ch_{q}^{M^{*}}(f).
\]
In particular, the natural lift of functions defined on $M^{*}$
induces an isometric embedding of $W^{1,q}(\m^{*})$ into $W^{1,q}(\m)$
whose image is the set of $\G$-invariant functions in $W^{1,q}(\m)$.
\end{thm}
\begin{proof}
Fix $f\in D(\Ch_{q}^{M^{*}})$. From the definition of the Cheeger energy, using the Mazur Lemma, we can find a sequence $\{f^{n}\}_{n \in \N}\subset\LIP(M^{*},\sfd^{*})$
 such that  $f^{n}\to f$ and $\Lip^{M^{*}}f^{n}\to|\nabla^{M^{*}}f|_{q}$ strongly
in $L^{q}(\m^{*})$. Denote the corresponding lifts by $\hat{f}$ and $\hat{f}^{n}$
respectively. Using Lemma \ref{lem:Lqm*Lqm}, we have that 
\begin{equation}
\hat{f}^{n} \to \hat{f}  \quad \text{strongly in $L^{q}(\mm)$}.
\end{equation}
 Moreover, by Proposition 	\ref{prop:LipfGf*G} we know that   $\Lip^{M}\hat{f}^{n}(x)=\Lip^{M^{*}}f^{n}(x^{*})$, in particular $\hat{f}^{n}\in \LIP(M,\sfd)$. Thus   $
\int_{M}(\Lip^{M}\hat{f}^{n})^{q}\, d\m=\int_{M^{*}}(\Lip^{M^{*}}f^{n})^{q}\, d\m^{*}$, and 
\begin{equation}\label{eq:ChqMleqVhqM*}
\Ch_{q}^{M}(\hat{f})  \le  \liminf_{n\to\infty}\frac{1}{q}\int_{M}(\Lip^{M}\hat{f}^{n})^{q}d\m 
  =  \lim_{n\to\infty}\frac{1}{q}\int_{M^{*}}(\Lip^{M^{*}}f^{n})^{q}d\m^{*}=\Ch_{q}^{M^{*}}(f).
\end{equation}
In particular, $\hat{f}\in D(\Ch_{q}^{M})$. 

To obtain the equality of the Cheeger energies, observe that Proposition
\ref{prop:invar-approx} gives a $\G$-invariant Lipschitz
approximation $g_{n}$ of $\hat{f}$ such that 
\begin{equation}\label{eq:Chgn}
g_{n} \to \hat{f} \; \text{ strongly in $L^{q}(\mm)$ \quad and} \quad\Ch_{q}(\hat{f})=\lim_{n\to \infty}\frac{1}{q}\int_{M}(\Lip^{M}g_{n})^{q}d\m.
\end{equation}
By $\G$-invariance of $g_{n}$ we can define  $g_{n}^{*}(x^{*}):=g_{n}(x)$. From Lemma \ref{lem:Lqm*Lqm} and \eqref{eq:Chgn}  we get that $g_{n}^{*} \to f$ strongly in $L^{q}(\mm^{*})$. Therefore, since by  Proposition 	\ref{prop:LipfGf*G} we know  that  $\Lip^{M^{*}}g^{*}_{n}(x^{*})=\Lip^{M} g_{n}(x)$,  we get
\begin{align*}
\Ch_{q}^{M^{*}}(f) \le  \liminf_{n\to\infty}\frac{1}{q}\int_{M^{*}}(\Lip^{M^{*}}g_{n}^{*})^{q} \, d\m^{*} =  \lim_{n\to\infty}\frac{1}{q}\int_{M}(\Lip^{M}g_{n})^{q}d\m=\Ch_{q}^{M}(\hat{f})
\end{align*}
which, combined with  \eqref{eq:ChqMleqVhqM*}, proves that $\Ch_{q}^{M^{*}}(f)=\Ch_{q}^{M}(\hat{f})$. As above this implies the claim.
\end{proof}

% COROLLARY

\begin{cor}\label{cor:infHilQuotGen}
Let $(M,\sfd,\m)$  be a m.m. space with $(M,\sfd)$ geodesic. If $(M,\sfd,\m)$ is infinitesimally Hilbertian so is $(M^{*},\sfd^{*},\m^{*})$.
Furthermore, if 
\[
|\nabla^{M}f|_{q}=\Lip^{M}f ,\quad \mm\text{-a.e.}, \quad\forall f \in \LIP(M,\sfd),
\]
then also 
\[
|\nabla^{M^{*}}h|_{q}=\Lip^{M^{*}}h  ,\quad \mm^{*}\text{-a.e.},  \quad\forall h \in \LIP(M,\sfd).
\]
\end{cor}

% REMARK

\begin{remark}
The converse of this statement is false. Indeed, this can be verified by taking a product
of any infinitesimally Hilbertian space with a torus equipped with
a flat Finsler structure or, more generally, with a compact higher rank symmetric space equipped with a non-Riemannian
Berwald structure. 
\end{remark}

\begin{remark}
The assumptions of this section can be weakened in two directions without affecting the validity of the results. On the one hand, we can consider more general m.m. spaces, specifically, spaces in which the topological closure of any open ball $B_{r}(x):=\{y\in M\,:\, \sfd(y,x)<r\}$ coincides with the closed metric ball $\bar{B}_{r}(x):=\{y\in M\,:\, \sfd(y,x)\leq r\}$.  
Moreover, there is some flexibility on the behavior of the push-forward measure under the induced action. It is sufficient to assume that the compact Lie group $\G$ acts on $(M,\sfd,\mm)$ by almost measure-preserving isometries, i.e. $ \sfd(gx, gy)=\sfd(x,y)$ for all $x,y \in M, g \in \G$ and  there exists $C>0$ such that for all $g\in\G$ 
\[
C^{-1}\cdot \m \le (\tau_{g})_{\sharp} \m \le C\cdot \m.
\]
Note however that, without any assumption on the interaction between the group
and the measure, Theorem \ref{thm:IdentCh;ChM*} might fail to be true. The reason is that, in general, without such an assumption we can not hope for any kind of control of minimal relaxed slopes on the quotient since the slopes depend on both: metric and measure.
\\

For example look at the torus $\mathbb{T}^2$ 
as $[0,1]\times[0,1]$ with corresponding identifications and consider  the isometric action of $\mathbb{S}^1$ on the second factor.
On $(\epsilon,1]\times [0,1]$ choose a measure $\tilde{\m}$ with trivial $q$-th Sobolev
space as constructed in \cite{DMS2015}. Note that $\tilde{\m}$ 
is absolutely continuous and has density bounded by some $C>0$. 
Denote the second marginal of $\tilde{\m}$ by $\tilde{\m}^*$.
Let now $\hat{\m}$ be the Lebesgue measure on $[0,\epsilon]\times[0,1]$ and $\m$ be 
the sum of $\tilde{\m}$ and $\hat{\m}$. Then the $q$-minimal relaxed slope of any
Sobolev function is zero on $(\epsilon,1]\times [0,1]$ and any $1$-Lipschitz function
has Cheeger energy at most $\epsilon/q$. However, the $q$-th Sobolev space
on the quotient is non-trivial. Indeed, the measure $\m^*$ on $M^*=\mathbb{S}^1$ satisfies 
\[
\epsilon\cdot \vol_{\mathbb{S}^1} \le \m^* \le (C+\epsilon)\cdot \vol_{\mathbb{S}^1}.
\]
Let $f \in \LIP(M^*, \sfd^{*})$ with $|\nabla^{M^{*}}f|_{q}=\Lip^{M^{*}}f \ne 0$ $\mm^{*}$-a.e.. By the construction of $\tilde{\mm}$, it holds $|\nabla^{M}\hat f|_{q}(x)=0$ whenever $x\in (\epsilon,1]\times[0,1]$. In particular, Theorem \ref{thm:IdentCh;ChM*}  
fails in this case.
\end{remark}

%-----------------------------------------------------------------------------------------------------------
% SECTION: STRUCTURE THEORY FOR QUOTIENTS OF RCD*(K,N)-SPACES
%-----------------------------------------------------------------------------------------------------------

\section{Structure theory for quotients of $\RCD^{*}(K,N)$-spaces}\label{sec:StructQuotRCD}

 We start by recalling the definition of  the $\RCD(K,\infty)$-condition and its finite dimensional refinement the $\RCD^{*}(K,N)$-condition. We refer to  \cite{Ambrosio-Gigli-Savare11b,AmbrosioGigliMondinoRajala,gigli:laplacian, ErbarKuwadaSturm, AmbrosioMondinoSavare} for a more detailed exposition and for a historical recount of the developments of the condition.  

% DEFINITION - RCD space

\begin{defn}[$\RCD(K,\infty)$ and $\RCD^{*}(K,N)$-spaces]
Let  $(M,\sfd,\mm)$ be a  metric measure space. 
We say that $(M,\sfd,\mm)$ is an \emph{$\RCD(K,\infty)$-space} if  it is infinitesimally Hilbertian, i.e. \eqref{eq:paralrule} holds, and satisfies the $\CD(K,\infty)$-condition. We say that  $(M,\sfd,\mm)$ is an  \emph{$\RCD^{*}(K,N)$-space} if it is infinitesimally Hilbertian and satisfies the $\CD^{*}(K,N)$-condition. 
\end{defn}

We recall a useful property of $\RCD(K,\infty)$-spaces proved in \cite[Theorem 1.2]{GRS2016}. On an $\RCD(K,\infty)$-space $(X,\sfd,\mm)$, there is a  unique $W_{2}$-geodesic joining any two measures $\mu_{0},\mu_{1}\in \mathcal{P}^{ac}_{2}(X,\sfd,\mm)$. Thus, since every $\RCD^{*}(K,N)$ is also $\RCD(K,\infty)$, we know a fortiori that every $\RCD(K,\infty)$-space is also a strong  $\CD(K,\infty)$-space, and similarly, every $\RCD^{*}(K,N)$-space respectively is a strong $\CD^{*}(K,N)$-space.
This observation,  combined with Theorem \ref{thm:CDbyGisCD} and Corollary \ref{cor:infHilQuotGen} (or Corollary \ref{cor:quotInfHilb} for $N\in [1,\infty)$), gives the next result, which is one the main achievements of the present paper.

% THM: STABILITY OF RCD CONDITIONS

\begin{thm}\label{thm:M*RCD}
Let $(M,\sfd,\mm)$ be a metric measure space such that $(M,\sfd)$ is geodesic. Let $\G$ be a compact Lie Group acting on $M$ effectively by measure-preserving isometries.
If $(M,\sfd,\mm)$ is an $\RCD(K,\infty)$-space (respectively, an $\RCD^{*}(K,N)$-space for $N\in [1,\infty)$) for  $K \in \R$, then the  corresponding quotient metric measure space $(M^{*},\sfd^{*},\mm^{*})$ is an $\RCD(K,\infty)$-space (respectively, an $\RCD^{*}(K,N)$-space).
\end{thm}

In this section we  equip $\G$ with a bi-invariant inner metric $\sfd_\G$.
Note that the compactness of $\G$ ensures the existence of such a metric. Moreover, the arguments that follow are independent of the choice of such a metric since any other bi-invariant metric on $\G$ is bi-Lipschitz equivalent to $\sfd_\G$ (see for instance \cite[Theorem 7]{Berestovskii1}).

In the remainder of the section we will assume that the metric on $\G$ is such that the map $\star_y:\G\to\G(y),\, g\mapsto gy$, is locally Lipschitz and co-Lipschitz continuous, for some (and hence for all) $y\in M$ with principal orbit type.  Specifically, in view of the bi-invariance of the metric,  we assume that for every $y\in M$ in a principal orbit, there exist constants $R,C>0$ such that, for all $r\in(0,R)$,
\[
 B_{C^{-1} r}(y)\cap\G(y) \subset \{g \cdot y ~|~ g\in B^\G_r(1_\G)\}\subset B_{Cr}(y) \cap\G(y),
\]
where $B_r^\G(1_\G)$ denotes the $\sfd_\G$-ball of radius $r$ around the identity $1_\G\in \G$.  

From now on we consider the finite dimensional case, i.e.  $\RCD^{*}(K,N)$ for  $N\in [1,\infty)$. We recall that in \cite{MN2014} it was proved that $\RCD^{*}(K,N)$-spaces with $N<\infty$
have $\m$-almost everywhere unique Euclidean tangent spaces of possibly
varying dimension. Since the infinitesimal regularity is a metric-measure property and the group $\G$ acts on $(M,\sfd,\mm)$ by measure-preserving isometries, we have that $x\in M$  has a unique tangent space $\R^{n}$ if and only if all the points in $\G(x)$ satisfy the same property. To stress the dependence on the base point we denote the dimension of the tangent space by  $n(x)$ for $\mm$-a.e. $x \in M$,  and by $n(x^{*})$ for $\mm^{*}$-a.e. $x^{*}\in M^{*}$. 
We also recall that thanks to Theorem \ref{thm:ExGTB} we know that $\RCD^*$-spaces satisfy the assumptions of the Principal Orbit Theorem \ref{thm:PrincipleOrbit}.

We point out that Guijarro and Santos-Rodr\'iguez considered certain  fundamental aspects of isometric actions of Lie groups on $\RCD^*(K,N)$-spaces in \cite{GS2016IsoRCD}. The proof of the next result should be compared with the proof of \cite[Theorem 3]{GS2016IsoRCD}, which also used the work of Berestovskii \cite{Berestovskii2}. 
Here we do not assume the action to be transitive and, as a consequence, we gain information on the dimension of tangent spaces in the orbit space. To do so, the Lipschitz and co-Lipschitz assumptions, as well as  Pansu's Differentiability Theorem for Lipschitz maps between Carnot-Caratheodory spaces \cite{Pansu}, play an important role. In turn, this information on the dimension of tangent spaces in the orbit space becomes essential for the structure theorems that follow.

% THM

\begin{thm}\label{thm:struct}
Let $(M,\sfd,\mm)$ be an $\RCD^{*}(K,N)$ m.m. space for some $K\in \R, N \in (1,\infty)$ and let $\G$ be a compact connected Lie group acting locally Lipschitz and co-Lipschitz continuously by measure preserving isometries on $(M,\sfd,\mm)$. 
Then, for $\m$-almost every $x_{0}\in M$,
\[
n(x_{0})=\dim\G(x_{0})+n(x_0^{*}),
\]
where $\G(x_{0})$ is a principal orbit.
\end{thm}

% PROOF

\begin{proof}
Up to a subset of $\mm$-measure zero we know that every point  $x_{0}\in M$ is of principal orbit type and $n(x_{0})$-regular; clearly the same properties hold for all $x \in \G(x_{0})$.
Up to a further $\mm$-negligible subset, we may also assume that $x_{0}^{*}=\quotient (x_{0})$ is an $n(x_{0}^{*})$-regular point in the quotient space $(M^{*}, \sfd^{*},\mm^{*})$.
 
We first analyze the structural properties of the orbit $\G(x_{0})$. 
Denote by $\sfd_{\G(x_{0})}$ the induced inner metric.  
Since by assumption $\G$ is a  compact connected Lie group which acts Lipschitz and co-Lipschitz continuously on $M$, it follows that the metrics $\sfd_{\G(x_{0})}$ and $\sfd$ restricted to  $\G(x_{0})$ are bi-Lipschitz equivalent.
Notice also  that the isotropy group of $x_{0}$, denoted by $\G_{x_{0}}$, is a compact subgroup of $\G$ and that $\G(x_{0})$ is homeomorphic to the homogeneous space $\G/ \G_{x_{0}}$. 
From \cite[Theorem 3 (i)]{Berestovskii2} we infer that there exists a connected Lie group $\G'$ and a compact subgroup $\H'<\G'$ such that $(\G(x_{0}), \sfd_{\G(x_{0})})$ is isometric to the quotient space $\G' / \H'$ endowed with a Carnot-Caratheodory-Finsler metric   defined
invariantly with respect to the canonical action of $\G'$ on $\G' /\H'$ by a completely non-holonomic
distribution  on $\G' /\H'$. In particular it follows that at every $x \in \G(x_{0})$, the GH-tangent cone $T_{x}\G(x_{0})$ at  $x$ is unique and isometric to a Carnot group endowed with a Finsler metric.

On the other hand, since by construction for every $x\in \G(x_{0})$ the tangent cone in $(M,\sfd)$ is unique and isometric to $\R^{n}$, it follows that  $T_{x}\G(x_{0})$ can be isometrically embedded in $\R^{n(x_{0})}$.
By an observation  of Semmes \cite{Semmes} (attributed independently to an unpublished work of Assouad, see also \cite[Section 14]{CheegerGAFA}), based on Pansu's Differentiability Theorem for
Lipschitz maps between Carnot-Caratheodory spaces \cite{Pansu}, it follows that   $T_{x}\G(x_{0})$ is isometric to the Euclidean space $\R^{\dim\G(x_{0})} \subset \R^{n(x_{0})}$.
Note that in particular this implies that $(\G(x_{0}), \sfd_{\G(x_{0})})$ is isometric to a  homogenous Riemannian manifold.

Theorefore, the equivariant  Gromov-Hausdorff blow-up based at $x_{0}$ converges to 
\[
\nicefrac{T_{x_{0}}M}{T_{x_{0}} \G(x_{0})}\cong\nicefrac{\mathbb{R}^{n(x_{0})}}{\mathbb{R}^{\dim\G(x_{0})}}\cong\mathbb{R}^{n(x_{0})-\dim\G(x_{0})}.
\]
Since such equivariant limit is nothing but the tangent space of
$M^{*}$ at $x_{0}^{*}$, which by assumption is isometric to $\mathbb{R}^{n(x_{0}^{*})}$, the thesis follows.
\end{proof}

The departure point for our next result is the following classical result of Kobayashi \cite[Ch.~II, Theorem~3.1]{Kobayashi}.

% THM: SYMMETRY DEGREE FOR RIEMANNIAN MANIFOLDS

\begin{thm}[\cite{Kobayashi}]\label{thm:kobayashi} 
Let $(M^{n},g)$ be a connected $n$-dimensional smooth Riemannian manifold. Then the isometry group has dimension at most $n(n+1)/2$. If equality holds, then $M$ is isometric
to one of the following space forms:
\begin{itemize}
\item An $n$-dimensional Euclidean space $\mathbb{R}^{n}$
\item An $n$-dimensional sphere $\mathbb{S}^{n}$
\item An $n$-dimensional real projective space $\mathbb{RP}^{n}$
\item An $n$-dimensional simply connected hyperbolic space $\mathbb{H}^{n}$.
\end{itemize}
\end{thm}

The following simple corollary to Theorem~\ref{thm:kobayashi} will play a useful role in the rest of the section. 

% COROLLARY: BOUND DIM ISOTROPY GROUPS

\begin{cor}\label{cor:dimGx}
Let $(M^{n},g)$ be an $n$-dimensional Riemannian manifold and assume that a compact Lie group $\G$ acts  transitively and isometrically on $M$. Then, for every $x\in M$,
\[
\dim(\G_{x})\leq \frac{n(n-1)}{2}.
\]
Moreover, if equality is achieved, then  $\dim (\G)=\frac{n(n+1)}{2}$. 
\end{cor}

\begin{proof}
Since $\G$ acts on $M$ by isometries, the isotropy group $\G_x$ acts by isometries on the unit tangent-sphere $\mathbb{S}^{n-1}\subset T_xM$. The bound on $\dim(\G_x)$ now follows from Theorem~\ref{thm:kobayashi}. Assume now that equality holds. Since $\G$ acts transitively on $M$,  $n=\dim(\G) - \dim(\G_x)$ and we conclude that $\dim(\G)=\frac{n(n+1)}{2}$.
\end{proof}

This corollary has been extended to $\RCD^*(K,N)$-spaces by Guijarro and Santos-Rodr\'iguez \cite{GS2016IsoRCD}. They have also generalized 
Theorem~\ref{thm:kobayashi}  as follows.

% THM: G-SR

\begin{thm}[\protect{\cite[Theorems 2 and 3]{GS2016IsoRCD}}]\label{thm:GSR}
Let $(M,\sfd,\mm)$ be an   $\RCD^{*}(K,N)$-space for some $K \in \R, N \in (1,\infty)$, and let $\lfloor N \rfloor$ be the integer part of $N$. Then the isometry group has dimension at most $\lfloor N \rfloor(\lfloor N \rfloor+1)/2$ and if equality is attained then the same  rigidity statement of Theorem  \ref{thm:kobayashi} holds.
\end{thm}

Let us mention that in case $(M,\sfd,\mm)$ is a pointed measured Gromov Hausdorff limit of a sequence of Riemannian manifolds with Ricci $\geq K$ and dimension $\leq N$, then the assumption that the  topological dimension of $M$ coincides with the integer part of $N$ implies that such a Ricci limit space is non-collapsed \cite{CC1}.
 It is then natural to consider the more general setting which includes possibly collapsed Ricci limits.   Using the structure Theorem \ref{thm:struct}, we can strengthen the above rigidity Theorem  \ref{thm:GSR}, by  showing that it suffices to look at the
(essential-) least possible dimension of the tangent spaces. To this aim recall that for $\mm$-a.e. $x\in M$ the tangent space is unique and Euclidean of dimension $n(x)\in \N \cap [1, N]$. We set
\begin{equation}\label{eq:defn}
n:=\underset{x \in M}{{\rm{ess\,inf}}} \;  n(x).
\end{equation}
 
 % THM: RIGIDITY

\begin{thm}\label{thm:Rigidity}
Let $(M,\sfd,\mm)$ be an $\RCD^{*}(K,N)$-space for some $K \in \R, N \in (1,\infty)$ and let $n$  be defined in \eqref{eq:defn}.  Let $\G$ be a compact Lie subgroup of the Lie group of  measure-preserving isometries of $(M,\sfd,\mm)$ acting both Lipschitz and co-Lipschitz continuously.
Then   $\G$ has dimension at most $n(n+1)/2$. Moreover, if equality is attained, then the action is transitive and $(M,\sfd,\mm)$ is isomorphic as m.m. space to either  $\mathbb{S}^{n}$ or $\mathbb{RP}^{n}$, up to multiplying the measure $\mm$ by a normalizing constant.

\end{thm}

% PROOF

\begin{proof}
%Let us denote the group of measure preserving isometries of $(M,\sfd,\mm)$ by $\G$. 
Let $x_{0}\in M$ be an $n$-regular point  of principal orbit type such that $x_{0}^{*}=\quotient(x_{0})$ is an $n(x_{0}^{*})$-regular point. The existence of such an $x_{0}$ follows by the fact that $\mm$-a.e. is of principal orbit type combined with the very definition \eqref{eq:defn} of $n$ and the fact that $\mm^{*}$-a.e. point in $M^{*}$ is regular, since $(M^{*},\sfd^{*},\mm^{*})$ is an $\RCD^{*}(K,N)$-space from Theorem \ref{thm:M*RCD}. Note also that  if $M^{*}\neq \{x_{0}^{*}\}$ is not a singleton then we can choose $x_{0}\in M$ so that $n(x_{0}^{*})\geq 1$.
From the proof of the structure  Theorem \ref{thm:struct} we know that  the orbit  $\G(x_{0})$, endowed with the induced inner metric $\sfd_{\G(x_{0})}$, is isometric to a homogenous Riemannian manifold of dimension at most $n$. In particular, from Corollary \ref{cor:dimGx} we infer that
\begin{equation}\label{eq:dimGx0leq}
\dim\G_{x_{0}}\leq \frac{\dim(\G(x_{0}))\big(\dim(\G(x_{0}))-1\big)}{2} \leq \frac{n(n-1)}{2}.
\end{equation}
On the other hand,  Theorem \ref{thm:struct} yields  
\begin{align*}
n=\dim(\G(x_{0}))+n(x_{0}^{*})=\dim(\G)-\dim(\G_{x_{0}})+n(x_{0}^{*}),
\end{align*}
which, combined with \eqref{eq:dimGx0leq}, gives
\begin{equation}\label{eq:dimGleq}
\dim(\G)= n+\dim(\G_{x_{0}})-n(x_{0}^{*})\leq   \frac{n(n+1)}{2}-n(x_{0}^{*}).
\end{equation}
It follows that $\dim(\G)\leq  \frac{n(n+1)}{2}$, and if equality is achieved, then $n(x_{0}^{*})=0$. In particular, if equality is achieved then $M^{*}=\{x_{0}^{*}\}$;  in other words,  the action of $\G$ on $M$ is transitive  and  $(M,\sfd)$ is isometric to $(\G(x_{0}), \sfd_{\G(x_{0})})$. Then $(M,\sfd) $ is isometric   to  a Riemannian manifold and the rigidity statement follows directly from Kobayashi's Theorem  \ref{thm:kobayashi}. Note that the cases of $\mathbb{R}^{n}$ and $\mathbb{H}^{n}$ are excluded by the compactness assumption on $\G$. Finally, since $\G$ is acting by \emph{measure-preserving} isometries, it follows that $\mm$ must be a constant multiple of the standard Riemannian measure on $\mathbb{S}^{n}$ (respectively $\mathbb{RP}^{n}$).
\end{proof}

We conclude this section with the following two theorems which, to the best of our knowledge, are new also for Riemannian manifolds and Alexandrov spaces. Recall that the \emph{cohomogeneity} of an action of a compact Lie group on a m.m. space is the (Hausdorff) dimension of its orbit space.

% THM

\begin{thm}\label{thm:cohom1RCD}
Let $(M,\sfd,\mm)$ be an $\RCD^{*}(K,N)$ m.m. space for some $K\in \R, N \in [2,\infty)$ and  let $n$  be defined in \eqref{eq:defn}.   Let $\G$ be a compact Lie group acting effectively and locally Lipschitz and co-Lipschitz continuously by measure-preserving isometries on $(M,\sfd,\mm)$. If the action is not transitive and 
 if 
\[
\dim(\G) \geq \frac{(n-1)n}{2},
\]
then $\G$ has dimension  $\frac{(n-1)n}{2}$ and acts on $M$ by cohomogeneity one with principal orbit homeomorphic to $\mathbb{S}^{n-1}$ or $\mathbb{RP}^{n-1}$. 

Moreover, $M^*$ is isometric  to either a circle or a possibly unbounded closed  interval (i.e. possibly equal to the real line or half line). In the former case, $(M,\sfd)$ is (equivariantly) homeomorphic to a fiber bundle with fiber the principal orbit and base $\mathbb{S}^1$. In particular in this case $(M,\sfd)$ is a topological manifold.
\end{thm}

\begin{proof}
Let $x_{0}\in M$ be an $n$-regular point  of principal orbit type such that $x_{0}^{*}=\quotient(x_{0})$ is an $n(x_{0}^{*})$-regular point.
Since by assumption $\G$  does not act transitively, we can assume $n(x_{0}^{*})\ge1$
and hence using Theorem \ref{thm:struct} we infer
\begin{align*}
n=\dim(\G(x_{0}))+n(x_{0}^{*})=\dim(\G)-\dim(\G_{x_{0}})+n(x_{0}^{*}) \geq \frac{n(n-1)}{2}-\dim(\G_{x_{0}})+n(x_{0}^{*}) ,
\end{align*}
which gives
\begin{equation}\label{eq:dimGxgeq}
\dim(\G_{x_{0}})\geq \frac{n(n-1)}{2} - (n-n(x_{0}^{*}))\ge\frac{(n-1)(n-2)}{2},
\end{equation}
with equality if and only if $n(x_{0}^{*})=1$, i.e. if and only if we have a cohomogeneity one action.  We claim that equality holds in \eqref{eq:dimGxgeq}. Indeed, by applying Corollary \ref{cor:dimGx} to the homogenous space $\G(x_{0})$, we have
\begin{equation*}
\dim(\G_{x_{0}})\leq \frac{\dim(\G(x_{0}))(\dim(\G(x_{0}))-1)}{2}= \frac{(n-n(x_{0}^{*}))(n-n(x_{0}^{*})-1)}{2} \leq  \frac{(n-1)(n-2)}{2}.
\end{equation*}
Therefore  equality holds in \eqref{eq:dimGxgeq} and in particular $n(x_{0}^{*})=1$. In other words, the quotient space $(M^{*},\sfd^{*},\mm^{*})$ is an $\RCD^{*}(K,N)$-space with a regular point $x_{0}^{*}$ having unique tangent space isometric to $\R$. By the work of Kitabeppu-Lakzian \cite{KL} it follows that $(M^{*},\sfd^{*},\mm^{*})$ is isomorphic as m.m. space either to a closed  interval, or to a real line, or to a closed real half line, or to a circle endowed with a weighted measure equivalent to ${\mathcal L}^{1}$. 

Observe that $\G$ acts effectively and transitively on any principal orbit. Since  $(\G(x_{0}), \sfd_{\G(x_{0})})$ is isometric to  a homogenous Riemannian manifold of dimension $n-1$ with  an $\frac{(n-1)n}{2}$-dimensional compact Lie group of measure-preserving isometries acting on it,  by Kobayashi's Theorem \ref{thm:kobayashi} we infer that   $(\G(x_{0}), \sfd_{\G(x_{0})})$ must be isometric  to  either  $\mathbb{S}^{n-1}$ or $\mathbb{RP}^{n-1}$. By Corollary \ref{co:bundle-s1}, if $(M^*, \sfd^{*})$ is homeomorphic to  a circle then $(M, \sfd)$ is (equivariantly) homeomorphic to a fiber bundle over a circle with fiber equal to  the principal orbit.
 \end{proof}

\begin{thm}\label{thm:DimM*}
Let $(M,\sfd,\mm)$ be an  $\RCD^{*}(K,N)$ m.m. space for some $K\in \R, N \in (1,\infty)$, let $n$  be defined in \eqref{eq:defn} and assume $n\geq 2$. Let $\G$ be a compact Lie group acting effectively and locally Lipschitz and co-Lipschitz-continuously by measure preserving isometries on $(M,\sfd,\mm)$ and denote with $(M^{*},\sfd^{*},\mm^{*})$ the quotient $\RCD^{*}(K,N)$-space. If  
 \begin{equation}\label{eq:assGn-2}
\dim(\G) \geq \frac{(n-2)(n-1)}{2}
\end{equation}
then one and only one of the following three possibilities hold:
\begin{itemize}
\item $\mm^{*}$-a.e. point in $M^{*}$ is regular with unique tangent  space isomorphic to $(\R^{2},\sfd_{E}, {\mathcal L}_{2})$,
\item $\mm^{*}$-a.e. point in $M^{*}$ is regular with unique tangent  space isomorphic to $(\R,\sfd_{E}, {\mathcal L}_{1})$, and in this case the thesis of Theorem \ref{thm:cohom1RCD} holds,
\item $M^{*}$ is a singleton, or equivalently the action of $\G$ is transitive, and in this case  $(M,\sfd,\mm)$  is isomorphic as m.m. space to either ${\mathbb S}^{n}$ or $\mathbb{RP}^{n}$, up to multiplying the measure $\mm$ by a normalizing constant.
\end{itemize}
\end{thm}

\begin{proof}
Let $x_{0}\in M$ be an $n$-regular point  of principal orbit type such that $x_{0}^{*}=\quotient(x_{0})$ is an $n(x_{0}^{*})$-regular point. Combining Theorem \ref{thm:struct} and the assumption \eqref{eq:assGn-2} we know that
\begin{equation}\label{eq:Gx0geqn-2}
\dim(\G_{x_{0}})=\dim(\G)+n(x_{0}^{*})-n  \geq  \frac{(n-2)(n-1)}{2}+n(x_{0}^{*})-n=  \frac{(n-2)(n-3)}{2}+n(x_{0}^{*})-2.
\end{equation}
On the other hand, by applying Corollary \ref{cor:dimGx} to the homogenous manifold $\G(x_{0})$ we have
\begin{equation}\label{eq:Gx0leqn-2}
\dim(\G_{x_{0}})\leq \frac{\dim(\G(x_{0}))(\dim(\G(x_{0}))-1)}{2}= \frac{(n-n(x_{0}^{*}))(n-n(x_{0}^{*})-1)}{2}.
\end{equation}
The combination of \eqref{eq:Gx0geqn-2} and  \eqref{eq:Gx0leqn-2} yields the inequality 
$$
 \frac{(n-2)(n-3)}{2}+n(x_{0}^{*})-2 \leq   \frac{(n-n(x_{0}^{*}))(n-n(x_{0}^{*})-1)}{2}= \frac{(n-2)(n-3)}{2}-\frac{n(x_{0}^{*})-2}{2}(2n-2n(x_{0}^{*})-1),
$$
which in turn gives
$$
(n(x_{0}^{*})-2) (2n-2n(x_{0}^{*})+1) \leq 0.
$$
Since by construction $n(x_{0}^{*})\leq n$, the last inequality implies  that   $n(x_{0}^{*})\leq 2$. Therefore we have just the following three possibilities:
\begin{itemize}
\item $n(x_{0}^{*})=0$. In this case   $M^{*}$ is a singleton, or equivalently the action of $\G$ is transitive;  by the proof of Theorem \ref{thm:Rigidity} it follows that  $(M,\sfd,\mm)$  is isomorphic as m.m. space to either ${\mathbb S}^{n}$ or $\mathbb{RP}^{n}$, up to multiplying the measure $\mm$ by a normalizing constant.
\item $n(x_{0}^{*})=1$. In this case, by \cite{KL}  it follows that  $\mm^{*}$-a.e. point in $M^{*}$ is regular with unique tangent  space isomorphic to $(\R,\sfd_{E}, {\mathcal L}_{1})$. We can then repeat verbatim the proof of Theorem \ref{thm:cohom1RCD} and get the result.
\item the only remaining case is that where $n(x_{0}^{*})= 2$, and this holds if and only if  $\mm^{*}$-a.e. point in $M^{*}$ is regular with unique tangent  space isomorphic to $(\R^{2},\sfd_{E}, {\mathcal L}_{2})$ otherwise we fall into one of the two cases above.
\end{itemize}
\end{proof}

\begin{remark}The class of cohomogeneity one m.m. spaces satisfying $\RCD^*(K,N)$ includes many non-manifold examples. Indeed, by the work of Ketterer \cite{Ketterer}, cones  over $\RCD^*(N-1,N)$-spaces admit  metric-measure structures satisfying  $\RCD^*(K,N+1)$, for any $K \in \R$. Therefore, for instance, Euclidean and spherical cones over homogeneous $N$-dimensional Riemannian manifolds with Ricci curvature bounded below by  $N-1$ are examples of  $\RCD^*(K,N+1)$ spaces admitting isometric actions of cohomogeneity one. Different examples of this type are presented by products of homogenous Riemannian manifolds with a lower Ricci curvature bound with $1$-dimensional $\RCD$-spaces, and more generaly, with cohomogeneity one $\RCD^*(K,N)$-spaces. Regarding this point, recall that $1$-dimensional $\RCD$-spaces were characterized in \cite{KL}. 
\end{remark}

%----------------------------------------------------------------------
%	 SECTION: ORBIFOLDS AND ORBISPACES
%----------------------------------------------------------------------

\section{Orbifolds and orbispaces}\label{Sec:OrbFol}

% SUBUSECTION: ORBIFOLDS
%------------------------------------------------

\subsection{Orbifolds}
An orbifold $\orb$ is, roughly speaking,  a topological space that is locally homeomorphic to
a quotient of $\mathbb{R}^n$ by an orthogonal action of  some finite group. We recall the definitions from \cite{KleinerLott}.

% DEF: LOCAL MODEL

\begin{defn}
\label{D:local_model}
A \emph{local model of dimension $n$} is a pair $(\hat{U}, \G)$, where $\hat{U}$ is an open, connected subset of a Euclidean space $\RR^n$, and $\G$ is a finite group acting smoothly and effectively on $\hat{U}$. 

A \emph{smooth map} $(\hat{U}_1,\G_1)\to (\hat{U}_2,\G_2)$ between local models $(\hat{U}_i,\G_i)$, $i=1,2$, is a homomorphism $\varphi_{_\#}:\G_1\to \G_2$ together with a $\varphi_{_\#}$-equivariant smooth map $\hat{\varphi}:\hat{U}_1\to \hat{U}_2$, i.e. $\hat\vphi(\gamma\cdot \hat u) = \vphi_\#(\gamma) \cdot \hat\vphi(\hat u)$, for all $\gamma \in \G_1$, $\hat u \in \hat U_1$.
\end{defn}

Given a local model $(\hat{U},\G)$, denote by $U$ the quotient $\hat{U}/\G$. A smooth map $\hat{\varphi}:(\hat{U}_1,\G_1)\to (\hat{U}_2,\G_2)$ induces a map $\varphi:U_1\to U_2$. The map $\varphi$ is called an \emph{embedding} if $\hat{\varphi}$ is an embedding. In this case, the effectiveness of the actions in the local models implies that $\varphi_{_\#}$ is injective.

% DEF: (GOOD) LOCAL CHART

\begin{defn}
An \emph{$n$-dimensional orbifold local chart} $(U_x, \hat U_x, \G_x, \pi_x)$ around a point $x$ in a topological space $X$ consists of:
\begin{enumerate}
\item A neighborhood $U_x$ of $x$ in $X$;
\item A local model $(\hat{U}_x, \G_x)$ of dimension $n$;
\item A $\G_x$-equivariant projection $\pi_x:\hat{U}_x\to U_x$, where $\G_x$ acts trivially on $U_x$, that induces a homeomorphism $\hat{U}_x/\G_x\to U_x$.
\end{enumerate}
If $\pi_x^{-1}(x)$ consists of a single point, $\hat{x}$, then $(U_x, \hat U_x, \G_x, \pi_x)$ is called a \emph{good local chart} around $x$.  In particular, $\hat x$ is fixed by the action of $\G_x$ on $\hat{U}_x$. 
\end{defn}

Note that, given a good local chart $(U_x, \hat U_x, \G_x, \pi_x)$ around a point $x$ in a topological space $X$, the quadruple $(U_x, \hat U_x, \G_x, \pi_x)$ is also a local chart, not necessarily good, around any other point $y \in U_x$.  By abusing notation, a local chart $(U, \hat U, \G, \pi)$ will from now on be denoted simply by $U$.

% DEF: ORBIFOLD ATLAS

\begin{defn}
An \emph{$n$-dimensional orbifold atlas} for a topological space $X$ is a collection of $n$-dimensional local charts $\mc{A}=\{U_{\alpha}\}_\alpha$ such that the neighborhoods $U_\alpha \in \mc{A}$ give an open covering of $X$ and:
\newenvironment{Disp}% Definition of Disp environment
{\begin{list}{}{%
    \setlength{\leftmargin}{4mm}}
  \item[] \ignorespaces}
{\unskip \end{list}}
\begin{Disp}
For any $x\in U_{\alpha}\cap U_{\beta}$, there is a local chart $U_\gamma \in \mc{A}$ with $x\in U_{\gamma}\In U_{\alpha}\cap U_{\beta}$ and embeddings $(\hat{U}_{\gamma}, \G_{\gamma})\to(\hat{U}_{\alpha}, \G_{\alpha})$, $(\hat{U}_{\gamma}, \G_{\gamma})\to (\hat{U}_{\beta}, \G_{\beta})$.
\end{Disp}
Two $n$-dimensional atlases are called \emph{equivalent} if they are contained in a third atlas. 
\end{defn}

% DEF: ORBIFOLD

\begin{defn}
An \emph{$n$-dimensional (smooth) orbifold}, denoted by $\orb^n$ or simply $\orb$, is a second-countable, Hausdorff topological space $|\orb|$, called the \emph{underlying topological space} of $\orb$, together with an \emph{equivalence class of $n$-dimensional orbifold atlases}.  
\end{defn}

% DEF: ORDER OF A POINT/ORDER OF AN ORBIFOLD

Given an orbifold $\orb$ and any point $x\in |\orb|$, one can always find a good local chart $U_x$ around $x$.  Moreover, the corresponding group $\G_x$ does not depend on the choice of good local chart around $x$, and is referred to as the \emph{local group at $x$}.  From now on, only good local charts will be considered. In particular, given a good local chart $U_x$ around $x \in |\orb|$, a point $y \in U_x$ and $\hat y \in \pi_x^{-1}(y)$, one can identify the local group $\G_y$ at $y$ with $(\G_x)_{\hat y}$, where $(G_x)_{\hat y}$ is the isotropy of $G_x$ at $\hat y \in \hat{U}_x$. We  define the \emph{order} of a point $x\in |\orb|$ as the order of the local group $\G_x$, and denote it by $\ord(x)$. That is,
\[
\ord(x):= \#\G_x.
\]
We  define the \emph{order of $\orb$} by  
\[
\ord(\orb):=\sup_{x\in \orb}\ord(x).
\]

% DEF: GOOD/BAD ORBIFOLD

If a discrete group $\Gamma$ acts properly discontinuously on a manifold $M$, then the quotient topological space $M/\Gamma$ can be naturally endowed with an orbifold structure that will be denoted by $M//\Gamma$.
An orbifold $\orb$ is \emph{good} (or \emph{developable}) if $\orb = M//\Gamma$ for some manifold $M$ and some discrete group $\Gamma$. A \emph{bad} (or \emph{non-developable}) orbifold is one that is not good.

% DEF: RIEMANNIAN ORBIFOLDS (Following Kleiner & Lott p. 16 Sec. 2.5)

\begin{defn} [Riemannian (resp. Finsler) metric on an orbifold]
A \emph{Riemannian (resp. Finsler)  metric}  on an orbifold $\orb$ is given by a collection of Riemannian (resp. Finsler) metrics on the local models $\hat{U}_\alpha$ so that the following conditions hold:
\begin{enumerate}
	\item The local group $\G_\alpha$ acts isometrically on $\hat{U}_\alpha$.
	\item The embeddings $(\hat{U}_3,\G_3)\to (\hat{U}_1,\G_1)$ and $(\hat{U}_3,\G_3)\to (\hat{U}_2,\G_2)$ in the definition of orbifold atlas are isometric (with respect to the Riemannian metric).
\end{enumerate}
\end{defn}
Let us make the following observations: 
\begin{itemize}
\item Any smooth orbifold admits an orbifold Riemannian (resp. Finsler)  metric.
\item One can assume $g \cdot x=x$ for all $g\in\G_{x}$. Each point $x\in |\orb|$ with $\G_{x}=\{1_\G\}$ is called a \emph{regular point}. The subset $|\orb|_{reg}$ of regular points is called \emph{regular part};  it is a a smooth manifold that forms an open dense subset of $|\orb|$.  A point which is not regular is called \emph{singular}.
\item  Using  normal charts centered at $\hat{x}=\pi_{x}^{-1}(x)$  for the $\G_{x}$ invariant Riemannian structure $(\hat{U}_{x}, g_{x})$, it is possible to assume that $\G_{x}$ is a (possibly trivial) group of linear isometries acting effectively on $(\hat{U}_{x},g_{x})$. 
\item For Riemannian (resp. Finsler) orbifolds, it is possible to choose, shrinking $U_{x}$ if necessary,  $\hat{U}_{x}$ and $U_{x}$ to be geodesically convex.
\end{itemize}
\medskip
In the rest of the section we assume that the orbifold satisfies the above properties. 
\\

Note that the Riemannian (resp. Finsler) structure induces a natural metric $\sfd$ on $|\orb|$ that is locally isometric to the quotient metric of $(\hat{U}_{x},\hat{\sfd}_{x})$ by $\G_{x}$, where $\hat{\sfd}_{x}$ is induced
by the Riemannian (resp. Finsler) structure on $\hat{U}_{x}$. We say that $(\orb,\sfd)$ is a \emph{complete} Riemannian (resp. Finsler) orbifold if it is complete as a metric space. From now on, for simplicity of presentation, we will just consider Riemannian orbifolds; the Finsler case can be carried out analogously but is slightly more involved. 
For any Riemannian  orbifold $\orb$ there is a natural volume  measure $\operatorname{vol}_{\orb}$ given on the local orbifold charts by 
\[
\operatorname{vol}_{\orb}|_{U_{x}}:= \frac{1}{\ord(x)}(\pi_{x})_{\sharp}\operatorname{vol}_{g_{x}},
\]
  where of course $\operatorname{vol}_{g_{x}}$ is the Riemannian volume measure on $(\hat{U}_{x}, g_{x})$.

A continous function $\psi:|\orb|\to \R$ is said to be \emph{smooth} if, for every local model $\hat{U}_{x}$, the lift $\hat{\psi}_{x}:=\pi_{x}^{*} \psi=\psi \circ \pi_{x} :\hat{U}_{x} \to \R$ is smooth.  If $\psi:\orb\to [0,\infty)$  is a smooth  non-negative function,   we  say that $\mm= \psi \operatorname{vol}_{\orb}$ is a \emph{weighted measure}  on the Riemannnian orbifold $\orb$; moreover we will say that $(\orb, \sfd, \mm)$ is a \emph{weighted Riemannian orbifold}.

% DEFINITION: CURVATURE OF A RIEMANNIAN ORBIFOLD

\begin{defn}[Curvature of a Riemannian orbifold] \label{def:CurvOrb}
We say that the Riemannian orbifold $\orb$ has sectional (resp.~Ricci) curvature bounded below by $K\in \mathbb{R}$ if the Riemannian metric on each  local model $\hat{U}_\alpha$ has sectional (resp.~Ricci) curvature bounded below by $K$.  Analogously, given $N\geq n:=$dim$(|\orb|)$, we say that the weighted Riemannian orbifold $(\orb, \sfd, \psi  \operatorname{vol}_{\orb} )$ has (Bakry-\'Emery) $N$-Ricci curvature bounded below by $K$ if every weighted local model $(\hat{U}_{x}, g_{x}, \hat{\psi}_{x} \operatorname{vol}_{g_{x}})$ has $N$-Ricci curvature bounded below by $K$, i.e. if and only if
\begin{align*}
\textrm{Ric}_{g_{x},\hat{\psi}_{x}, N} &:=\textrm{Ric}_{g_{x}}- (N-n) \frac{\nabla^2 \hat{\psi}_{x}^{\frac{1}{N-n}}}{\hat{\psi}_{x}^{\frac{1}{N-n}}} \geq K g_{x}, \quad \text{on } \hat{U}_{x}, & \quad n< N< \infty \\
                                                          &:= \textrm{Ric}_{g_{x}}- \nabla^2(\log \hat{\psi}_{x})   \geq K g_{x},  \qquad \text{on } \hat{U}_{x}, &\quad  N= \infty.
\end{align*}

\end{defn}

 Observe that if $\orb$ has sectional curvature bounded below by $K$, then the associated metric  space $(\orb, \sfd)$ has  curvature bounded below by $K$ in the Alexandrov sense, since the triangle comparison condition is preserved by taking isometric quotients  by  isometric actions of finite groups.  
 
 As recalled above, the set of regular points $|\orb|_{reg}$ is open and dense. Letting $\sfd_{reg} = \sfd|_{|\orb|_{reg}}$ and $\mm_{reg}=\mm\llcorner {|\orb|_{reg}}$, we have that 
 $(|\orb|_{reg},\sfd_{reg},\mm_{reg})$ is  a smooth open weighted  Riemannian manifold; by  density, it is  clear that $(|\orb|_{reg},\sfd_{reg},\mm_{reg})$  has $N$-Ricci $\geq K$ in the usual sense of weighted smooth Riemannian manifolds if and only if   $(\orb, \sfd, \mm)$ has $N$-Ricci  $\geq K$ in the sense of Definition \ref{def:CurvOrb}.
 
% REMARK
 
 \begin{remark}
One can show that the volume form agrees with the $n$-dimensional Hausdorff measure of $(\orb,\sfd)$. It is also possible to define the Busemann-Hausdorff measure, the Holmes-Thomsen measure and the weighted Ricci tensor on Finsler orbifolds in a similar way from the corresponding measure and Ricci tensor of the underlying Finsler manifold (see \cite{Ohta2009}).
\end{remark}
 
The goal of this section is to characterize in a synthetic way the case of (weighted) Ricci curvature lower bounds. To this aim, let us first recall the following well-known equivalences for smooth metric measure spaces (see \cite[Theorem 1.7]{sturm:II}, \cite{BE1985} \cite{BS10}, \cite{ErbarKuwadaSturm}, \cite{AmbrosioMondinoSavare},  \cite{CaMi}).

% LEMMA

\begin{thm}\label{thm:smoothEq}
Let $(M^{n},g, \psi \vol_{g})$ be a weighted Riemannian $n$-manifold with empty
or convex boundary. Then the following equivalences hold:
\begin{align*}
\CD(K,N)\; & \Longleftrightarrow \CD^{*}(K,N)  \Longleftrightarrow \RCD^{*}(K,N) \Longleftrightarrow\;\Ric_{g,\psi,N}\ge K \, g.
\end{align*}
\end{thm}

Before proving the analogous result for orbifolds, let us recall the definition of local curvature-dimension conditions.

% DEF: local CD*(K,N) 
\begin{defn}
[Local $\CD^*(K,N)$ and $\RCD^{*}(K,N)$ conditions]
A metric measure space $(U,\sfd,\m)$ is said to satisfy the \emph{$\CD^{*}(K,N)$-condition
locally at $x$} if there is a neighborhood $V_{x}$ of $x$ such for
any $\mu_{0},\mu_{1}\in\mathcal{P}_{2}^{ac}(U)$ supported
in $V_{x}$  there is a $W_{2}$-geodesic $\{\mu_{t}\}_{t\in[0,1]}\subset\mathcal{P}_{2}^{ac}(U)$
such that \eqref{def:CD*KN} holds. If \eqref{def:CD*KN} holds for
any $W_{2}-$geodesic $\{\mu_{t}\}_{t\in[0,1]}$ between measures $\mu_{0},\mu_{1}\in\mathcal{P}_{2}(U)$
with support in $V_{x}$ then we say that the \emph{strong $\CD^{*}(K,N)$-condition
holds locally at $x$}. If moreover, for every $f,g \in W^{1,2}(U,\sfd,\m)$ supported in $V_{x}$ the parallelogram identity \eqref{eq:paralrule} holds, then we say that  $(U,\sfd,\m)$ satisfies the \emph{$\RCD^{*}(K,N)$-condition
locally at $x$}.
\end{defn}}

% THEOREM - Equivalence of curvature conditions

\begin{thm}
\label{THM:ORBI_EQUIV}
Let $(\orb,\sfd, \mm)$ be a weighted Riemannian orbifold and let $N\in [1,\infty]$. Then the following equivalences hold: %
\begin{align*}
\CD(K,N) & \Longleftrightarrow \CD^{*}(K,N)   \Longleftrightarrow \RCD^{*}(K,N) \quad \text{(resp. $\RCD(K,\infty)$ in case $N=\infty$)}\\
 & \Longleftrightarrow  (\orb, \sfd, \mm) \quad \text{has  $N$-Ricci curvature bounded below by $K$ in the sense of Definition \ref{def:CurvOrb}}.
\end{align*}
\end{thm}

% PROOF

\begin{proof} 
 We will give the arguments for $N\in [1,\infty)$, the case $N=\infty$ can be proved verbatim.   First of all, the implications $ \CD(K,N) \Longrightarrow \CD^{*}(K,N)$ and  $ \RCD^{*}(K,N) \Longrightarrow  \CD^{*}(K,N)$ are trivial.  We will show that  $(\orb, \sfd, \mm)$ has  $N$-Ricci $\geq K$ $\Longrightarrow \RCD^{*}(K,N)$,   $\CD^{*}(K,N)\Longrightarrow$ $(\orb, \sfd, \mm)$ has  $N$-Ricci $\geq K$,  and  that   $(\orb, \sfd, \mm)$ has  $N$-Ricci $\geq K$ $\Longrightarrow$ $\CD(K,N)$.
 \\

\noindent \emph{$(\orb, \sfd, \mm)$ has $N$-Ricci $\geq K \Longrightarrow \RCD^{*}(K,N)$}\vspace{.1cm}

By assumption, every local model $(\hat{U}_{x},\hat{\sfd}_x,\hat{\m}_{x})$ is a weighted Riemannian manifold
with convex boundary and with $N$-Ricci $\geq K$. Thus, by Theorem \ref{thm:smoothEq},  it is an $\RCD^{*}(K,N)$-space.
Theorem~\ref{thm:M*RCD} implies that the quotient space $U_{x}=\hat{U}_{x}/\G_{x}$
is also a $\RCD^{*}(K,N)$-space. Thus $(\orb, \sfd, \mm)$  satisfies $\RCD^{*}(K,N)$
locally; the thesis then follows by the local-to-global property of $\RCD^{*}(K,N)$ proved independently in \cite[Theorem 7.8]{AmbrosioMondinoSavareLocGlob} and \cite[Theorem 3.25]{ErbarKuwadaSturm}.
\\

\noindent\emph{$\CD^{*}(K,N) \Longrightarrow$   $(\orb, \sfd, \mm)$ has $N$-Ricci $\geq K$}\vspace{.1cm}

Since by  construction $(U_{x},\sfd)$ is geodesically convex, the assumption that $(\orb, \sfd, \mm)$ satisfies the $\CD^{*}(K,N)$-condition implies that  $(U_{x},\sfd , \mm\llcorner U_{x})$ satisfies $\CD^{*}(K,N)$ as well, for all $x \in |\orb|$. Recalling that, for
all regular points $x\in |\orb|_{reg}$, $(U_{x},\sfd, \mm\llcorner U_{x})$
is isomorphic as m.m.s  to $(\hat{U}_{x}, \hat{\sfd}_{x},\hat{\mm}_{x})$, we get that  $(\hat{U}_{x}, \hat{\sfd}_{x},\hat{\mm}_{x})$ satisfies $\CD^{*}(K,N)$ as well. Since by construction also  $\hat{U}_{x}$ is geodesically convex,  we can apply Theorem \ref{thm:smoothEq} and infer that $(\hat{U}_{x}, g_{x},\hat{\mm}_{x})$  has $N$-Ricci $\geq K$. This gives the claim for $x\in |\orb|_{reg}$.

Let now  $x \in |\orb|$ be arbitrary and assume, for the sake of contradiction, that $\Ric_{N}<K$ for some $y\in {U}_{x}$. Since $|\orb|_{reg}$ is open and dense in $|\orb|$, and since the condition $\Ric_{N}<K$ is open, it follows that 
there is a regular point $z \in U_{x}$ such that $\Ric_{N}<K$
at $\hat{z}\in\pi_{x}^{-1}(z)$. This contradicts the first part of the argument, as   the condition $\Ric_{N}\ge K$ at all regular points is independent
of the chart.
\\

\noindent\emph{$(\orb, \sfd, \mm)$ has $N$-Ricci $\geq K \Longrightarrow \CD(K,N)$.  }\vspace{.1cm}

 From the arguments above we know that $(\orb, \sfd, \mm)$ is $\RCD^{*}(K,N)$.  The fact that $\RCD^{*}(K,N)$ implies  $\CD(K,N)$ is proved  for general metric measure spaces in \cite{CaMi}. For the reader's convenience, we give here an independent, more self-contained, argument tailored to the orbifold case.

It is well-known (see, e.g. \cite[Proposition 15]{Borzellino1993MaxDiam}) that the set of regular points $|\orb|_{reg}$ is open, dense and geodesically convex. Called $\sfd_{reg} = \sfd|_{|\orb|_{reg}}$ and $\mm_{reg}=\mm\llcorner {|\orb|_{reg}}$, we then have that 
 $(|\orb|_{reg},\sfd_{reg},\mm_{reg})$ is isometric to
an open, geodesically convex,  weighted  Riemannian manifold $(\tilde{|\orb|},\tilde{\sfd},\tilde{\m})$ having  $N$-Ricci $\geq K$. Therefore, by \cite[Theorem 1.7]{sturm:II}, for every  $\mu_{0}, \mu_{1} \in \mathcal{P}_{2}^{ac} (|\orb|)$ supported in $|\orb|_{reg}$, there exists a   $W_{2}$-geodesic $\{\mu_{t}\}_{t \in [0,1]}$ satisfying the $\CD(K,N)$ convexity condition.

To conclude the proof,  let $\mu_{0}, \mu_{1} \in \mathcal{P}_{2}^{ac}(|\orb|)$ be arbitrary. Since $\RCD^{*}(K,N)$ is fulfilled on the whole $(\orb, \sfd, \mm)$, by \cite{GRS2016}, we know there is a unique dynamical 2-optimal plan $\nu\in \textrm{GeoOpt}(\mu_{0},\mu_{1})$ and it is induced by a map $F:|\orb| \to \textrm{Geo}(|\orb|)$ which is $\mu_{0}$-a.e. well defined and $\mu_{0}$-essentially injective.
\\It follows that we can write $\nu=\sum_{n \in \N} \nu^{n}$, where the $\nu^{n}$ satisfy the following properties:
\begin{itemize}
\item $(\ee_{0})_{\sharp} \nu^{n}$ and  $(\ee_{1})_{\sharp} \nu^{n}$ are absolutely continuous with respect to $\textrm{vol}_{\orb}$;
\item for every $n\in \N$, letting $c_{n}:=\nu^{n}\big(\textrm{Geo}(|\orb|)\big)>0$,  we have that    $\bar{\nu}^{n}:=\frac{1}{c_{n}} \nu^{n}$ is (the unique) dynamical  2-optimal   plan from $\bar{\mu}_{0}^{n}:=(\ee_{0})_{\sharp} \bar{\nu}^{n}$ to $\bar{\mu}_{1}^{n}:=(\ee_{1})_{\sharp} \bar{\nu}^{n}$;
\item  for every $t \in [0,1]$, the measure $\bar{\mu}_{t}^{n}:=(\ee_{t})_{\sharp} \bar{\nu}^{n}$ is supported in $|\orb|_{reg}$. In particular, by the first part of the argument, the geodesic $\{\bar{\mu}_{t}^{n}\}_{t \in [0,1]}$ satisfies the $\CD(K,N)$ convexity condition.
\end{itemize}
Using that the map $F$ above is $\mu_{0}$-essentially injective, summing up all the $\CD(K,N)$ convexity conditions of $\{\mu_{t}^{n}:=c_{n} \bar{\mu}_{t}^{n}\}_{t \in [0,1]}$ over all $n \in \N$ we conclude that $(\ee_{t})_{\sharp} \nu=\mu_{t}=\sum_{n \in \N}  \mu^{n}_{t}$ satisfies the  $\CD(K,N)$ convexity condition as well, giving the thesis.
\end{proof}

In the next result we generalize Cheng's Maximal Diameter Theorem to arbitrary weighted orbifolds  with $N$-Ricci $\geq N-1$; let us mention  that the generalization to a \emph{good} Riemannian orbifold with Ricci $\geq n-1$ (i.e. a quotient of a Riemannian $n$-manifold with Ricci $\geq n-1$)  of Cheng's Maximal Diameter Theorem was established by Borzellino \cite[Theorem 1]{Borzellino1993MaxDiam}. 

% THM: Cheng's maximal diameter theorem

\begin{thm}\label{thm:Cheng}
Let $(\orb,\sfd,\mm)$ be a weighted Riemannian $n$-dimensional orbifold with  $N$-Ricci $\geq N-1$. Then ${\rm diam} (\orb)\leq \pi$ 
 and equality is achieved if and only if $(\orb,\sfd,\mm)$ is isomorphic as a m.m. space to a spherical suspension over a smooth quotient $\Sigma:={\mathbb S}^{n-1}/\Gamma$ of ${\mathbb S}^{n-1}$ under a finite  group of isometries $\Gamma$, endowed with a  weighted measure $\mm_{\Sigma}$ so that $({\mathbb S}^{n-1}/\Gamma, \sfd_{{\mathbb S}^{n-1}/\Gamma}, \mm_{\Sigma})$ has $N-1$-Ricci curvature bounded below by $N-2$. In other terms, 
$$(\orb,\sfd,\mm) \simeq  [0,\pi]\times_{\sin}^{N-1} ({\mathbb S}^{n-1}/\Gamma, \sfd_{{\mathbb S}^{n-1}/\Gamma}, \mm_{\Sigma}).$$
Note that, in particular, the only two (possibly) singular points are the ones achieving the maximal distance.

If, moreover, $\mm={\rm vol}_{\orb}$ ( i.e. if $\orb$ is an $n$-dimensional Riemannian orbifold with {\rm Ricci} $\geq n-1$) with diameter $\pi$,  then  it is isomorphic as a m.m. space  to a quotient of ${\mathbb S}^{n-1}$ by  a finite  group of isometries.
\end{thm} 

% PROOF

\begin{proof}
Since by Theorem \ref{THM:ORBI_EQUIV}  we know that   $(\orb,\sfd,\mm)$ is an $\RCD^{*}(N-1,N)$-space, by the Maximal Diameter Theorem in $\RCD^{*}(N-1,N)$-spaces proved by Ketterer \cite{Ketterer}, we know that ${\rm diam}(\orb)\leq \pi$ and if equality is achieved then there exists an $\RCD^{*}(N-2,N-1)$-space $(\Sigma, \sfd_{\Sigma},\mm_{\Sigma})$ such that  $(\orb,\sfd,\mm)$ is isomorphic  as a m.m. space to the spherical suspension $[0,\pi]\times_{\sin}^{N-1} (\Sigma, \sfd_{\Sigma}, \mm_{\Sigma})$. 

The orbifold structure implies that at every $x \in \orb$, the tangent cone to $(\orb,\sfd,\mm)$  is unique and isometric to $\R^{n}/\G_{x}$, where $\G_{x}$ is a finite subgroup of $\mathrm{O}(n)$. Let $p_{0}=\{0\}\times \Sigma$, $p_{\pi}=\{\pi\}\times \Sigma \in \orb$ be the two vertices of the above spherical suspension. It is easily seen that the tangent cones at $p_{0}$ and $p_{\pi}$ are metric-measure cones with cross section $(\Sigma, \sfd_{\Sigma}, \mm_{\Sigma})$. It follows that $(\Sigma, \sfd_{\Sigma})$ is isometric to ${\mathbb S}^{n-1}/\G_{p_{0}}\simeq {\mathbb S}^{n-1}/\G_{p_{\pi}}$. Moreover, the orbifold assumption implies that  $(\Sigma, \sfd_{\Sigma}, \mm_{\Sigma})$ is a smoooth weighted Riemannian manifold. Indeed, if $\Sigma$ had a singular point $x$,  then all the segment $[0,\pi] \times \{x\}$  would be made of singular points, contradicting the orbifold structure of $\orb$. This proves the first part of the statement. 

For the last claim, just observe that if $\mm={\rm vol}_{\orb}$, then  $(\Sigma, \sfd_{\Sigma},\mm_{\Sigma}) \simeq ({\mathbb S}^{n-1}/\G_{p_{0}}, \sfd_{{\mathbb S}^{n-1}/\G_{p_{0}}}, {\rm vol}_{{\mathbb S}^{n-1}/\G_{p_{0}}})$. 
Extending the action of $\G_{p_{0}}$ fiberwise to $\mathbb{S}^{n}=[0,\pi] \times_{\sin}^{n-1} \mathbb{S}^{n-1}$, we get that 
\[
(\orb,\sfd,\mm) \simeq ({\mathbb S}^{n}/\G_{p_{0}}, \sfd_{{\mathbb S}^{n}/\G_{p_{0}}}, {\rm vol}_{{\mathbb S}^{n}/\G_{p_{0}}}),
\]
as desired.
\end{proof}
% PROP: Bishop's inequality
In what follows, $\mathbb{M}_{K}^{n}$ denotes the simply connected $n$-dimensional space form of sectional curvature equal to $K\in \R$.  The  next result was proved via an independent argument by Borzellino \cite[Proposition 20]{Borzellino1993MaxDiam}.

\begin{prop}[Bishop  inequality] \label{prop:Bishop}
 If $(\orb,\sfd,\vol)$ is an $n$-dimensional
orbifold with {\rm Ricci} $\ge K$, then 
\[
\vol_{\orb}(B_{r}(x))\le\frac{1}{\ord(x)}\vol_{\mathbb{M}_{K}^{n}}(B_{r}^{\mathbb{M}_{K}^{n}}).
\]
\end{prop}
\begin{proof}
By the definition of the volume form on an orbifold, we have
\[
\lim_{r\to0}\frac{\vol_{\orb}(B_{r}(x))}{r^{n}}=\frac{1}{\ord(x)}.
\]
Furthermore, 
\[
\frac{\vol_{\mathbb{M}_{K}^{n}}(B_{r}^{\mathbb{M}_{K}^{n}})}{r^{n}}\to1.
\]
Since the $\CD(K,N)$-condition implies Bishop-Gromov volume comparison (see \cite{sturm:II}),
we get 
\[
\frac{\vol_{\orb}(B_{R}(x))}{\vol_{\mathbb{M}_{K}^{n}}(B_{R}^{\mathbb{M}_{K}^{n}})}\le\lim_{r\to0}\frac{\vol_{\orb}(B_{r}(x))}{\vol_{\mathbb{M}_{K}^{n}}(B_{r}^{\mathbb{M}_{K}^{n}})}=\frac{1}{\ord(x)}
\]
which proves the claim.
\end{proof}

% DEF: MAXIMAL VOLUME

\begin{defn}
[Maximal Volume] An $n$-dimensional orbifold with Ricci $\ge K$ is
said to have \emph{maximal volume} if 
\[
\lim_{r\to\infty}\frac{\vol_{\orb}(B_{r}(x))}{\vol_{\mathbb{M}_{K}^{n}}(B_{r}^{\mathbb{M}_{K}^{n}})}=\frac{1}{\ord(x)}. 
\]
In case $\orb$ (resp. $\mathbb{M}_{K}^{n}$) is bounded, the limit above has to be read as the eventual value of the ratio.
\end{defn}

% COROLLARY

\begin{thm}\label{cor:OrbRig}
If an $n$-dimensional orbifold  $(\orb,\sfd,\m)$ with {\rm Ricci} $\ge K$ has maximal volume,
then it is isometric to 
\[
\mathbb{M}_{K}^{n}/\G_{x},
\]
where $x\in \orb$ is a point with $\ord(\orb)=\ord(x)$. In case $K>0$
and $M$ is not a manifold, i.e. $\ord(\orb)>1$, there are exactly  two singular points. If $K\le0$
and $\ord(\orb)>1$ then there is  exactly  one singular point.
\end{thm}

% PROOF

\begin{proof} 
Let us first consider the case $K>0$. Up to a rescaling we can assume $K=n-1$, where $n=$ dim$(\orb)$. The Bishop inequality  Proposition \ref{prop:Bishop}, combined with the maximal volume assumption, implies that diam$(\orb)=\pi$. The thesis follows then from Theorem \ref{thm:Cheng}.

Let now $K\leq 0$.  The maximal volume assumption combined with the Bishop inequality  (Proposition \ref{prop:Bishop}) gives that 
\[
\frac{\vol_{\orb}(B_{r}(x))}{\vol_{\mathbb{M}_{K}^{n}}(B_{r}^{\mathbb{M}_{K}^{n}})}=\frac{1}{\ord(x)}
\] 
for every $r>0$. Since, 
by Theorem \ref{THM:ORBI_EQUIV},  we know that   $(\orb,\sfd,\mm)$ is an $\RCD^{*}(K,n)$-space, we can apply the ``volume cone implies metric cone'' Theorem proved in  $\RCD^{*}(K,n)$-spaces by Gigli and De Philippis \cite{GdP2015} and infer that $(\orb,\sfd,\mm)$ is isomorphic as a m.m. space to a metric-measure (hyperbolic, in case $K<0$) cone over an $\RCD^{*}$ space $(\Sigma, \sfd_{\Sigma},\mm_{\Sigma})$. The end of the argument is then analogous to the proof of Theorem \ref{thm:Cheng}.
\end{proof}

% SUBSECTION: ORBISPACES
%------------------------------------------

\subsection{Orbispaces}
The proof of the $\CD^{*}(K,N)$-condition did not require the full
structure theory of orbifolds. Indeed, a similar proof holds for the
following more general class of  spaces. An \emph{orbispace} is obtained by replacing the local $n$-dimensional models in the definition of
orbifolds by \emph{local models} defined as pairs   $(\hat{U},\G)$,
where $\hat{U}$ is a topological space and $\G$ a finite group of
homeomorphisms acting effectively on $\hat{U}$. The concepts of maps
of local models, charts at points and atlases are defined in a similar
way. At each point $x\in\orb$ we obtain good charts $(\hat{U}_{x},U_{x},\G_{x},\pi_{x})$
in the sense $\G_{x}$ acts effectively on $\hat{U}_{x}$ and $\pi^{-1}(x)$
contains exactly one point $\hat{x}\in\hat{U}_{x}$, i.e. $\G_{x}$
fixes $\hat{x}$. Furthermore, for any $y\in U_{x}$ and good chart
$(\hat{U}_{y},U_{y},\G_{y},\pi_{y})$ at $y$, the group $\G_{y}$
can be identified with $(\G_{x})_{\hat{y}}$ where $\hat{y}\in\pi_{x}^{-1}(y).$
Thus, up to isomorphism, for each $x\in |\orb|$ the group $\G_{x}$
and $\ord(x)=\#\G_{x}$ are well-defined and it is easy to see that
the set $\orb_{reg} = \{ x\in \orb \mid  \ord(x)=0\}$ of regular points is open and
dense.

% DEF: ORBISPACE

\begin{defn}[Metric Orbispace] 
We say that $(\orb,\sfd)$ is a \emph{metric orbispace}
if for each good chart $(\hat{U}_{x},U_{x},\G_{x},\pi_{x})$,
there is a metric $\hat{\sfd}$ on $\hat{U}_{x}$ such that $\G_{x}$
acts isometrically on $(\hat{U}_{x},\hat{\sfd})$ and the metric $\sfd$
agrees with the quotient metric, i.e. for all $y,z\in U_{x}$ and
$\hat{y}\in\pi_{x}^{-1}(y),\hat{z}\in\pi_{x}^{-1}(z)$, the following holds:
\[
\sfd(y,z)=\inf_{g,h\in\G_{x}}\hat{\sfd}(\hat{y},\hat{z}).
\]
Furthermore, the embeddings $(\hat{U}_{3},\G_{3})\to(\hat{U}_{1},\G_{1})$
and $(\hat{U}_{3},\G_{3})\to(\hat{U}_{2},\G_{2})$ in the definition
of atlas are isometric embeddings (in the the metric sense, i.e. they preserve distances). 
\end{defn}

% REMARK

\begin{remark}
This definition does not take into account the global nature of a
metric. However, if each local chart is locally geodesic then there
is a unique global geodesic metric on $|\orb|$. Recall that $(U,\sfd)$ is \emph{locally geodesic} at $x$ if there is a neighborhood
$V_x$ of $x$ such that each point $y,z\in V_x$ is connected by a geodesic
lying in $U$. 
\end{remark}

Let $\m$ be a $\sigma$-finite, non-negative Borel measure on $(|\orb|,\sfd)$.
We call $(\orb,\sfd,\m)$ a \emph{metric measure orbispace}. As in the
case of orbifolds, it is possible to locally lift $\m$ in each chart,
i.e.~given a local chart $(\hat{U},U,\G,\pi)$ there is a unique $\G$-invariant
measure $\m_{\hat{U}}$ on $\hat{U}$ such that $ \m \llcorner U = \frac{1}{\#\G}  \; \pi_{\sharp} \,\m_{\hat{U}}$.
Note that in general, the metric embeddings in the definition of atlas are
not measure preserving. 

% LEM: regular points are open

\begin{lem}
Let $(\orb,\sfd,\m)$ be a metric measure orbispace and assume
that, for each local chart $(\hat{U},U,\G,\pi)$, the set of fixed
points $\operatorname{Fix}_{\hat{U}}(g)$ in $\hat{U}$ of each element $g\in\G$, $g\neq 1_{\G}$, has zero measure w.r.t.
the lift $\m_{\hat{U}}$. Then the measure $\m$ is concentrated on
the set $|\orb_{reg}|$ of regular points.
\end{lem}
\begin{proof}
Let $\m_{U}$ be the lift of a chart $(\hat{U},U,\G,\pi)$. Note that 
for $x\in U$, $\ord(x)>1$ if and only if $\G_{x}\ne \{1_{\G}\}$. Hence, we have
\begin{align*}
\pi^{-1}(\{\ord>1\}\cap U) & =\pi^{-1}(\{\ord>1\})\cap\hat{U}\\
 & =\{\hat{x}\in\hat{U}\,|\,\exists g\in\G, \, g \neq 1_{\G}\, :\hat{x}\in\operatorname{Fix}_{\hat{U}}(g)\}\\
 & =\bigcup_{g\in\G,\, g\neq 1_{\G}}\operatorname{Fix}_{\hat{U}}(g).
\end{align*}
By assumption $\m_{\hat{U}}(\operatorname{Fix}_{\hat{U}}(g))=0$ for
all $g\in\G$. Thus $\m(\{\ord>1\}\cap U)=0$. By $\sigma$-finiteness
we see that $\m(\{\ord>1\})=0$. 
\end{proof}

By Theorem \ref{thm:CDbyGisCD} we immediately obtain the following.

% LEM: local CD*(K,N) for chart charts

\begin{lem}
Let $(\orb,\sfd,\m)$ be a metric measure orbispace and $x\in\orb$.
If there is a metric measure chart $(\hat{U}_{x},\hat{\sfd}_{x},\m_{x})$
at $x$ that satisfies the strong $\CD^{*}(K,N)$-condition at $\pi_{x}^{-1}(x)$,
then $(\orb,\sfd,\m)$ satisfies the strong $\CD^{*}(K,N)$-condition
locally at $x$.
\end{lem}

Combining this with the local-to-global property of the $\CD^{*}(K,N)$-condition under essential non-branching \cite{BS10, CMMapsNB}
we obtain the following. 

% COR: CD*(K,N) from chart to global

\begin{thm}\label{thm:StabMetricOrbSpa}
Let  $(\orb,\sfd,\m)$ be an essentially non-branching   m.m. orbispace  and assume  that for each $x \in \orb$ there is  a chart at $x$ such that the local model satisfies the strong
$\CD^{*}(K,N)$-condition locally at $\pi_{x}^{-1}(x)$. Then $(\orb,\sfd,\m)$
satisfies the strong $\CD^{*}(K,N)$-condition. 
\end{thm}

Analogously, by using the local-to-global property of $\RCD^{*}(K,N)$ proved independently in \cite[Theorem 7.8]{AmbrosioMondinoSavareLocGlob} and \cite[Theorem 3.25]{ErbarKuwadaSturm}, we have the next result.

% THM

\begin{thm}\label{thm:RCDorbispaces}
Let  $(\orb,\sfd,\m)$ be a metric measure orbispace  and assume  that for each $x \in \orb$ there is  a chart at $x$ such that the local model satisfies  the
$\RCD^{*}(K,N)$-condition locally at $\pi_{x}^{-1}(x)$. Then $(\orb,\sfd,\m)$
satisfies the  $\RCD^{*}(K,N)$-condition. 
\end{thm}

% SUBSECTION: APPLICATION TO ALMOST FREE DISCRETE GROUP ACTIONS 
\subsubsection{An application to almost free discrete group actions }

In this subsection we assume that $\G$ is an infinite discrete closed subgroup
of the group of isometries acting measure-preserving on the proper metric measure
space $(M,\sfd,\m)$. Note that this implies that the action is effective. Furthermore, we assume that $\G$ acts \emph{almost freely},
i.e. for each $x\in M$ the isotropy subgroup 
\[
\G_{x}=\{g\in\G\,|\,g\cdot x=x\}
\]
is finite. We say $x\in M$ is a \emph{regular point} if $\G_{x}=\{1_{\G}\}$. 
Under our assumptions the action of $\G$ has closed orbits.

% LEM: 
\begin{lem}
The set of regular points is an open subset of $M$.
\end{lem}

% PROOF

\begin{proof}

Let $x_{n}\to x$ and $g_{n}\cdot x_{n}=x_{n}$. Since $(M,\sfd)$ is
proper and
\begin{align*}
\sfd(g_{n}\cdot y,x) & \le \sfd(g_{n}\cdot y,g_{n}\cdot x_{n})+\sfd(g_{n}\cdot x_{n},x)\\
 & \le \sfd(y,x)+\sfd(x_{n},x),
\end{align*}
we see that $\{g_{n}\}_{n=1}^{\infty}$ is precompact in the space
of isometries. In particular, because $\G$ is a closed subgroup, there
is a $g\in\G$ such that, up to passing to a subsequence, $g_{n}\to g$
with $g\cdot x=x$. 

Since $\G$ is discrete we have that $g_{n}=g$ for every sufficiently
large $n$. Therefore, if $\{x_n\}_{n\in \mathbb{N}}$ is a sequence of points which are not regular, i.e. $\G_{x_{n}}\ne\{1_{\G}\}$ then we can pick $g_{n}\neq 1_{\G}$.  This implies that $\G_{x}\ne\{1_\G\}$. In particular, the set of regular
points is open.
\end{proof}

The assumptions on $(M,\sfd,\m)$ and $\G$ imply that the quotient space $(M^{*},\sfd^{*})$ is a proper metric space and for each $x^{*},y^{*}\in M^{*}$ and $x\in\quotient^{-1}(x^{*})$ there is a $y\in\quotient^{-1}(y^{*})$
with $\sfd(x,y)=\sfd^{*}(x^{*},y^{*})$.

% LEM: Local version of stabilizer theorem in the discrete setting

\begin{lem}\label{lem:epsx}
For all $x\in M$ there is an $\epsilon=\epsilon_{x}>0$ such that
for all $y\in B_{\epsilon}(x)$ and $g\in\G\backslash\G_{x}$,
\begin{align*}
\sfd(x,y)  =\sfd^{*}(x^{*},y^{*}), \qquad  \sfd(g\cdot x,y)  >\sfd^{*}(x^{*},y^{*}), \qquad  \inf_{g \in \G\setminus \G_{x}} \sfd(x, g\cdot x) > \epsilon_{x},
\end{align*}
  and 
\[
\G_{y}\le\G_{x}.
\]
\end{lem}
\begin{proof}
We first show that the orbit of $x$ has only isolated points. To this aim  let $x_{n}=g_{n}\cdot x$ such that $x_{n}\to x'=g'\cdot x$. Then  $g_{n}'\cdot x\to x$ for
$g_{n}'=(g')^{-1}g_{n}$. As above, this means that $(g_{n}')_{n\in\mathbb{N}}$
is precompact and with limit in $\G_{x}$. Hence, it is eventually
constant and thus $(x_{n})$ is eventually constant. In particular,
the orbit of $x$ is discrete and there is an $\epsilon>0$ such that
$B_{2\epsilon}(x)\cap\G(x)=\{x\}$. 

Recall the definition of the quotient metric 
\[
\sfd^{*}(x^{*},y^{*})=\inf_{g\in\G}\sfd(g\cdot x,y).
\]
If $y\in B_{\epsilon}(x)$ then we see 
\[
\sfd^{*}(x^{*},y^{*})=\sfd(x,y)
\]
and 
\[
\sfd(g\cdot x,y)>\sfd(x,y)
\]
for all $g\notin\G_{x}$.  Finally, if $g\in\G_{y}$, then 
\begin{align*}
\sfd(x,y) & = \sfd(g\cdot x,y),
\end{align*}
which implies that $g\in\G_{x}$.
\end{proof}

% THM: quotients are metric orbispaces

\begin{thm} \label{thm:Applorbispaces}
Let $\epsilon_{x}>0$ be given by Lemma \ref{lem:epsx}. Then $(B_{\frac{\epsilon_{x}}{4}}(x^{*}),\sfd^{*})$
is isomorphic to the quotient of $(B_{\frac{\epsilon_{x}}{4}}(x),\sfd)$
by $\G_{x}$. In particular, $(M^{*},\sfd^{*})$ is a metric orbifold
with good  local charts 
\[
\Big\{\big((B_{\frac{\epsilon_{x}}{4}}(x),\sfd),(B_{\frac{\epsilon_{x}}{4}}(x^{*}),\sfd^{*}),\G_{x},\quotient \big) \Big\}_{x\in M}.
\]
\end{thm}

% PROOF

\begin{proof}
Let $y,z\in B_{\frac{\epsilon_{x}}{4}}(x)$.  Then, for $g\in\G\backslash\G_{x}$,
\begin{align*}
\sfd(y,g\cdot z)	& \ge \sfd(x, g \cdot x)-\sfd(x,y) -\sfd(g \cdot x,g\cdot z)\\
			& \geq  \sfd(x,g\cdot x)- \frac{ \epsilon_{x}}{2}\\
			&>  \frac{ \epsilon_{x}}{2}.
\end{align*}
Since $B_{\frac{\epsilon_{x}}{4}}(x)$ is $\G_{x}$-invariant, 
\[
\inf_{g\in\G_{x}} \sfd(y,g\cdot z)<\frac{\epsilon_{x}}{2}
\]
implying
\[
\sfd^{*}(y^{*},z^{*})=\inf_{g\in\G_{x}} \sfd(y,g\cdot z).
\]
In other words, $(B_{\frac{\epsilon_{x}}{4}}(x^{*}),\sfd^{*})$ is isometric to
the quotient of $(B_{\frac{\epsilon_{x}}{4}}(x),\sfd)$ by $\G_{x}$. 
\end{proof}
Now, similarly to the Riemannian orbifold volume form, it is possible
to push-down the measure $\m$.

We next define a (unique) measure $\m^{*}$ on  $(M^{*},\sfd^{*})$, by specifying its restriction on  $B_{\frac{\epsilon_{x}}{4}}(x^{*})$:
\[
\m^{*}\llcorner B_{\frac{\epsilon_{x}}{4}}(x^{*}) =  \frac{1}{\#\G_{x}} \; \quotient_{\sharp} \left(\m \llcorner B_{\frac{\epsilon_{x}}{4}}(x)\right).
\]
In the following we call $\m^{*}$ the \emph{natural quotient measure} of
$(M,\sfd,\G)$.

Combining  Theorem \ref{thm:RCDorbispaces} and  Theorem \ref{thm:Applorbispaces} we obtain the next result.

%THM: CD*(K,N) is stable under quotients by discrete almost free group actions

\begin{thm}\label{thm:discreteActRCD}
Let  $(M,\sfd,\m)$  be a m.m. space satisfying the $\RCD^{*}(K,N)$-condition.
Let $\G$ be a discrete closed subgroup of the measure-preserving isometry group acting almost freely on $(M,\sfd,\m)$. Then $(M^{*},\sfd^{*},\m^{*})$ is
a  metric measure orbispace satisfying the $\RCD^{*}(K,N)$-condition.
\end{thm}
Via a two-step procedure (i.e. first consider the quotient by $\G/\mathsf{L}$ and then by $\mathsf{L}$) we obtain the following general result for non-compact Lie groups.

% COR: Combination of compact quotient and discrete almost free quotient

\begin{cor}\label{cor:discreteActRCD}
Let $\G$ be a non-compact Lie group admitting a normal lattice
$\mathsf{L}$, i.e.~a discrete normal subgroup $\mathsf{L}$ such
that $\G/\mathsf{L}$ is compact. Let $\G$ and $\mathsf{L}$
be closed subgroups of the  measure-preserving isometry group  acting almost freely on a metric measure
space $(M,\sfd,\m)$ that satisfies the $\RCD^{*}(K,N)$-condition.   Then there is a quotient measure $\m^{*}$ on the quotient metric space $(M^{*},\sfd^{*})$ making
$(M^{*},\sfd^{*},\m^{*})$  an $\RCD^{*}(K,N)$-space.
\end{cor}

In the rest of the subsection we will prove that, under natural assumptions, the set of regular points has full measure.

The following  lemma was observed by the fourth author \cite{Sosa2016} and, independently,
by Guijarro\textendash Santos-Rodr{\'i}guez \cite{GS2016IsoRCD}. We present an independent argument for the reader's convenience. Before stating the result, recall that the \textit{good transport behavior} property was introduced in Definition \ref{def:GTB} and that it includes strong $\CD^{*}(K,N)$-spaces and essentially non-branching
$\mathsf{MCP}(K,N)$-spaces, $N\in[1,\infty)$, as well as $\RCD^{*}(K,\infty)$-spaces.

% LEM: Fixed point set has zero measure.
\begin{lem}
Let $(M,\sfd,\m)$ be a m.m.\ space having good transport
behavior $\GTB_{p}$ for some $p\in (1,\infty)$. Then, for every non-trivial isometry $\varphi:(M,\sfd)\to(M,\sfd)$,
the set of fixed points $\operatorname{Fix}(\varphi)$ has zero $\m$-measure.
\end{lem}
\begin{proof}
Since $\varphi$ is not the identity we can find $\bar{x}\in M$ so that $\varphi(\bar{x})\neq \bar{x}$.
Assume, for the sake of contradiction, that there is a compact set $A\subset\operatorname{Fix}(\varphi)$
with $\m(A)\in(0,\infty)$. Let 
\[
\mu=\frac{1}{\m(A)}\m \llcorner A, \qquad \nu=\frac{1}{2}\delta_{\bar{x}}+\frac{1}{2}\delta_{\varphi(\bar{x})}.
\]
Let $\pi_{1}$ be a $p$-optimal coupling for $(\mu, \delta_{\bar{x}})$ and  $\pi_{2}$  a $p$-optimal coupling for $(\mu, \delta_{\varphi(x)})$. The assumed $\GTB_{p}$ implies that  both couplings are induced by a transport
map.  Note that $\pi:=\frac{1}{2}\pi_{1}+\frac{1}{2}\pi_{2}$ is a
coupling for $(\mu, \nu)$ which cannot be   represented by a transport
map, since by construction $\varphi(\bar{x})\neq \bar{x}$.

If we prove that $\pi$ is $p$-optimal we then get a contradiction  with  $\GTB_{p}$. To this aim,  it suffices
to show that $\pi$ is concentrated on a  $\sfd^{p}$-cyclically monotone set. However, this
holds trivially since for all $y\in A$
\[
\sfd^{p}(y,\bar{x})=\sfd^{p}(\varphi(y), \varphi(\bar{x}))=\sfd^{p}(y,\varphi(\bar{x})).
\]
\end{proof}

% COR
\begin{cor}
Let $(M,\sfd,\m)$ be a proper m.m.\ space  with good transport
behavior $\GTB_{p}$ for some $p\in (1,\infty)$. If  a discrete closed subgroup $\G$ of isometries  acts almost freely  on $(M,\sfd,\m)$,  then 
\[
\m(\{x\in M\thinspace|\thinspace\G_{x}\ne\{1_{\G}\}\}=0.
\]
In particular, the set of regular points is open and dense, and has full $\m$-measure.
\end{cor}

\begin{remark} Let $(M,\sfd,\m)$ be a m.m.\ space satisfying the strong $\CD^{*}(K,N)$-condition. Since $M$ is locally compact, every closed discrete group $\G$ acting isometrically on $M$ has finite isotropy groups  (see \cite[Ch.~II, Theorem 1.1.]{Kobayashi}).
\end{remark}

% SECTION - FOLIATIONS
%-----------------------------------

\section{Foliations}\label{sec:foliations}

In this section we show that most results on isometric group actions can be generalized
to foliations that are compatible with the metic measure structure.

Recall that the class of spaces with sectional curvature bounded below
by $k\in\mathbb{R}$ is closed under quotients of equidistant closed
foliations, which we call \textit{metric foliations},
see for instance \cite[Section 4.6]{BGP}. However, such a result
in the context of m.m. spaces and Ricci curvature would be too much
to expect without requiring a connection between the foliation and the
measurable structure. An essential step in the construction of an
equivariant Wasserstein geometry in Theorem  \ref{thm:LiftIsom}
is to provide a family of measures representing a natural lift of
Dirac masses in the quotient. It turns out that foliations for
which we can find such a family of measures preserve curvature-dimension
conditions. 

% DEF - METRIC FOLIATION

\begin{defn}
[Metric Foliation] A partition $\mathcal{\mathcal{F}}$ of a metric
space $(M,\sfd)$ into closed subsets is called a\emph{ foliation}\textsl{\textcolor{black}{\emph{.
Elements of the foliation }}}$F,G\in\mathcal{F}$ are called \emph{leaves.}
Furthermore, we refer to a foliation $\mathcal{F}$ as a \textit{\textcolor{black}{metric
foliation}} if, for all $F,G\in\mathcal{F}$ and $x\in F$, 
\[
\sfd(F,G)=\sfd(x,G),
\]
where the first distance is the distance between subsets of $M$.
That is, if the distance from a point $x\in F$ to a leaf $G$ is
independent of the chice of point in the leaf $F$. In case that each
leaf is bounded we say that the foliation is \emph{bounded}.
\end{defn}

% REMARK

\begin{remark}
One may verify that any Riemannian foliation induces a metric
folation in the sense above (see \cite{Walschap1992}).
\end{remark}

Given a foliation $\mathcal{F}$ on a metric space $M$, its quotient $M^{*}=M/\sim$ is the set of equivalence classes under the equivalence relation 
\[
x\sim y \text{ if and only if } \mathcal{F}_{x}=\mathcal{F}_{y},
\]
 where $\mathcal{F}_{x}$ denotes the leaf containing $x$. That is, $M^*$ is the leaf space of the foliation.
Analogously to the case of group actions, we denote the projection
onto the quotient by $\quotient:M\to M^{*}$ and elements of $M^*$ with $\quotient(x)=x^{*}\in M^{*}$.
Note that for every $x^{*}\in M^{*}$ there is a canonical association
of a leaf $\mathcal{F}_{x^{*}}\in\mathcal{F}$, namely, the unique
leaf such that $\quotient(\mathcal{F}_{x^{*}})=x^{*}$;  we can then
write the foliation as $\mathcal{F}=\{\mathcal{F}_{x^{*}}\}_{x^{*}\in M^{*}}$.
If $\mathcal{F}$ is a metric foliation we define a quotient distance
$\sfd^{*}$ on $M^{*}$ as 
\[
\sfd^{*}(x^{*},y^{*}):=\inf_{x'\in \mathcal{F}_{x^{*}}}\sfd(x', \mathcal{F}_{y^{*}})=\sfd(\mathcal{F}_{x^{*}}, \mathcal{F}_{y^{*}}),
\]
 for $x^{*},y^{*}\in M^{*}$. 

Another notion of use to us is that of a \textit{submetry}. 
\begin{defn}
[Submetry]\label{def:submetry} A map $f:M\to N$ between metric spaces is called \emph{submetry}
if, for all $x\in M$ and $r>0$,
\[
f(B_{r}(x))=B_{r}(f(x)).
\]
\end{defn}

% LEM - FOLIATION = SUBMETRY

The next lemma shows that the concepts of submetry and metric foliation
are equivalent.
\begin{lem}
There is a one-to-one correspondence between metric foliations and
submetries up to an isometry. Namely, the projection $\quotient:M\to M^{*}$
of a metric foliation is a submetry and, given a submetry $f:M\to N$,
the foliation given by $\{f^{-1}(y)\}_{y\in N}$ is a metric foliation
for which there is an isometry $i_{f}:N\to M^{*}$ with 
\[
i_{f}\circ f=\quotient.
\]
\end{lem}

% PROOF

\begin{proof}
The fact that $\quotient:M\to M^{*}$ is a submetry follows directly
from the definitions. Now consider a submetry $f:M\to N$. The continuity
of $f$ guarantees that $\mathcal{F}_{f}=\{f^{-1}(y)\}_{y\in N}$
is a foliation so we just have to check the equidistance property. 
This follows from the next observation. Let $F,G\in\mathcal{F}$ and
suppose, for the sake of contradiction, that there exists $x\in F$ such that 
$\sfd(x,G) - 2\varepsilon \geq r:=\sfd(F,G)$, for some $\varepsilon>0$.
Then  there exist $x'\in F, y' \in G$ with $\sfd(x',y')<r+\varepsilon$ and the submetry assumption gives that 
\[
f(y')\in B_{r+\varepsilon}(f(x'))= B_{r+\varepsilon}(f(x))=f(B_{r+\varepsilon}(x)).
\]
Therefore, there exists $y \in G \cap B_{r+\varepsilon}(x)$, contradicting that $\sfd(x,G)  \geq r +2\varepsilon$.
Next, by noting that $\quotient:M\to M^{*}$ is by construction
independent of the representative $x\in F\in\mathcal{F}$, we see that the
function $i_{f}:N\to M^{*}$ given by $i_{f}:=\quotient\circ f^{-1}$
is a well defined isometry by using the definition of the quotient
metric.

Finally, suppose that there exists another submetry, $g:M\to\tilde{N},$
which induces the same foliation of $f,$ that is $\mathcal{F}=\mathcal{G}=\{g^{-1}(z)\}_{z\in\tilde{N}}$.
Then we have that $i_{f}^{-1}\circ i_{g}:\tilde{N}\to N$ is an isometry
and $g=i_{g}^{-1}\circ i_{f}\circ f$. Thus, up to isometries, $f$
is unique.
\end{proof}

To be able to obtain a foliated Wasserstein geometry in the class
of metric measure spaces we require a consistent interaction of the
foliation and the measure. We write again $\m^{*}=\quotient_{\sharp}\m$
for the push-forward of $\m$ under $\quotient.$
By applying the Disintegration Theorem to the measure $\m$  with respect to $\quotient:M\to M^{*}$, we get that there exists a measurable assignment $\mm_{(\cdot)}:M^{*}\to\mathcal{P}(M)$ such that 
\begin{equation}\label{eq:DismmFol}
\mm=\int_{M^{*}} \mm_{x^{*}} \, d\mm^{*}.
\end{equation}
Note that  $\mm_{x^{*}}$ is concentrated on the leaf $\mathcal{F}_{x^{*}}=\quotient^{-1}(x^{*})$. 

% DEF - Metric-measure foliation.

\begin{defn}\label{def:mmfol}
[Bounded Metric Measure Foliation] A bounded metric foliation $\mathcal{F}$
of a m.m.\ space $(M,\sfd,\m)$
is called a \emph{ bounded  metric measure foliation} if
\begin{equation}\label{eq:defMMF} 
W_{2}(\mm_{x*},\mm_{y*})=\sfd({\mathcal F}_{x^{*}}, {\mathcal F}_{y^{*}})=\sfd^{*}(x^{*},y^{*}), \quad \text{for $\mm^{*}$-a.e. $x^{*}\in M^{*}$}.
\end{equation}
\end{defn}
Examples of bounded metric measure foliations are given by the  foliations induced by isomorphic actions of  compact groups on m.m. spaces. Many other examples arise from  Riemannian submersions (with bounded leaves) of Riemannian manifolds equipped with the natural volume measure, as in the case of submetries (with bounded leaves) of Alexandrov spaces of curvature bounded below with the Hausdorff measure.
% REMARK
 
\begin{remark}
In the definition above, it is possible as well to consider a general
$p$-Wasserstein space with $p\in(1,\infty)$. Moreover, note that
since we have considered bounded metric foliations, the push-forward
measure $\m^{*}$ is $\sigma$-finite. It is also possible to consider
more general leaves, however, the measure $\m^{*}$ might not be $\sigma$-finite
or unique. For this one may replace $\m_{x^*}$ by a family of measures $\nu_{x^*}$ supported on the leaves that 
satisfy Equation \ref{eq:defMMF} and whose naturally defined lifts preserve the entropy up to a fixed constant.
 \end{remark}

Analogously to the group action sections, we define
a lifting function $\Lambda:\mathcal{P}(M^{*})\to\mathcal{P}(M)$
as follows :
\begin{equation}\label{eq:defLambdaFol}
\hat{\mu}:=\Lambda(\mu)=\int_{M^{*}}\mm_{x^{*}}\, d\mu(x^{*}).
\end{equation}
We refer to $\hat{\mu}$ as the lift via the metric measure foliation,
as one can verify that $\quotient_{\sharp}\hat{\mu}=\mu$. Consider the subset $\mathcal{OD}$ (defined in \eqref{eq:defOD})  of pairs of points in $M\times M$ that achieve the distance between leaves. For every $x,y \in \mathcal{OD}$ let $\pi_{x,y}\in {\mathcal P}(M\times M)$ be a $p$-optimal coupling for $(\mm_{x^{*}}, \mm_{y^{*}})$. 
Note that, assumption \eqref{eq:defMMF} and the next inequality
\begin{equation}
W_{p}(\mm_{x^{*}}, \mm_{y^{*}})^{p}=\int_{{\mathcal F}_{x^{*}}\times {\mathcal F}_{y^{*}}} \sfd(x',y')^{p} \,d \pi_{x,y}(x',y') \geq   \sfd^{*}(x^{*}, y^{*})^{p},
\end{equation}
imply that  for all $x,y \in \mathcal{OD}$, the  $p$-optimal coupling
 $\pi_{x,y}\in {\mathcal P}(M\times M)$ for $(\mm_{x^{*}}, \mm_{y^{*}})$ is concentrated on $\mathcal{OD} \cap \big({\mathcal F}_{x^{*}}\times {\mathcal F}_{y^{*}}\big)$.
 With such a notation one can follow verbatim the proof of Theorem \ref{thm:LiftIsom} and show that the Wasserstein geometry is well-behaved under metric measure foliations.

% PROP

\begin{prop}\label{prop:LiftFol}
Let $(M^{*},\sfd^{*},\m^{*})$ be a bounded metric measure foliation of the m.m.\ space $(M,\sfd,\m)$. Then the following hold:
\vspace{.2cm}
\begin{enumerate}
\item The lift $\Lambda|_{W_{p}}:\mathcal{P}_{p}(M^{*})\to\mathcal{P}_{p}(M)$
is an isometric embedding onto its image which preserves absolutely
continuous measures. In particular, if $(M,\sfd)$ is geodesic, then
the subset $\Lambda(\mathcal{P}_{p}(M^{*}))\subset\mathcal{P}_{p}(M)$
is geodesically convex.\vspace{.2cm}
\item Given a measurable section $M^{*} \times M^{*} \ni(x^{*}, y^{*}) \mapsto  \big(\bar{x}(x^{*}, y^{*}), \bar{y} (x^{*}, y^{*}) \big)\in M \times M$ with  \linebreak $\big(\bar{x}(x^{*}, y^{*}), \bar{y} (x^{*}, y^{*}) \big)  \in \mathcal{OD}\cap \big({\mathcal F}_{x^{*}}\times {\mathcal F}_{y^{*}}\big)$,  if $\pi\in {\mathcal P}(M^{*}\times M^{*})$ is a $p$-optimal coupling for $(\mu,\nu)\in  {\mathcal P}_{p}(M^{*})\times  {\mathcal P}_{p}(M^{*})$, then the lift
\begin{equation} \label{eq:defhatpiFol}
\hat{\pi}:=\int_{M^{*}\times M^{*}} \pi_{\bar{x} (x^{*}, y^{*}), \bar{y}(x^{*}, y^{*})} \, d \pi(x^{*}, y^{*})
\end{equation}
is a $p$-optimal coupling for $(\hat{\mu}, \hat{\nu})$.
\end{enumerate}
\end{prop}

Using the last proposition we can repeat the proof of  Theorem \ref{thm:CDbyGisCD} with minor changes and show that the strong curvature-dimension conditions
$\CD(K,N)$, $\CD(K,\infty),$ and $\CD^{*}(K,N)$
are stable under quotients by bounded metric measure foliations. For the reader's convenience we include a full proof.

% PROOF - Stability of CD  under quotients by m.m. foliations

\begin{thm}
\label{thm:CDbyMMFisCD}
Let $(M,\sfd,\mm)$ be a strong $\CD(K,N)$ (resp. $\CD(K,\infty)$ or $\CD^{*}(K,N)$) m.m. space and let $(M^{*},\sfd^{*},\m^{*})$ be a bounded metric measure foliation of $(M,\sfd,\mm)$.Then $(M^{*},\sfd^{*},\m^{*})$ is a strong $\CD(K,N)$ (resp. $\CD(K,\infty)$ or $\CD^{*}(K,N)$) m.m. space as well.
\end{thm}

% PROOF

\begin{proof}
We give the proof for the strong $\CD(K,N)$-condition, the other cases are analogous.

Let $\mu_{0}=\varrho_{0} \mm^{*},\mu_{1}=\varrho_{1} \mm^{*} \in\mathcal{P}_{2}^{ac}(M^{*})$ and let  $\{\mu_{t}=\varrho_{t} \mm^{*}\}_{t \in [0,1]}$ be a $W_{2}$-geodesic between them inducing the $2$-optimal coupling $\pi$. 
For every $t \in [0,1]$,  let $\hat{\mu}_{t}:=\Lambda(\mu_{t})$ be the lift of $\mu_{t}$ given in Proposition \ref{prop:LiftFol}. We first claim that
 \begin{equation}\label{eq:mutvarrot}
 \hat{\mu}_{t}= \hat{\varrho}_{t} \mm, \quad \text{ where }  \quad \hat{\varrho}_{t}(x):=\varrho_{t}(\quotient(x)).
 \end{equation}
Indeed, by definition of the functions $\varrho_t$ and $\hat{\varrho_t}$ and the lift \eqref{eq:defLambdaFol} we have, for every measurable set $B\subset M$, that
\begin{eqnarray*}\label{eq:hatmut1}
\hat{\mu}_{t}(B)&:=& \int_{M^{*}}\int_{\quotient^{-1}(x^*)\cap B}\!\!\! d\mm_{x^{*}}(x) \, d \mu_{t}(x^{*}) = \int_{M^{*}}\int_{\quotient^{-1}(x^*)\cap B}\!\!\! \varrho_{t}(x^{*})\,d\mm_{x^{*}} (x)  d \mm^*(x^{*})\\& = & \int_{M^{*}}\int_{\quotient^{-1}(x^*)\cap B}\!\!\! \hat{\varrho}_{t}(x)\,d\mm_{x^{*}}(x)   d \mm^*(x^{*})=\int_{B} \hat{\varrho}_{t}(x)\,  d \mm(x).
\end{eqnarray*}
The last equality is simply a consequence of the disintegration of the measure $\m$.

By the strong $\CD(K,N)$-condition on $(M,\sfd,\mm)$, we have that
 \begin{equation*}
\int_{M}\hat{\rho}_{t}^{1-\frac{1}{N'}}\, d\m\ge\int_{M\times M}\left[\tau_{K,N'}^{(1-t)}(\sfd(x,y))\hat{\rho}_{0}^{-\frac{1}{N'}}(x)+\tau_{K,N'}^{(t)}(\sfd(x,y))\hat{\rho}_{1}^{-\frac{1}{N'}}(y)\right]d\hat{\pi}(x,y),
\end{equation*}
for every $t \in [0,1]$ and every $N'\geq N$, where $\hat{\pi}$ is the lift of $\pi$ defined in \eqref{eq:defhatpiFol} or, equivalently, the $2$-optimal coupling from $\hat{\mu}_{0}$ to $\hat{\mu}_{1}$ induced by the geodesic $\{\hat{\mu}_{t}\}_{t \in [0,1]}$.

Therefore, 
\begin{align*}
\int_{M^{*}}&\rho_{t}^{1-\frac{1}{N'}}(x^{*}) \, d\m^{*}(x^{*})  =  \int_{M^{*}}\int_{ {\mathcal F}_{x^{*}} }\hat{\rho}_{t}^{1-\frac{1}{N'}}(y) \, d\m_{x^{*}}(y) \, d\m^{*}(x^{*})=\int_{M}\hat{\rho}_{t}^{1-\frac{1}{N'}}(y)\, d\m(y) \\
&  \ge  \int_{M\times M}\left[\tau_{K,N'}^{(1-t)}(\sfd(x,y))\hat{\rho}_{0}^{-\frac{1}{N'}}(x)+\tau_{K,N'}^{(t)}(\sfd(x,y))\hat{\rho}_{1}^{-\frac{1}{N'}}(y)\right]d\hat{\pi}(x,y)\\
 & =  \int_{M^{*}\times M^{*}}\left[\tau_{K,N'}^{(1-t)}(\sfd^{*}(x^{*},y^{*}))\rho_{0}^{-\frac{1}{N'}}(x^{*})+\tau_{K,N'}^{(t)}(\sfd^{*}(x^{*},y^{*}))\rho_{1}^{-\frac{1}{N'}}(y^{*})\right]d\pi(x^{*},y^{*}),
\end{align*}
for every $t \in [0,1]$ and $N'\geq N$, where we have used that $\rho_{i}(x^{*})=\hat{\rho}_{i}(x)$ and
$\sfd(x,y)=\sfd^{*}(x^{*},y^{*})$ for $\hat{\pi}$-almost all $(x,y)\in M\times M$. 
\end{proof}

Furthermore, repeating  verbatim the proofs of Proposition  \ref{prop:LipfGf*G}, Proposition \ref{prop:UGALSQuot}, Corollary \ref{cor:W1qmw1qm*}, and Corollary \ref{cor:quotInfHilb}, we obtain the next results.

% PROP
\begin{prop}
\label{prop:stab-an-str} Let $(M,\sfd,\m)$ be a m.m. space and denote
by $(M^{*},\sfd^{*},\m^{*})$ its quotient induced by a bounded metric
measure foliation. Then 
\vspace{.2cm}
\begin{enumerate}
\item If $(M,\sfd,\m)$ satisfies the $\Lip$-$\lip$-condition,
then $(M^{*},\sfd^{*},\m^{*})$ does as well.\vspace{.2cm}
\end{enumerate}
Let $q \in (1,\infty)$ and assume further that on $(M,\sfd,\m)$ the upper asymptotic Lipschitz constant and the $q$-minimal relaxed slope agree $\mm$-a.e. for every locally Lipschitz function in $W^{1,q}(\m)$. Then the following holds,
\vspace{.2cm}
\begin{enumerate}\setcounter{enumi}{1}
\item The upper asymptotic Lipschitz constant and the
$q$-minimal relaxed slope agree $\mm^{*}$-a.e. for every local Lipschitz function
in $W^{1,q}(\m^{*})$ for $(M^{*},\sfd^{*},\m^{*})$.\vspace{.2cm}
\item The space $W^{1,q}(\m^{*})\cap\LIP(M^{*},\sfd^{*})$ and the subspace of 
functions in $W^{1,q}(\m)\cap\LIP(M,\sfd)$ constant on each leaf are isometric. 
\end{enumerate}
Lastly if, in addition, Lipschitz functions are dense in both $W^{1,q}(\m)$
and $W^{1,q}(\m^{*})$ it follows that\vspace{.2cm}
\begin{enumerate}\setcounter{enumi}{3}
\item The natural lift of functions defined on $M^{*}$ induces
an isometric embedding of $W^{1,q}(\m^{*})$ into $W^{1,q}(\m)$ whose
image is the set of functions in $W^{1,q}(\m)$ which are $\mm_{x*}$-essentially constant on $\mm^{*}$-a.e. leaf $x^{*}$. In particular, all functions $f\in D(\Ch_{q}^{M^{*}})$ satisfy that
\[
\Ch_{q}^{M}(\hat{f})=\Ch_{q}^{M^{*}}(f).
\]

\item 
$(M^{*},\sfd^{*},\m^{*})$ is infinitesimally Hilbertian granted that $(M,\sfd,\m)$ is so as well.
\end{enumerate}
\end{prop}
 Combining  Theorem \ref{thm:CDbyMMFisCD} and  (4) of Proposition \ref{prop:stab-an-str} with Remark \ref{rem:LocLipRelGradCD}, we obtain the stability of the $\mathsf{RCD}^{*}(K,N)$-condition
under quotients induced by bounded  metric measure foliations.

% COROLLARY

\begin{cor} \label{cor:CDbyMMFisCD}
Let $(M,\sfd,\m)$ be a m.m. space and denote
by $(M^{*},\sfd^{*},\m^{*})$ its quotient induced by a bounded metric
measure foliation.
\vspace{.2cm}
\begin{itemize}
\item  If $(M,\sfd,\m)$ is an  $\RCD^{*}(K,N)$-space  for some $K\in\mathbb{R}$ and $N\in[1,\infty)$, then  $(M^{*},\sfd^{*},\m^{*})$ is also an $\RCD^{*}(K,N)$-space.\vspace{.2cm}
\item   If $(X,\sfd,\mm)$ is an  $\RCD(K,\infty)$-space such that   the upper asymptotic Lipschitz constant and the $2$-minimal
relaxed slope agree $\mm$-a.e. for every locally Lipschitz function in $W^{1,2}(\m)$, and that Lipschitz functions are dense  in both $W^{1,2}(\m)$ and   $W^{1,2}(\m^{*})$, then also $(M^{*},\sfd^{*},\m^{*})$ is an $\RCD(K,\infty)$-space.
\end{itemize}
\end{cor}

We conclude this section with Theorem~\ref{ref:weakONeillRicciFormula} below, an application to Riemannian submersions of weighted Riemannian manifolds. 
The result  can be extracted also from  Lott's article on Ricci curvature of 
Riemannian submersions \cite{Lott}. Here we present an independent proof based on the theory developed in this work. Note that the result also holds for weighted Finsler manifolds.

Let $f:(M,g)\to(M^{*},g^{*})$ be a proper  $C^{2}$-Riemannian submersion and let $x^{*}$ and $y^{*}$ be two regular values of $f$ such
that $x^{*}$ and $y^{*}$ are not cut points of each other.
Then there exists a unique geodesic $\gamma:[0,1]\to M^{*}$ with $\gamma_{0}=x^{*}$ and $\gamma_{1}=y^{*}$. For every $x\in f^{-1}(x^{*})$, let $\hat{\gamma}^{x}:[0,1]\to M$ be the horizontal lift of $\gamma$ with $\hat{\gamma}^{x}_{0}=x$ and set  $y_{x}:=\hat{\gamma}^{x}_{1}\in f^{-1}(y^{*})$.
Furthermore, the assignment $x\mapsto y_{x}$ is smooth. Denote such a map
by $\rho_{x^{*},y^{*}}:f^{-1}(x^{*})\to f^{-1}(y^{*})$ and call it \emph{ the fiber transport from $x^{*}$ to $y^{*}$}.
In accordance with the notation used above, we denote with $\sfd^{*}$ (resp. $\mm^{*}$) the Riemannian distance on $(M^{*}, g^{*})$ (resp. the measure $f_{\sharp} \mm$).
We also write $\mm=\int_{M^{*}} \mm_{x^{*}} \, d \mm^{*}$ for the disintegration of $\mm$ with respect to $f$.

Assume $(\rho_{x^{*},y^{*}})_{\sharp}\m_{x^{*}}=\m_{y^{*}}$. Note that by properness of the submersion we guarantee a (finite) disintegration $\m_{x^{*}}(M)<\infty$ 
for $\mm^{*}$-a.e. $x^{*}$.  
It is not difficult to check that
 $W_{2}(\m_{x^{*}},\m_{y^{*}})=\sfd^{*}(x^{*},y^{*})$.
In particular, the measures $\{\m_{x^{*}}\}_{x^{*}\in M^{*}}$
and $f$ form a bounded  metric-measure foliation. In \cite{Lott} Lott calls this
condition \emph{the fiber transport preserves the fiberwise measure}. 

\begin{thm}
\label{ref:weakONeillRicciFormula}
Let $f:(M^{n},g,\Phi\cdot\vol_{g})\to(L^{l},g^{*},\Psi\cdot\vol_{g^{*}})$
be a Riemannian submersion with $f_{\sharp} (\Phi\cdot\vol_{g})= \Psi\cdot\vol_{g^{*}}$, whose natural foliation is a bounded metric-measure
foliation. Then 
\[
\Ric_{g^{*},\Psi,N}(v^{*},v^{*})\ge\Ric_{g,\Phi,N}(v,v)
\]
for all horizontal vectors $v\in TM$ such that $g(v,v)=g^{*}(v^{*},v^{*})$
and $f_{*}v=v^{*}$.
\end{thm}

% PROOF

\begin{proof}
The proof of the equivalence $\mathsf{CD}(K,N) \Leftrightarrow \Ric_{g^{*},\Psi,N}\ge K$
in \cite[Proof of Theorem 1.7]{sturm:II} and \cite[Proof of Theorem 1.2]{Ohta2009} proves 
actually the existence of a Wasserstein geodesic with specific direction where the
interpolation inequality holds. We repeat the construction and refer
to \cite[p. 240]{Ohta2009} for details on the calculation. 

If $\Ric_{g,\Phi,N}(v,v)>K$, then for all $w$ in a sufficiently small neighborhood $\mathcal{U}$
of $v$ in $TM$ it holds $\Ric_{g,\Phi,N}(w,w)>K$. 
The proof of the interpolations inequality shows that if the transport is perfomed in the direction 
of a vector $w$ with $\Ric_{g,\Phi,N}(w,w)>K$, then the $\mathsf{CD}(K,N)$-convexity
inequality holds along this interpolation.   

Let $\m=\Phi\cdot\vol_{g}$ and $\m^{*}=\Psi\cdot\vol_{g^{*}}=e^{-\mathcal{V}}\cdot\vol_{g^{*}}$. 
Choose a geodesic $\eta:(-\delta,\delta)\to L$ with $\dot{\eta}_{0}=v^{*}\in TL$,  
set $a=d\mathcal{V}(v^*)/(N-l)$ and for $0<\epsilon\ll r\ll\delta$ consider
the balls 
\begin{align*}
A_{0} & =B_{\epsilon(1-r\,a)}^{g^{*}}(\eta_{r}),\\
A_{1} & =B_{\epsilon(1+r\,a)}^{g^{*}}(\eta_{-r}).
\end{align*}
Let $A_{\frac{1}{2}}$ be the set of $\frac{1}{2}$-midpoints defined
as 
\[
A_{\frac{1}{2}}=\{\gamma_{\frac{1}{2}}\,|\,\text{\ensuremath{\gamma} is a geodesic connecting \ensuremath{x\in A_{0}} and \ensuremath{y\in A_{1}}}\}.
\]
Now we can lift the geodesic $\eta$ to a geodesic $\xi = \hat \eta$ in $M$ such that
$\dot\xi_0 = v$ for a horizontal vector with $f_* v = v^*$. Then there is a neighborhood 
$\hat A_0\subset \quotient^{-1} (A_0)$ such that $\quotient(\hat A_0) = A_0$. 
Analogously, we obtain lifts $\hat A_{1}$ (resp. $\hat A_{\frac{1}{2}}$) of $A_{1}$ (resp. $A_{\frac{1}{2}}$). 
Note that we can choose $\epsilon > 0$ and $\hat A_0$ such that each geodesic connecting $A_0$ and $A_1$ has tangent vector arbitarily close to $v$.

Because the fiber transport preserves the fiberwise measure, we have 
\begin{equation}\label{A0/hatA0}
\frac{\m^*(A_0)}{\m(\hat A_0)} = \frac{\m^*(A_{\frac{1}{2}})}{\m(\hat A_{\frac{1}{2}})} = \frac{\m^*(A_1)}{\m(\hat A_1)}.
\end{equation}

Thus, if $\Ric_{g,\Phi,N}(v,v)>K$, we can choose $\epsilon>0$ and  $\hat A_0$ so that  the $(K,N)$-Brunn\textendash Minkowski inequality
 holds for the lifted sets $\hat A_0$ and $\hat A_1$. Since, by \eqref{A0/hatA0}, the volume ratios agree, we get that the $(K,N)$-Brunn\textendash Minkowski inequality also holds for $A_0$ and $A_1$, yielding
\[
\m^{*}(A_{\frac{1}{2}})^{\frac{1}{N}}\ge\frac{1}{2}\tau_{K,N}^{\frac{1}{2}}(2r+O(\epsilon))^{\frac{1}{2}}\left\{ \m^{*}(A_{0})^{\frac{1}{N}}+\m^{*}(A_{1})^{\frac{1}{N}}\right\} 
\]
if we choose $\delta>0$ sufficiently small. Then the choice of $a=d\mathcal{V}(v^*)/(N-l)$ gives  
\[
\frac{\m^{*}(A_{\frac{1}{2}})}{c_{l}\epsilon^{l}}\ge e^{-\mathcal{V}(\eta_0)}\left\{ 1+\frac{1}{2}\left(K-\nabla^2_{g^*}\mathcal{V}(v^*,v^*)+\frac{(d\mathcal{V}(v^*))^{2}}{N-l}\right)r^{2}\right\} .
\]
Now, since we always have
\[
\frac{\m^{*}(A_{\frac{1}{2}})}{c_{l}\epsilon^l}=e^{-\mathcal{V}(\eta_0)} \left(1+\frac{\Ric_{g^{*}}(v^{*},v^{*})}{2}r^{2}\right)+O(r^{3}),
\]
we get 
\[
\Ric_{g^{*}}(v^{*},v^{*})\ge K-\nabla^2_{g^*}\mathcal{V}(v^*,v^*)+\frac{(d\mathcal{V}(v^*))^{2}}{N-l}.
\]
This shows that $\Ric_{g^{*},\Psi,N}(v^{*},v^{*})\ge K$. 

To finish the proof assume that the claim is wrong, then $\Ric_{g,\Phi,N}(v,v) > \Ric_{g^{*},\Psi,N}(v^{*},v^{*})$. In particular, there is a $K\in \mathbb{R}$ such that $\Ric_{g,\Phi,N}(v,v) > K > \Ric_{g^{*},\Psi,N}(v^{*},v^{*})$. But we have just proved that this implies $\Ric_{g^{*},\Psi,N}(v^{*},v^{*}) \ge K$, which is a contradiction.
\end{proof}

\color{black} 

\appendix

% APPENDIX: DISCRETE CURVATURE NOTIONS
%------------------------------------------------------------------

\section{Discrete curvature notions}\label{App:Discrete}

In this Appendix we show that the ideas developed in the article also apply to
discrete spaces. As the proofs are simpler in this setting, we present
a self-contained account on quotients with respect to Ollivier's curvature
condition and Ricci curvature of graphs. For simplicity, we only consider
isometric, measure-preserving group actions. It is also possible to  prove the results in the more general case of metric measure foliations.

% SUBSECTION: RICCI CURVATURE OF MARKOV CHAINS
%---------------------------------------------------------------------------------

\subsection{Ricci curvature of Markov chains}
Let $(M,\sfd)$ be a separable complete metric space.
\begin{defn}
[Markov Chains]  A measurable assignment 
\[
\mu_{(\cdot)}:M\to\mathcal{P}(M)
\]
is called a \emph{Markov chain}. 
\end{defn}
Assume a compact Lie group $\G$ acts isometrically on $(M,\sfd)$. As usual, denote by $\nu_{\G}$ the Haar measure on $\G$. Then there is a natural Markov chain defined on $(M^{*},\sfd^{*})$ given
by 
\[
x^{*}\mapsto\check{\mu}_{x^{*}}=\int\quotient_{\sharp}\mu_{gx} \, d\nu_{\G}(g).
\]

% SUBSUBSECTION: OLLIVER'S COARSE RICCI CURVATURE
%--------------------------------------------------------------------------------

\subsubsection{Ollivier's coarse Ricci curvature}
Assume that the Markov chain $\mu_{(\cdot)}$ has image in $\mathcal{P}_{1}(M)$.

% DEF - COARSE RICCI CURVATURE

\begin{defn}[Coarse Ricci curvature]The triple $(M,\sfd,\mu_{(\cdot)})$ is said
to have \emph{coarse Ricci curvature bounded below by $k\in\mathbb{R}$}
(in the sense of Ollivier \cite{Ollivier2009}) if 
\[
W_{1}(\mu_{x},\mu_{y})\le(1-k) \, \sfd(x,y).
\]
\end{defn}

% PROP - COARSE RICCI CURVATURE PASSES DOWN TO THE QUOTIENT

\begin{prop}
If $(M,\sfd,\mu_{(\cdot)})$ has coarse Ricci curvature bounded below
by $k\in\mathbb{R}$, then so does $(M^{*},\sfd^{*},\check{\mu}_{(\cdot)})$. 
\end{prop}

% PROOF

\begin{proof}
We first claim that we can assume that $\mu_{(\cdot)}$ is $\G$-invariant,
i.e. $(\tau_{g})_{\sharp}\mu_{x}=\mu_{g^{-1}x}$ for all $g\in\G$. To see this,
note that $x\mapsto\int\mu_{gx}\, d\nu_{\G}(g)$ is $\G$-invariant and
by \cite[Theorem 4.8]{Villani},
\begin{align*}
W_{1}\left(\int\mu_{gx}\, d\nu_{\G}(g),\int\mu_{gy}\, d\nu_{\G}(g)\right) & \le\int W_{1}(\mu_{gx},\mu_{gy})\, d\nu_{\G}(g)\\
& \le(1-k)\int \sfd(gx,gy)\, d\nu_{\G}(g) \\
& =(1-k)\, \sfd(x,y).
\end{align*}

In the following, assume $\mu_{(\cdot)}$ is $\G$-invariant. In this case,
$\check{\mu}_{x^{*}}=\quotient_{\sharp}\mu_{x}$. Since $\quotient$ is
$1$-Lipschitz we obtain 
\[
W_{1}(\check{\mu}_{x^{*}},\check{\mu}_{y^{*}})\le W_{1}(\mu_{x},\mu_{y})\le(1-k)\, \sfd(x,y).
\]
In particular,  
\begin{align*}
W_{1}(\check{\mu}_{x^{*}},\check{\mu}_{y^{*}})  \le(1-k)\inf_ {\substack{x' \in \quotient^{-1}(x^{*})\\ y' \in \quotient^{-1} (y^{*})}} \sfd(x',y') =(1-k) \, \sfd^{*}(x^{*},y^{*}).
\end{align*}
\end{proof}

% PROOF

\begin{remark}
Instead of a constant lower bound $k\in\mathbb{R}$ a similar result holds
for  a varying lower bound  $k$, i.e. 
\[
W_{1}(\mu_{x},\mu_{y})\le(1-k(x,y)) \, \sfd(x,y).
\]
In that case, a natural choice for $k^{*}$ is given by 
\[
k^{*}(x^{*},y^{*})=\sup_{\sfd(x,y)= \sfd(x^{*},y^{*})}\int_{\G}	k(gx,gy) \, d\nu_{\G}(g).
\]
If $\mu_{(\cdot)}$ is $\G$-invariant, then one can choose $k$ to
be $\G$-invariant, i.e. $k(gx,gy)=k(x,y)$ so that $k^{*}(x^{*},y^{*})=k(x,y)$.
\end{remark}

\subsubsection{Maas-Mielke Ricci curvature on finite spaces}

In this section we show that the Ricci curvature proposed independently
by Maas \cite{Maas2011} and Mielke \cite{Mielke2013} (see also \cite{EM2016}) on discrete spaces
is preserved under taking quotients. 

In the following, let $M$ be a finite set and $\mu_{(\cdot)}$ be
a Markov chain. In this case the Markov chain $\mu_{(\cdot)}$ can
be described as a function $K:M\times M\to\mathbb{R}$, called \emph{Markov
kernel}, as follows. Consider $A\subset M$ and let
\[
\mu_{(x)}(A)=\sum_{y\in A}K(y,x).
\]
Recall that, for any measure $\mu$ on $M$, there is a function $\rho:M\to[0,\mu(M)]$
such that 
\[
\mu(A)=\sum_{x\in A}\rho(x).
\]
We say that $\rho$ is a \emph{probability function} if $\mu$ is a probabilty
measure. We denote the set of probability functions by $\mathscr{P}(M)$. We also assume that $K$ is \emph{irreducible}, i.e. for any $x,y\in M$
there are $x_{0}=x,\ldots,x_{n}=y$ such that $K(x_{i-1},x_{i})>0$.
It is known that any irreducible Markov chain on a finite set has
a \emph{unique stationary probability measure $\mu$ }such that 
\[
\mu=\int\mu_{(x)}d\mu(x).
\]
If represented as a function $\pi:M\to[0,1]$, this means that
\[
\pi(y)=\sum_{y\in M}K(x,y)\pi(x).
\]

In order to show that the metric on the space of probability measures
on $M$ constructed below is symmetric we also need to assume the
Markov chain is symmetric, i.e. 
\[
K(x,y)\pi(x)=K(y,x)\pi(y), \quad \text { for all }  x,y \in M.
\]

Using  the Markov kernel, Maas \cite{Maas2011} and Mielke \cite{Mielke2013} defined
a non-local metric on the space of probability measures which imitates
the Benamou-Brenier characterization of the Wasserstein metric in
$\mathbb{R}^{n}$ \cite{BB2000}. Instead of the original definition, we present
an equivalent approach due to   Erbar and Maas \cite{EM2016}, 
which behaves better
under convex combinations and which will allow to us to prove that the 
space of  probability functions $\mathscr{P}(M^{*})$ in the quotient  (endowed with a suitable distance that we will define below) is isometric to  the subset  $\mathscr{P}^{\G}(M)$ of $\G$-invariant probability functions on $M$.

Define the  functions $\theta:M\times M\to\mathbb{R}$ and
$\alpha:\mathbb{R}\times[0,\infty)^{2}\to[0,\infty]$ by
\[
\theta(s,t)=\int_{0}^{1}s^{p}t^{1-p}dp
\]
and 
\[
\alpha(x,s,t)=\begin{cases}
0 & \theta(s,t)=0,x=0;\\
\frac{x^{2}}{\theta(s,t)} & \theta(s,t)\ne0;\\
\infty & \theta(s,t)=0,x\ne0.
\end{cases}
\]

Given $\rho\in\mathscr{P}(M)$ and a function $V\colon M\times M \to \mathbb{R}$, define 
\[
\mathscr{A}'(\rho,V)=\frac{1}{2}\sum_{x,y\in M}\alpha(V(x,y),\rho(x),\rho(y))K(x,y)\pi(x).
\]

Given two probability functions $\bar{\rho}_{0}$ and $\bar{\rho}_{1}$,
define a set $\mathscr{CE}(\bar{\rho}_{0},\bar{\rho}_{1})$ 
composed of pairs $(\rho,V)$, where 
\begin{align*}
\rho:[0,1] & \to\mathbb{R}^{M} &  & \text{is continuous with values in }\mathscr{P}(M),  \text{ with } \rho_{0} =\bar{\rho}_{0}, \,
\rho_{1}  =\bar{\rho}_{1},\\
V:[0,1] & \to\mathbb{R}^{M\times M} &  & \text{is locally integrable}\\
\end{align*}
and for all $x\in M$ the following holds in the sense of distributions
\[
\dot{\rho}_{t}(x)+\frac{1}{2}\sum_{y\in M}(V_{t}(x,y)-V_{t}(y,x))K(x,y)=0.
\]
Now we define a geodesic metric as follows 
\[
\mathscr{W}(\bar{\rho}_{0},\bar{\rho}_{1})^{2}=\inf\left\{\int_{0}^{1}\mathscr{A}'(\rho_{t},V_{t})dt\,|\,(\rho,V)\in\mathscr{CE}(\bar{\rho}_{0}\bar{\rho}_{1}) \right\}.
\]

Assume $\G$ is a compact group acting on $M$ which preserves the
Markov chain. Let $\mathscr{P}^{\G}(M)$ be the set of $\G$-invariant
probability functions. Note that $(\mathscr{P}^{\G}(M),\mathscr{W})$
is a subspace of $(\mathscr{P}(M),\mathscr{W})$.  Given $\bar{\rho}_{0},\bar{\rho}_{1}\in\mathscr{P}(M)$ and $(\rho,V)\in\mathscr{CE}(\bar{\rho}_{0},\bar{\rho}_{1})$,   define  the $\G$-averages $(\rho^{\G},V^{\G})$   by
\begin{align*}
\rho_{t}^{\G}(x) & =\int_{\G}\rho_{t}(gx)\, d\nu_{\G}(g)\\
V_{t}^{\G}(x,y) & =\int_{\G}\int_{\G}V_{t}(gx,g'y) \, d\nu_{\G}(g) \, d\nu_{\G}(g').
\end{align*}
It is easy to see that   $(\rho^{\G},V^{\G}) \in \mathscr{CE}  (\bar{\rho}_{0}^{\G},\bar{\rho}_{1}^{\G})$. 
\\One can verify that $\mathscr{A}'$ is convex w.r.t. linear interpolations
(see \cite[Corollary 2.8]{EM2016}), thus 
\[
\mathscr{A}'(\rho^{\G},V^{\G})\le\mathscr{A}'(\rho,V).
\]
Hence, letting $\mathscr{CE}^{\G}(\bar{\rho}_{0},\bar{\rho}_{1})$
be the pairs which agree with their $\G$-averages, we get  
\[
\mathscr{W}(\bar{\rho}_{0},\bar{\rho}_{1})^{2}=\inf\left\{\int_{0}^{1}\mathscr{A}'(\rho_{t},V_{t})dt\,|\,(\rho,V)\in\mathscr{CE}^{\G}(\bar{\rho}_{0}\bar{\rho}_{1}) \right\}, \quad  \text{ for all }  \bar{\rho}_{0},\bar{\rho}_{1} \in \mathscr{P}^{\G}(M).
\]

Let $K^{*}$ be the quotient Markov chain, i.e. 
\[
K^{*}(x^{*},y^{*})=\sum_{\substack{x\in\quotient^{-1}(x^{*})\\y\in\quotient^{-1}(y^{*})}}K(x,y).
\]
As above, it is possible to define a metric $\mathscr{W}^{*}$ via the
$K^{*}$ on $\mathscr{P}(M^{*})$ and show that the natural lift
$\Lambda:\mathscr{P}(M^{*})\to\mathscr{P}(M)$ defined by 
$$\Lambda(\bar{\rho})(x)=\bar{\rho}^{\wedge}(x)=\bar{\rho}(x^{*})$$
gives a bijection between $\mathscr{P}(M^{*})$ and $\mathscr{P}^{\G}(M)$.
Furthermore, given $\bar{\rho}_{0},\bar{\rho}_{0}\in\mathscr{P}(M^{*})$
there is also a one-to-one correspondence between $\mathscr{CE}(\bar{\rho}_{0},\bar{\rho}_{1})$
and $\mathscr{CE}^{\G}(\bar{\rho}_{0},\bar{\rho}_{1})$ given by
lifting each $(\rho,V)$, i.e. by defining
\begin{align*}
\hat{\rho}_{t}(x) & :=  \rho_{t}(x^{*}),\\ \hat{V}_{t}(x,y) & :=  V_{t}(x^{*},y^{*}).
\end{align*}
Directly from the definitions, it is not difficult to prove the next result.

% PROPOSITION

\begin{prop} The natural lift  $\Lambda:\mathscr{P}(M^{*})\to\mathscr{P}(M)$  induces an isometry between $(\mathscr{P}(M^{*}),\mathscr{W}^{*})$
and $(\mathscr{P}^{\G}(M),\mathscr{W})$. In particular, $\mathscr{P}^{\G}(M)$ is a weakly convex subspace of $\mathscr{P}(M)$, i.e. between any
two points in $\mathscr{P}^{\G}(M)$ there is a geodesic lying in $\mathscr{P}^{\G}(M)$ connecting those points.
\end{prop}

Given a probability function $\bar{\rho}\in\mathscr{P}(M)$, we define the \emph{entropy} of the Markov chain as follows 
\[
\mathscr{H}(\bar{\rho})=\sum_{x\in M}\bar{\rho}(x)\log\bar{\rho}(x)\pi(x).
\]

% DEF

\begin{defn}[Maas-Mielke Ricci curvature]  The Markov chain given by the kernel
$K$ is said to have \emph{Ricci curvature bounded below by $K\in\mathbb{R}$} (in the sense of Maas \cite{Maas2011} and Mielke \cite{Mielke2013})
if, for any constant speed geodesic $\{\rho_{t}\}_{t\in[0,1]}$ in
$(\mathscr{P}(M),\mathscr{W})$,
\[
\mathscr{H}(\rho_{t})\le(1-t)\mathscr{H}(\rho_{0})+t\mathscr{H}(\rho_{1})-\frac{K}{2}t(1-t)\mathscr{W}^{2}(\rho_{0},\rho_{1}).
\]
\end{defn}

On the quotient space there is also a (Markov chain) invariant probability
function called $\pi^{*}$ and a corresponding entropy $\mathscr{H}^{*}$.  Since by assumption $\G$ preserves the Markov chain, the group also preserves the (Markov
chain) invariant probability function $\pi$. In particular, $\pi$
agrees with the natural lift of the invariant probability function
$\pi^{*}$ on $M$. Therefore, if $\bar{\rho}\in\mathscr{P}(M^*)$, we obtain  
\begin{align*}
\mathscr{H}(\bar{\rho}^{\wedge}) & =\sum_{x\in M}\bar{\rho}^{\wedge}(x)\log\bar{\rho}^{\wedge}(x)\pi(x) \\ 
						  &=\sum_{x^{*}\in M^{*}}\sum_{x\in\quotient^{-1}(x^{*})}\bar{\rho}(x)\log\bar{\rho}(x^{*})\pi(x)\\
						 & =\sum_{x\in\quotient^{-1}(x^{*})}\bar{\rho}(x)\log\bar{\rho}(x^{*})\pi^{*}(x^{*})=\mathscr{H}^{*}(\bar{\rho}).
\end{align*}
Therefore we can entirely work in $\mathscr{P}^{\G}(M)$ instead of $\mathscr{P}(M^{*})$. 

Since $\mathscr{P}^{\G}(M)$ is a geodesic subspace, we immediately
get the following result.
\begin{prop}
If the Markov chain $K$ has Ricci curvature bounded below by $K\in\mathbb{R}$
and $\G$ preserves the Markov chain, then the quotient Markov chain
has Ricci curvature bounded below by $K\in\mathbb{R}$ as well.
\end{prop}

\subsection{Bakry-\'Emery and curvature-dimension conditions on graphs}

Based on the Bakry-\'Emery condition for diffusion operators \cite{BE1985}, Lin-Yau \cite{LY2010} 
introduced the Bakry-\'Emery condition to the setting of graph Laplacians. A stronger 
variant yielding the Li-Yau inequality appeared recently in \cite{BauerETAL2015}.

An \emph{undirected graph} $G=(V,E)$ is a countable possibly infinite
set of \emph{vertices} $V$ and a set of \emph{edges} $E$ such that
each element $e\in E$ is  a subset of $V$ with one or two elements.
For two vertices $x,y\in V$ we write $x\sim y$ whenever $\{x,y\}\in E$.

A \emph{weight} on $G$ is a symmetric function $\omega:V\times V\to[0,\infty)$
such that $\omega(x,y)>0$ if and only if $x\sim y$ and $x\ne y$. The
graph is said to be \emph{locally finite} if 
\[
d(x)=\sum_{x\sim y}\omega(x,y)<\infty.
\]
Given any undirected graph there is a natural weight function $\omega(x,y)=1$, whenever $x\sim y$.
Such graphs are usually called \emph{unweighted graphs}.

A final ingredient is a \emph{positive, locally finite measure} given
as a function $\m:V\to[0,\infty)$. A natural choice is given by $\m(x)=d(x)$.

\subsubsection{Function spaces on graphs and curvature-dimension conditions}

Let $C_{0}(V)$ be the set of real-valued functions on $V$. For $p\in[1,\infty)$,  the
$\ell^{p}$-space is given by 
\[
\ell^{p}(V,\m)= \Big\{f\in C_{0}(V)\,|\,\|f\|_{p}^{p}:=\sum_{x\in V}|f(x)|^{p}\m(x)<\infty \Big\}
\]
and 
\[
\ell^{\infty}(V,\m)=\big\{f\in C_{0}(V)\,|\,\|f\|_{\infty}:=\sup_{x\in V}|f(x)|<\infty \big\}.
\]
The $\ell^{2}$-norm is associated to an inner product, i.e. for $f,g\in\ell^{2}(V,\m)$
the inner product 
\[
\langle f,g\rangle=\sum f(x)g(x)\m(x)
\]
induces the norm $\|\cdot\|_{2}$.  The $p$-power of the modulus of the discrete gradient is defined as
\[
|\nabla f|^{p}(x)=\frac{1}{p\m(x)}\sum_{y\sim x}\omega(x,y)|f(y)-f(x)|^{p},
\]
see \cite{LY2010, Mugnolo2013}.

It is now possible to define the discrete Sobolev spaces $w^{1,p}(V,\m)$
as 
\[
w^{1,p}(V,\m)=\{f\in\ell^{p}(V,\m)\,|\,\||\nabla f|^{p}\|_{w^{1,p}}<\infty\}
\]
with the  norm given by 
\[
\|f\|_{w^{1,p}}^{p}=\|f\|_{p}^{p}+\||\nabla f|^{p}\|_{1}.
\]

If $p=2$ this is induced by a scalar product denoted by 
\[
2\Gamma(f,g)(x)=\frac{1}{\m(x)}\sum_{y\sim x}\omega(x,y)(f(y)-f(x))(g(y)-g(x)).
\]
This functional is often called \emph{$\Gamma$-operator} or \emph{carr\'e-du-champ}. 
 
The \emph{discrete $\m$-Laplacian} is defined as 
\[
\Delta f(x)=\frac{1}{\m(x)}\sum_{y\sim x}\omega(x,y)(f(y)-f(x)).
\]
The definition implies that
\[
2\Gamma(f,g)=\Delta(f\cdot g)-f\cdot\Delta g-g\cdot\Delta f.
\]
Following Bakry-\'Emery \cite{BE1985,LY2010} we define the \emph{$\Gamma_{2}$-functional} (\emph{carrÈ-du-champ itÈrÈ})
as follows:
\[
2\Gamma_{2}(f,g)=\Delta(\Gamma(f,g))-\Gamma(f,\Delta g)-\Gamma(\Delta f,g).
\]

% DEF - BAKRY-EMERY CONDITION

\begin{defn}
[Bakry-\'Emery condition] The graph $G$ is said to satisfy the \emph{discrete
curvature-dimension condition} $\CD(K,N)$, $K\in\mathbb{R}$, $N\in[1,\infty]$
if, for every function $f:V\to\mathbb{R}$, the following inequality holds:
\[
\Gamma_{2}(f,f)\ge K\Gamma(f,f)+\frac{1}{N}(\Delta f)^{2}.
\]
\end{defn}
This condition already makes it possible to characterize certain graphs
and to obtain  eigenvalue bounds (see \cite{LY2010}). However, the lack of a chain
rule prevents to transfer the result from diffusion operators and continuous
times heat flows to the discrete setting. In \cite{BauerETAL2015} the authors 
introduced a variant of this condition which gives a  Li-Yau-type inequality 
for the heat
flow.
\begin{defn}
[Exponential curvature-dimension] The graph $G$ satisfies the exponential
curvature-dimension condition $\CDE(K,N)$ if 
\[
\Gamma_{2}(f,f)-\Gamma\left(f,\frac{\Gamma(f,f)}{f}\right)\ge K\Gamma(f,f)+\frac{1}{N}(\Delta f)^{2}.
\]
\end{defn}

% REMARK

\begin{remark}
It is possible to obtain also a pointwise curvature condition by assuming that $K:V\to\mathbb{R}$ and $N:V\to[1,\infty]$ are functions on $V$. 
\end{remark}

% SUBSUBSECTION: GROUP ACTIONS ON GRAPHS
%-------------------------------------------------------------------------

\subsubsection{Group actions on graphs} We say that a  group \emph{$\G$ acts isometrically on the graph $G$}, if for
all $g\in\G$, $g\cdot x\sim g\cdot y$ whenever $x\sim y$.
We say $\G$ \emph{preserves weights} if 
\[
\omega(g\cdot x,g\cdot y)=\omega(x,y).
\]
Note that if $\G$ preserves weights then  $\G$  also preserves the weighted degree, i.e.
$d(g\cdot x)=d(x)$. 

If $\m(x)=\m(g\cdot x)$, then the action is said to be \emph{measure-preserving}.
The action is said to be  locally of locally finite order if, for all $x\in V$, the orbit 
\[
\G(x)=\{g\cdot x\,|\,g\in\G\}
\]
is finite.

Assume in the following that $\G$ acts isometrically  on  a locally finite
graph $G$. We can define a \emph{quotient graph $G^{*}$}
as follows.   The vertex set $V^{*}$ is given by the set of orbits; as usual, $x^{*}$ denotes the orbit corresponding to $x$ and $\quotient$ is the quotient map. Furthermore, whenever $\{x,y\}\in E$
then by definition $\{x^{*},y^{*}\}\in E^{*}$.  The quotient graph $V^{*}$ can be endowed with the quotient  measure function $\m^{*}$ given by
\[
\m^{*}(x^{*})=\sum_{x\in\quotient^{-1}(x)}\m(x).
\]
If $\G$ acts measure-preserving, then $\m^{*}(x^{*})=\# \G(x)\cdot\m(x)$. Finally, the weight function $\omega^{*}$ of the quotient graph is
given by 
\[
\omega^{*}(x^{*},y^{*})=\m^{*}(x^{*})\sum_{x\in\quotient^{-1}(x^{*}),y\in\quotient^{-1}(y^{*})}\frac{\omega(x,y)}{\m(x)}.
\]

\begin{remark}
(1) The edge set allows one to define a metric 
which makes $(V,\sfd)$ into a discrete metric space by just defining the distance to be the number of edges of a shortest path. Then it is easy
to see that the metric of the quotient graph obtained in this way agrees
with the quotient metric 
\[
\sfd^{*}(x^{*},y^{*})=\inf_{g\in\G} \sfd(x,g\cdot y).
\]

(2) One may replace the action $\G$ by any partition $\mathcal{F}=\{\mathcal{F}_{x}\}_{x\in V}$
of $V$ into finite sets such that  whenever $x'\in\mathcal{F}_{x}$
and $x\sim y$ then there is a $y'\in\mathcal{F}_{y}$ with $x'\sim y'$.
Or, equivalently, if $\quotient:V\to V/\mathcal{F}$ is the natural
quotient map and $E^*:=\quotient\times \quotient ((V\times V) \cap E) $ then 
\[
(\{x^*\}\times V^*) \cap E^* = \quotient\times \quotient((\{x\}\times V) \cap E)
\]
for all $x\in V$.
As in the metric setting, we call such a partition a \emph{metric foliation}.
\end{remark}
Given a function $f:V^{*}\to\mathbb{R}$ there is a natural lift $\hat{f}:V\to\mathbb{R}$
defined as usual by
\[
\hat{f}(x)=f(x^{*}).
\]
From the definition we see that 
\[
\sum_{x^{*}\in V^{*}}|f(x^{*})|^{p}\m^{*}(x^{*})=\sum_{x\in V}|\hat{f}(x)|^{p}\m(x).
\]
Furthermore, the weight and measure functions are chosen to show that
the Dirichlet form on $G^{*}$ agrees with the Dirichlet form of $\G$-invariant
functions on $G$. 

% PROPOSITION

\begin{prop}
Let $f:V^{*}\to\mathbb{R}$ be a function and $p\in[1,\infty)$. Then
\[
\frac{1}{\m^{*}(x^{*})}\sum_{y^{*}\sim x^{*}}\omega^{*}(x^{*},y^{*})(f(y^{*})-f(x^{*}))^{p}=\sum_{x\in\quotient^{-1}(x^{*})}\frac{1}{\m(x)}\sum_{y\sim x}\omega(x,y)(\hat{f}(y)-\hat{f}(x))^{p}.
\]
In particular, $(\Delta^{*}f)^{\wedge}=  \Delta \hat{f}$ and 
\[
(\Gamma^{*}(f,g))^{\wedge}=\Gamma(\hat{f},\hat{g})
\]
for all functions $f,g:V^{*}\to\mathbb{R}$, where $\Delta^{*}$ and
$\Delta$ are the Laplacian on $G$ and $G^{*}$ and $\Gamma$ and
$\Gamma^{*}$ are the corresponding $\Gamma$-operators.
\end{prop}

% PROOF

\begin{proof}
Since $\hat{f}$ is constant along orbits, we see that 
\[
f(y^{*})-f(x^{*})=\hat{f}(y)-\hat{f}(x)
\]
whenever $\quotient(x)=x^{*}$ and $\quotient(y)=y^{*}$. Thus, for
$y^{*}\sim x^{*}$, we have
\begin{align*}
\frac{1}{\m^{*}(x^{*})}\omega^{*}(x^{*},y^{*})(f(y^{*})-f(x^{*}))^{p} & =\left(\sum_{\substack{x\in\quotient^{-1}(x^{*})\\y\in\quotient^{-1}(y^{*})}}\frac{\omega(x,y)}{\m(x)}\right)\cdot(f(y^{*})-f(x^{*}))^{p}\\[.3cm]
 & =\sum_{x\in\quotient^{-1}(x^{*})}\frac{1}{\m(x)}\sum_{y\in\quotient^{-1}(y^{*})}\omega(x,y)(\hat{f}(y)-\hat{f}(x))^{p}.
\end{align*}
Finally, recalling that $\omega(x,y)>0$ if and only if $x\sim y$, it follows  that, for fixed $x^*$ and $x$,
\[
\sum_{x^{*}\sim y^{*}}\sum_{y\in\quotient^{-1}(y^{*})}\omega(x,y)(\hat{f}(y)-\hat{f}(x))^{p}=\sum_{y\sim x}\omega(x,y)(\hat{f}(y)-\hat{f}(x))^{p}
\]
which gives the claim. 
\end{proof}
\begin{cor}
The natural lift $f\mapsto\hat{f}$ induces an isometry between the $\ell^{p}$-spaces
$\ell^{p}(V^{*},\m^{*})$ and the closed subspaces of $\G$-invariant
function of $\ell^{p}(V,\m)$. Similarly, it induces an isometry
between the discrete Sobolev spaces $w^{1,p}(V^{*},\m^{*})$ and
the $\G$-invariant functions in $w^{1,p}(V,\m)$. 
\end{cor}
Since the calculation of $\Gamma^{*}$ and $\Delta^{*}$ can be done
either on the quotient graph $G^{*}$ or via the lift on $G$ we obtain
the following theorem.
\begin{thm}
Assume $G$ is a locally finite graph satisfying the $\CD(K,N)$-condition
(resp. $\CDE(K,N)$-condition). If $\G$ acts isometrically and locally of  finite order on $G$, then the quotient graph satisfies the $\CD(K,N)$-condition
(resp. $\CDE(K,N)$-condition).
\end{thm}
\begin{remark}
Again as in the setting of general metric spaces, the curvature-dimension
conditions are only preserved in the weighted theory. Indeed, the
quotient of an unweighted graph  in general  can be weighted.
\end{remark}

\section{Group actions and  (super-)Ricci flow}\label{App:RF}

In this section we show that the ideas of the paper also apply
to super-Ricci flows.  First of all, recall that  a  super-solution to the Ricci flow (a \emph{super-Ricci flow} for short) is 
 a time-dependent Riemannian manifold $(M,g_{t})_{t\in[0,T)}$ satisfying 
\[
\partial_{t}g_{t}+2\mathrm{Ric}_{g_{t}}\ge0.
\]
Super-Ricci flows have many common features with spaces with non-negative Ricci curvature.
Based on a characterization of super-Ricci flow by contractivity of
the induced time-dependent heat flow, Sturm \cite{Sturm:SuperRicciI} analyzed the
time-convexity of the entropy functional  and used it to  characterize
super Ricci flows.
\\ By the uniqueness results of Hamilton \cite{Ham} and Chen-Zhu \cite{ChenZhu}, it is known that the Ricci flow for metrics of bounded curvature preserves the isometry group of the solution, i.e.    if $(M,g_{t})_{t\in[0,T)}$ is a Ricci flow of complete metrics with bounded curvature then  $\textrm{Isom}(M,g_{0})\subset \textrm{Isom}(M,g_{t})$ for all $t \in [0,T]$. Conversely, by Kotschwar's backward uniqueness of Ricci flow \cite{Kotschwar} also the converse inclusion holds;   this gives that if $(M,g_{t})_{t\in [0,T]}$ is a Ricci flow of complete metrics  with  bounded curvature, then the isometry group does not change along the flow, i.e. $\textrm{Isom}(M,g_{0})= \textrm{Isom}(M,g_{t})$ for all $t \in [0,T]$.  The goal of this section is to show that in  such a situation, the time-dependent quotient space is a (weighted $N$-dimensional) super-Ricci flow.

Let $(M, \sfd_{t},\m_{t})_{t\in[0,T)}$ be a time-dependent
family of m.m.\ spaces. We assume that each $\sfd_{t}$ induces the same
topology on $M$ and that the measures $\m_{t}$ are mutually absolutely
continuous with respect to each other. Also denote by $(\mathcal{P}_{2}^{(t)}(M),W_{t})$
the $2$-Wasserstein space and its metric induced by $\sfd_{t}$.

Let $u:(t-\epsilon,t+\epsilon)\to\mathbb{R}\cup\{\infty\}$. Define
\begin{align*}
\partial_{t}^{+}u(t) & =\limsup_{s\to t}\frac{u(t)-u(s)}{t-s},\quad \partial_{t}^{-}u(t)=\liminf_{s\to t}\frac{u(t)-u(s)}{t-s},\\
\partial_{t}^{+}u(t-) & =\limsup_{s\nearrow t}\frac{u(t)-u(s)}{t-s},\quad \partial_{t}^{-}u(t+)=\liminf_{s\searrow t}\frac{u(t)-u(s)}{t-s}.
\end{align*}

Let $\Ent:[0,T)\times M$ denote the time-dependent Shannon entropy
functional 
\[
\Ent_{t}(\mu)=\begin{cases}
\int f^{(t)}\log f^{(t)}d\m_{t} & \mu=f^{(t)}\m\\
\infty & \text{otherwise}.
\end{cases}
\]

\begin{defn}
[\protect{Super-$N$-Ricci flow \cite[Definition 0.6]{Sturm:SuperRicciI}}] The time-dependent m.m.\ space $(M,\sfd_{t},\m_{t})_{t\in[0,T)}$
is a \emph{super-$N$-Ricci flow} if the Shannon entropy $\Ent$
is dynamical $N$-convex, i.e. if for  almost all $t\in[0,T)$ and every $\mu_{0},\mu_{1}\in\mathcal{P}_{2}^{(t)}(M)$
with $\Ent_{t}(\mu_{0}),\Ent_{t}(\mu_{1})\in\mathbb{R}$, there is
a geodesic $\tau\mapsto\mu_{\tau}$ in $\mathcal{P}_{2}^{(t)}(M)$
connecting $\mu_{0}$ and $\mu_{1}$ such that $\tau\mapsto\Ent_{t}(\mu_{\tau})$
is absolutely continuous and
\[
\partial_{\tau}^{+}\Ent_{t}(\mu_{1-})-\partial_{\tau}^{-}\Ent_{t}(\mu_{0+})\ge  - \frac{1}{2} \partial_{t}^{-}W_{t-}(\mu_{0},\mu_{1})^{2}+\frac{1}{N}\left|\Ent_{t}(\mu_{0})-\Ent_{t}(\mu_{1})\right|^{2}.
\]
We say it is a \emph{strong super-$N$-Ricci flow} if the inequality
holds along all geodesics $\tau\mapsto\mu_{\tau}$.
\end{defn}

% PROPOSITION

\begin{prop}\label{Prop:QuotSRF}
Assume $(M, \sfd_{t},\m_{t})_{t\in[0,T)}$ is a time-dependent m.m.\ space
and let  $\G$ be a compact Lie group acting  effectively by isomorphisms of m.m. spaces on each
$(M, \sfd_{t},\mm_{t})$. Furthermore, assume the induced quotient metrics $\sfd_{t}^{*}$
induce all  the same topology on $M^{*}$ and the quotient measures $\m_{t}^{*}$  are all mutually absolutely continuous.
If $(M,\sfd_{t},\m_{t})_{t\in[0,T)}$
is a strong super-$N$-Ricci flow then $(M^{*}, \sfd_{t}^{*},\m_{t}^{*})$
is a  strong  super-$N$-Ricci flow as well.
\end{prop}

%  PROOF

\begin{proof}
Just note that by Theorem \ref{thm:LiftIsom}, for fixed $t\in[0,T)$
the $\G$-invariant Wasserstein space $\mathcal{P}_{2}^{(t),\G}(M)$
is isometric to $\mathcal{P}_{2}^{(t)}(M^{*})$ and the entropy with
respect to $\m_{t}^{*}$ of a measure in $\mathcal{P}_{2}(M^{*})$ is
given by the entropy with respect to $\m_{t}$ of the $\G$-invariant
lift. Thus if we let $\Ent^{*}$ denote the entropy of the quotient
spaces and $\tau\mapsto\hat{\mu}_{\tau}$ the $\G$-invariant lift
of a geodesic $t\mapsto\mu_{\tau}$ then
\begin{align*}
\partial_{\tau}^{+}\Ent_{t}^{*}(\mu_{1-})-\partial_{\tau}^{-}\Ent_{t}^{*}(\mu_{0+}) & =\partial_{\tau}^{+}\Ent_{t}(\hat{\mu}_{1-})-\partial_{\tau}^{-}\Ent_{t}(\hat{\mu}_{0+})\\
 & \geq    - \frac{1}{2} \partial_{t}^{-}W_{t-}(\hat{\mu}_{0}, \hat{\mu}_{1})^{2}+\frac{1}{N}\left|\Ent_{t}(\hat{\mu}_{0})-\Ent_{t}(\hat{\mu}_{1})\right|^{2}  \\
 & =  - \frac{1}{2} \partial_{t}^{-}  W_{t-}^{*}(\mu_{0},\mu_{1})^{2} +\frac{1}{N}\left|\Ent_{t}^{*}(\mu_{0})-\Ent_{t}^{*}(\mu_{1})\right|^{2}.
\end{align*}
\end{proof}

% REMARK

\begin{remark} \label{rem:quotSRFGnoT}
Let $\G$ be as in   Proposition \ref{Prop:QuotSRF} and assume furthermore that the action of $\G$ is independent of the  time $t\in [0,T)$. Then the assumption that all the metrics $\sfd_{t}$ induce the same topology on $M$ and all the measures $\mm_{t}$ are mutually absolutely continuous ensure that the same is true on the quotient $M^{*}$, i.e.   the induced quotient metrics $\sfd_{t}^{*}$
induce all  the same topology on $M^{*}$ and the quotient measures $\m_{t}^{*}$  are all mutually absolutely continuous. So in this case the quotient of a strong $N$-super-Ricci flow is an strong $N$-super-Ricci flow in the quotient spaces without extra assumptions on the quotient structures.
\end{remark}

Combining Remark \ref{rem:quotSRFGnoT} with Proposition \ref{Prop:QuotSRF}, together with the fact that a smooth $N$-dimensional Ricci flow is a strong super-$N$-Ricci flow, we get the next result.

% COROLLARY

\begin{cor}\label{cor:QuotRF}
Let $(M,g_{t})_{t\in[0,T)}$ be a time-dependent family of complete $N$-dimensional
Riemannian manifolds with bounded curvature solving  the Ricci flow, and let $\G< {\rm Isom}(M,g_{0})$ be  a compact Lie subgroup of the isometry group. Then the induced quotient metric-measure  spaces $(M^{*}, \sfd_{t}^{*},\m_{t}^{*})_{t\in[0,T)}$
are a strong   super-$N$-Ricci flow.  Moreover, if   the induced quotient spaces are smooth weighted Riemannian manifolds 
$(M^{*},g_{t}^{*}, \Psi^*_t\vol_{g_{t}^*})_{t\in[0,T)}$, with $\Psi^*_t\vol_{g_{t}^*}=\quotient_{\sharp}\vol_{g_{t}}$, then  
\[
\frac{1}{2}\partial_{t}g_{t}^{*}+\mathrm{Ric}_{g_{t}^{*}}-(N-n)\frac{\nabla_{g_{t}^{*}}^{2}(\Psi_{t}^{*})^{\frac{1}{N-n}}}{(\Psi_{t}^{*})^{\frac{1}{N-n}}}\ge0
\]
where $n=\dim M^{*}$.
\end{cor}
\begin{remark}
Under the assumptions of the last part of Corollary \ref{cor:QuotRF}, since the orbits $\G(x)$ at different times $t\in[0,T)$ are bi-Lipschitz
equivalent, their Haar measure $\nu_{x^{*}}$ is time independent. Thus it is
possible to obtain $f_{t}$ via the following formula
\[
\vol_{g_{t}}=\int\nu_{x^{*}}e^{-f_{t}}d\vol_{g_{t}^{*}}(x^{*}).
\]
\end{remark}

%----------------------------------------------------------------------------------
% 				BIBLIOGRAPHY
%----------------------------------------------------------------------------------

\end{document}